\documentclass{article}

\usepackage{amsmath}
\usepackage{amscd}
\usepackage{amssymb}
\usepackage{enumerate}
\usepackage{amsfonts}
\usepackage{graphicx}
\usepackage[all]{xy}
\usepackage{mathrsfs}
\usepackage[hypertex]{hyperref}

\usepackage{bbm}

\newcommand{\e}{\mathbbm{1}}

\numberwithin{equation}{section}

\newcommand{\rar}[1]{\stackrel{#1}{\longrightarrow}}
\newcommand{\xrar}[1]{\xrightarrow{#1}}

\newcommand{\limto}{{\displaystyle\lim_{\longrightarrow}}}

\newcommand{\iso}{\buildrel{\sim}\over{\longrightarrow}}

\newcommand{\into}{\hookrightarrow}
\newcommand{\onto}{\twoheadrightarrow}

\newcommand{\al}{\alpha}
\newcommand{\be}{\beta}
\newcommand{\ga}{\gamma}
\newcommand{\Ga}{\Gamma}
\newcommand{\de}{\delta}
\newcommand{\De}{\Delta}
\newcommand{\la}{\lambda}

\newcommand{\eps}{\epsilon}
\newcommand{\sg}{\sigma}

\newcommand{\te}{\theta}

\newcommand{\om}{\omega}

\newcommand{\vp}{\varphi}

\newcommand{\bC}{{\mathbb C}}
\newcommand{\bD}{{\mathbb D}}

\newcommand{\bF}{{\mathbb F}}
\newcommand{\bG}{{\mathbb G}}

\newcommand{\bK}{{\mathbb K}}
\newcommand{\bL}{{\mathbb L}}

\newcommand{\bN}{{\mathbb N}}

\newcommand{\bQ}{{\mathbb Q}}

\newcommand{\bZ}{{\mathbb Z}}

\newcommand{\cA}{{\mathcal A}}
\newcommand{\cB}{{\mathcal B}}

\newcommand{\cE}{{\mathcal E}}
\newcommand{\cF}{{\mathcal F}}

\newcommand{\cH}{{\mathcal H}}

\newcommand{\cL}{{\mathcal L}}
\newcommand{\cM}{{\mathcal M}}
\newcommand{\cN}{{\mathcal N}}
\newcommand{\cO}{{\mathcal O}}
\newcommand{\cP}{{\mathcal P}}

\newcommand{\cU}{{\mathcal U}}

\newcommand{\sA}{{\mathscr A}}

\newcommand{\sC}{{\mathscr C}}
\newcommand{\sD}{{\mathscr D}}

\newcommand{\sG}{{\mathscr G}}
\newcommand{\sH}{{\mathscr H}}
\newcommand{\sI}{{\mathscr I}}

\newcommand{\sO}{{\mathscr O}}
\newcommand{\sP}{{\mathscr P}}

\newcommand{\fc}{{\mathfrak c}}

\newcommand{\fm}{{\mathfrak m}}

\newcommand{\fp}{{\mathfrak p}}
\newcommand{\fq}{{\mathfrak q}}

\newcommand{\Gb}{\overline{G}}

\newcommand{\Mb}{\overline{M}}
\newcommand{\Nb}{\overline{N}}

\newcommand{\kbar}{{\overline{k}}}

\newcommand{\abs}[1]{\lvert #1\rvert}

\newcommand{\Ker}{\operatorname{Ker}}

\newcommand{\End}{\operatorname{End}}
\newcommand{\Hom}{\operatorname{Hom}}
\newcommand{\Ext}{\operatorname{Ext}}
\newcommand{\Aut}{\operatorname{Aut}}
\newcommand{\Spec}{\operatorname{Spec}}

\newcommand{\id}{\operatorname{id}}

\newcommand{\pr}{\mathrm{pr}}

\newcommand{\rk}{\operatorname{rk}}
\newcommand{\Ad}{\operatorname{Ad}}

\newcommand{\Ind}{\operatorname{Ind}}
\newcommand{\ind}{\operatorname{ind}}
\newcommand{\tr}{\operatorname{tr}}

\newcommand{\tens}{\otimes}

\newcommand{\st}{\,\big\vert\,}

\newcommand{\sbr}{\smallbreak}
\newcommand{\mbr}{\medbreak}
\newcommand{\bbr}{\bigbreak}

\newcommand{\Ann}{\operatorname{Ann}}

\newcommand{\matr}[4]{\left(\begin{array}{cc} \! #1 & \! #2 \! \\ \! #3 & \! #4 \! \end{array}\right)}

\newtheorem{thm}{Theorem}[section]
\newtheorem{cor}[thm]{Corollary}
\newtheorem{lem}[thm]{Lemma}

\newtheorem{prop}[thm]{Proposition}

\newtheorem{rem}[thm]{Remark}
\newtheorem{rems}[thm]{Remarks}

\newtheorem{defin}[thm]{Definition}

\newcommand{\qed}{$\square$}

\newenvironment{proof}[1]{\medbreak\noindent\emph{#1}.}{\hfill\qed\bigbreak}

\newcommand{\Fun}{\operatorname{Fun}}

\newcommand{\Fr}{\operatorname{Fr}}

\newcommand{\Gal}{\operatorname{Gal}}

\newcommand{\ql}{\overline{\bQ}_\ell}

\renewcommand{\hom}{\sH\!om}

\newcommand{\ig}{\ind_{G'}^G}
\newcommand{\Ig}{\Ind_{G'}^G}
\newcommand{\iga}{\ind_{\Ga'}^{\Ga}}

\newcommand{\com}{\operatorname{com}}

\newcommand{\cpu}{\operatorname{\mathfrak{cpu}}}
\newcommand{\cpuc}{\cpu^\circ}

\newcommand{\gap}{\mathbb{G}_{a,\,per\!f}}

\newcommand{\qzp}{\bQ_p/\bZ_p}

\newcommand{\biext}{\underline{\operatorname{Bi-ext}}}

\newcommand{\et}{\widetilde{e}}
\newcommand{\mt}{\widetilde{\mu}}

\newcommand{\Mt}{\widetilde{M}}
\newcommand{\Nt}{\widetilde{N}}
\newcommand{\Xt}{\widetilde{X}}
\newcommand{\Gt}{\widetilde{G}}

\newcommand{\Ht}{\widetilde{H}}

\begin{document}

\title{Characters of unipotent groups over finite fields}

\author{Mitya Boyarchenko}

\maketitle

\begin{abstract}
Let $G$ be a connected unipotent group over a finite field $\bF_q$. In this article we propose a definition of $\bL$-packets of complex irreducible representations of the finite group $G(\bF_q)$ and give an explicit description of $\bL$-packets in terms of the so-called ``admissible pairs'' for $G$. We then apply our results to show that if the centralizer of every geometric point of $G$ is connected, then the dimension of every complex irreducible representation of $G(\bF_q)$ is a power of $q$, confirming a conjecture of V.~Drinfeld. This paper is the first in a series of three papers exploring the relationship between representations of a group of the form $G(\bF_q)$ (where $G$ is a unipotent algebraic group over $\bF_q$), the geometry of $G$, and the theory of character sheaves.
\end{abstract}


\section{Introduction}

In 1960, G.~Higman asked\footnote{I am grateful to Jon Alperin for providing this reference.} \cite[p.~29]{higman} whether if $q$ is a prime power and
$n\in\bN$, then the dimension of every complex irreducible
representation of $UL_n(\bF_q)$ is a power of $q$. Here, $\bF_q$ is
a finite field with $q$ elements and $UL_n(\bF_q)$ denotes the group
of unipotent upper-triangular matrices of size $n$ over $\bF_q$. This question was later advertised and popularized by J.~Thompson and A.~Kirillov, among others. The answer is affirmative, and a
generalization of this fact (to the so-called ``algebra
groups'' over finite fields) was proved by I.M.~Isaacs in
\cite{isaacs}. It is natural to ask whether Isaacs's result can be
further generalized to groups of the form $G(\bF_q)$, where $G$ is a
connected unipotent group over $\bF_q$.

\mbr

If $G$ is as above, it is not always the case that the dimension of
every complex irreducible representation of $G(\bF_q)$ is a power of
$q$. This interesting phenomenon was first observed by G.~Lusztig,
who showed in \cite{lusztig} that if $U$ is a maximal unipotent
subgroup of the symplectic group $Sp_4$ over a finite field
$\bF_{2^r}$, where $r\in\bN$, then $U(\bF_{2^r})$ always has
irreducible representations of dimension $2^{r-1}$. The fake Heisenberg groups introduced in \cite{intro} (see also \S\ref{ss:fake-Heis} below) provide similar counterexamples in every characteristic $p>2$.

\mbr

In 2005, V.~Drinfeld conjectured that if a unipotent group $G$ has the property that
every geometric point of $G$ is contained in the neutral connected
component of its centralizer, then the dimension of every
irreducible representation of $G(\bF_q)$ is a power of $q$. One of
the main goals of our paper is to prove this conjecture (see Theorem
\ref{t:dim-reps-easy}). From the viewpoint of character theory
for finite groups, this is the most appealing result of the paper. However, we must point out that
the proof we present (which is the only one known to us) heavily
relies on geometric techniques, and may appear to be somewhat
indirect. In particular, it is based on the notion of an
\emph{$\bL$-packet}\footnote{In this article, we use ``$\bL$-packet'' as an abbreviation of ``Lusztig packet,'' and the letter ``$\bL$'' is used in place of ``$L$'' in order to distinguish Lusztig packets from Langlands' $L$-packets.} of irreducible representations of $G(\bF_q)$,
which we introduce for an \emph{arbitrary} connected unipotent group
$G$ over $\bF_q$, and on our second main result, which provides a
description of $\bL$-packets in more concrete terms.

\mbr

The idea of using geometry to study representations of groups of the
form $G(\bF_q)$, where $G$ is an algebraic group over $\bF_q$, is
not new. For \emph{reductive} $G$, one has the
theory of Deligne and Lusztig, which constructs many virtual
representations of $G(\bF_q)$ in the $\ell$-adic cohomology of
certain varieties $X$ over $\bF_q$ with a $G$-action, as well as Lusztig's theory of character sheaves, which expresses
the irreducible characters of $G(\bF_q)$ over $\ql$ as linear
combinations of the ``trace of Frobenius functions'' of certain
irreducible perverse $\ell$-adic sheaves on $G$.

\mbr

The case of \emph{unipotent} $G$ was also originally considered by
Lusztig. In \cite{lusztig} he predicted the existence of an
interesting theory of character sheaves for unipotent groups in
positive characteristic\footnote{In characteristic zero, a unipotent
group is ``the same'' as a finite dimensional nilpotent Lie algebra,
and in this case the theory of character sheaves is essentially
equivalent to Kirillov's orbit method \cite{kirillov}. In
particular, character sheaves themselves are simply the Fourier
transforms of the constant rank $1$ local systems on the coadjoint
orbits for the group.} and defined the character sheaves in an
\emph{ad hoc} manner for the maximal unipotent subgroup\footnote{This is
the first interesting example where the orbit method does not apply,
cf.~\cite{intro}.} $U$ of the
symplectic group $Sp_4$ over a field of characteristic $2$. Lusztig proved, moreover, that if the ground
field is a finite field $\bF_{2^r}$, then the trace functions
associated to the character sheaves on $U$ form a basis of the space
of class functions on $U(\bF_{2^r})$, and the relationship between
these functions and the irreducible characters of $U(\bF_{2^r})$ is
similar to the one that exists in the theory of character sheaves
for reductive groups over finite fields.

\mbr

Lusztig's work led Drinfeld to formulate a series of definitions and
conjectures that should form a basis of a general theory of
character sheaves for unipotent groups in positive characteristic.
We refer the reader to \cite{intro} for an overview. At present many of these conjectures are already known \cite{foundations,charsheaves}.

\mbr

The approach taken in the present article is somewhat different,
although the methods we use are closely related to those proposed
by Drinfeld and Lusztig, and they form a basis for \cite{foundations,charsheaves}. Character sheaves do not appear in this
paper, but our goal is still to study irreducible representations of
$G(\bF_q)$ by relating them to constructible $\ell$-adic complexes
on $G$, or, more precisely, to objects of the equivariant derived
category $\sD_G(G)$.

\mbr

Our work can be thought of as an attempt to geometrize two classical
and well known results of character theory for finite groups, which
we now state. If $\Ga$ is a finite group, let $\Fun(\Ga)^\Ga$ denote
the space of conjugation-invariant functions $\Ga\rar{}\bC$. It is a
commutative algebra under convolution of functions. The first result
is that there is a natural bijection between complex irreducible
characters of $\Ga$ and the minimal (in other terminology:
``indecomposable'' or ``primitive'') idempotents in $\Fun(\Ga)^\Ga$,
given by $\chi\longleftrightarrow \abs{\Ga}^{-1}\chi(1)\cdot\chi$
(see, e.g., \cite{intro}). The second
result is that if $\Ga$ is nilpotent, then every complex irreducible
representation of $\Ga$ is induced from a $1$-dimensional
representation of a subgroup of $\Ga$.

\mbr

In this paper, the word ``geometrization'' refers to replacing
finite groups with algebraic groups $G$ over finite fields $\bF_q$
and studying representations of groups of the form $G(\bF_q)$ by
relating them to the geometry of $G$. The ground field for the
representations is taken to be $\ql$ rather than $\bC$.
Geometrization also involves replacing functions on finite groups
with (complexes of) constructible $\ell$-adic sheaves on algebraic
groups and using Grothendieck's sheaves-to-functions correspondence.

 \mbr

The geometric analogue of the bijection between irreducible
characters of a finite group $\Ga$ and minimal idempotents in
$\Fun(\Ga)^{\Ga}$ is not a result at all, but rather a definition.
More precisely, for a connected unipotent group $G$ over $\bF_q$, we
propose a definition of $\bL$-packets of irreducible representations
of $G(\bF_q)$ based on the notion of a ``weak idempotent'' in the
equivariant derived category $\sD_G(G)$ (Definition
\ref{d:L-packets}). For groups $G$ of this type, the result on
representations of finite nilpotent groups mentioned above has a
geometric analogue, which is more subtle: it becomes an
\emph{explicit description of $\bL$-packets} in terms of the so-called
``admissible pairs'' for $G$ (Theorem \ref{t:descr-L-packets}). This
is the second main result of our work.

\mbr

We tried to keep the amount of geometry involved in our proofs to a
minimum. In particular, not all the structures present on $\sD_G(G)$
have been explored. Notably, we avoided using the \emph{braided}
monoidal structure on this category: only the \emph{square} of the
braiding appears in the proof of Theorem \ref{t:dim-reps-easy}, and
only does so implicitly.

\subsection*{Acknowledgements} This paper is partially based on
my PhD dissertation. However, the results on easy
unipotent groups are new. I am deeply grateful to my PhD advisor,
Vladimir Drinfeld, who conjectured the main results of this article.
In many instances, the key ideas of proofs were due to him as well. In particular, he kindly allowed me to include Appendix B, which contains an argument of his that strengthens the first main result of the present paper.
I also thank the Clay Mathematics Institute and the National Science
Foundation for providing financial support during the period of time
when this paper was written. Last but not least, I am grateful to the anonymous referee for suggesting multiple improvements and corrections.



\section{Main definitions and results}\label{s:results}

In this section we state the two main results of our work (Theorems
\ref{t:dim-reps-easy} and \ref{t:descr-L-packets}), explaining most
of the relevant definitions (although some technical details are
postponed until later sections) and giving some historical
background. In \S\ref{ss:fake-Heis} we illustrate our theory by
describing the $\bL$-packets of irreducible representations of
$G(\bF_q)$, where $G$ is a so-called ``fake Heisenberg group'' over
$\bF_q$. The strategy we use to prove our main results is outlined
in Section \ref{s:structure} below.

\subsection{Conventions}\label{ss:conventions} If $k$ is any field, an
\emph{algebraic group} over $k$ is defined as a smooth group scheme
of finite type over $k$. We recall that ``smooth'' is equivalent to
``geometrically reduced'' in this situation, and if $k$ is perfect,
the word ``geometrically'' can be omitted. By a \emph{unipotent
group over $k$} we will mean a unipotent algebraic group (in
particular, smooth) over $k$. We denote by $\ql$ a fixed algebraic
closure of the field $\bQ_\ell$ of $\ell$-adic numbers, and whenever
the notation $\ql$ is used, we invariably assume that $\ell$ is a
prime different from the characteristic of the base field $k$.

\mbr

Our conventions regarding finite fields are as follows. Let $p$ be a
prime number, fixed once and for all, and let $\bF$ be a fixed
algebraic closure of the finite field $\bF_p$ with $p$ elements. If
$q=p^r$ for some $r\in\bN$, we write $\bF_q$ for the unique subfield
of $\bF$ consisting of $q$ elements. All representations of finite
groups that we consider are assumed to be defined over $\ql$, where
$\ell\neq p$. (This restriction only becomes relevant when we use
geometric methods coming from $\ell$-adic cohomology \cite{sga5}.
Much of our theory can be developed over an arbitrary algebraically
closed field of characteristic $0$, but for consistency we will work
over $\ql$ throughout this article.)

\subsection{Easy unipotent groups} Let us recall a
definition from \cite{intro}.

\begin{defin}\label{d:easy}
Let $k$ be a field and $\overline{k}$ an algebraic closure of $k$.
An algebraic group $G$ over $k$ is said to be \emph{easy} if every
$g\in G(\overline{k})$ is contained in the neutral connected
component, $Z(g)^\circ$, of its centralizer, $Z(g)$, in
$G\tens_k\overline{k}$.
\end{defin}

It is clear that an easy algebraic group $G$ over $k$ has to be
connected (if not, then applying the definition to any element $g\in
G(\overline{k})$ that does not belong to the neutral connected
component $(G\tens_k\overline{k})^\circ$ leads to a
contradiction).

\mbr

The group $GL_n$ is easy. A connected reductive group in
characteristic $0$ is easy if and only if its derived group is
simply connected and its center is connected. From this point on,
all easy groups discussed in this article will be unipotent.

\begin{rem} We know of no examples of easy unipotent groups $G$
that do not satisfy the stronger condition that the centralizer of
every geometric point of $G$ is \emph{connected}. It appears plausible that there are no such examples.
\end{rem}

We observe that if $k$ has characteristic zero, then every unipotent
group over $k$ is connected, and since closed subgroups of unipotent
groups are unipotent, it follows that every unipotent group over $k$
is easy. Therefore, from now on we will only be interested in the
case $\operatorname{char}k>0$.

\mbr

The first obvious example of an easy unipotent group in positive
characteristic is provided by $UL_n$, the so-called \emph{unipotent
linear group}, defined as the group of unipotent upper-triangular
matrices of size $n$. More generally, if $G$ is any reductive group
over $k$ and $U$ is a maximal connected unipotent subgroup of $G$,
then $U$ is easy provided the characteristic of $k$ is large enough
(depending on the types of the simple constituents of
$G\tens_k\overline{k}$). For instance, if $G$ is the symplectic
group $Sp_{2n}$, where $n\geq 2$, then $U$ is easy if and only if
$\operatorname{char}k>2$.

\mbr

Another type of generalizations of the group $UL_n$ comes from
the so-called ``algebra groups''. If $A$ is a finite
dimensional associative unital $k$-algebra, let $J$ be the Jacobson
radical\footnote{Since $A$ is clearly Artinian as a ring, $J$ is
also the maximal two-sided nilpotent ideal of $A$.} of $A$, and let
$G(A)$ denote the algebraic group over $k$ defined as follows. For
any commutative $k$-algebra $R$, we let $G(A)(R)$ denote the
multiplicative group of all elements of $R\tens_k A$ of the form
$1+x$, where $x\in R\tens_k J$. Then $G(A)$ is an
easy unipotent group over $k$ because the centralizers of geometric
points of $G(A)$ can be identified with linear subspaces of
$\overline{k}\tens_k J$. We call $G(A)$ the \emph{unipotent algebra
group} associated to $A$. Observe that if $A$ is the algebra of all
upper-triangular matrices of size $n$ over $k$, then $G(A)\cong
UL_n$.

\mbr

The example of a maximal unipotent subgroup $U$ of $Sp_4$ over the
finite field $\bF_{2^r}$, where $r\in\bN$, was originally considered
by Lusztig. This group is not easy. Lusztig computed the character
table of $U(\bF_{2^r})$ in \S7 of \cite{lusztig} and found that this
group has irreducible representations of dimension $2^{r-1}$. The
fake Heisenberg groups defined in \cite{intro} (see also \S\ref{ss:fake-Heis} below) are also not easy
(Lemma \ref{l:conn-sbgrps-2-dim}).

\subsection{Representations of algebra groups over finite fields}
It is known that the dimension of
every irreducible representation of $UL_n(\bF_q)$ is a power of $q$, which yields an affirmative answer to a question of Higman \cite{higman}.
A stronger and more general result is provided by

\begin{thm}[Halasi]\label{t:gut-is-hal}
If $A$ is a finite dimensional algebra over $\bF_q$, then every
irreducible representation of $G(A)(\bF_q)$ is induced from a
$1$-dimensional representation of a subgroup of the form
$G(B)(\bF_q)$, where $B\subset A$ is an $\bF_q$-subalgebra.
\end{thm}

\begin{cor}[Isaacs]\label{c:isaacs}
In the situation of Theorem \ref{t:gut-is-hal}, the dimension of
every irreducible representation of $G(A)(\bF_q)$ is a power of $q$.
\end{cor}

Theorem \ref{t:gut-is-hal} was first stated by
E.~Gutkin in \cite{gutkin}; however, Gutkin's proof of it was
incomplete. I.M.~Isaacs proved Corollary \ref{c:isaacs} in
\cite{isaacs}. Later, Z.~Halasi proved Theorem \ref{t:gut-is-hal} in
\cite{halasi}; it is worth noting that his proof uses Corollary
\ref{c:isaacs} in an essential way. A more direct proof of Theorem
\ref{t:gut-is-hal}, based on Halasi's methods, was given in
\cite{base-change}, and an improved version later appeared in \cite{gutkin-proof}.

\mbr

One of the main goals of this article is to extend Corollary
\ref{c:isaacs} to all easy unipotent groups over finite fields. The
result (Theorem \ref{t:dim-reps-easy}) is stated below.

\subsection{Character degrees of easy unipotent groups}
One of the main results of this paper is the following theorem,
proved in \S\ref{ss:proof-t:dim-reps-easy}.

\begin{thm}[Main Theorem 1]\label{t:dim-reps-easy}
If $G$ is an easy unipotent group over $\bF_q$, the dimension of
every irreducible representation of $G(\bF_q)$ is a power of $q$.
\end{thm}

This result was conjectured by V.~Drinfeld in 2005. In our opinion,
it explains the ``geometry behind [the positive answer to] Higman's question.''

\mbr

After the first version of this article was written, V.~Drinfeld informed us that the following extension of Theorem \ref{t:gut-is-hal} (which gives a weaker result when applied to unipotent algebra groups) can be proved.

\begin{thm}\label{t:strong-higman}
If $G$ is an easy unipotent group over $\bF_q$,
then every irreducible representation of $G(\bF_q)$ is induced from a $1$-dimensional representation of a subgroup of the form $P(\bF_q)$, where $P\subset G$ is a closed \emph{connected}
subgroup.
\end{thm}

We note that this result implies Theorem \ref{t:dim-reps-easy}, because the index of $P(\bF_q)$ in $G(\bF_q)$ equals $q^{\dim G-\dim P}$. However, the proof of Theorem \ref{t:strong-higman} relies on many of the key ingredients needed for our proof of Theorem \ref{t:dim-reps-easy}. With his kind permission, we reproduce Drinfeld's proof of Theorem \ref{t:strong-higman} in Appendix B.

\subsection{From easy to arbitrary connected unipotent groups}
It turns out that in order to prove Theorem \ref{t:dim-reps-easy}
one has to formulate and prove a more general statement about
irreducible characters of $G(\bF_q)$ for an \emph{arbitrary}
connected unipotent group $G$ over $\bF_q$. The reason is that all
approaches to representation theory for unipotent groups known to us
are based on induction on $\dim G$ in one way or another, reducing
the questions one is interested in to similar questions for
subgroups of $G$ of smaller dimension. For instance, the proof of
Theorem \ref{t:gut-is-hal} ultimately relies on the possibility of
constructing many nontrivial multiplicatively closed subspaces
inside the Jacobson radical $J(A)$ of a finite dimensional algebra
$A$. However, if $G$ is an arbitrary easy unipotent group over
$\bF_q$, it is not known to us how to construct sufficiently many
\emph{easy} subgroups of $G$ to make it possible to give an
inductive proof of Theorem \ref{t:dim-reps-easy}. On the other hand,
$G$ has lots of \emph{connected} closed subgroups, and most of our
paper is devoted to the study of arbitrary connected unipotent
groups over finite fields.

\subsection{Definition of $\bL$-indistinguishability}\label{ss:def-L-indist}
In the remainder of this section we will freely use the language of
$\ell$-adic cohomology \cite{sga4,sga4.5,sga5}. A brief review of
the terminology appears in the first half of Section
\ref{s:equivariant-derived}.

\mbr

Let $G$ be a connected unipotent group over $\bF_q$, and let
$\mu:G\times G\rar{}G$ be the multiplication morphism. The
definition of the equivariant derived category $\sD_G(G)$, together
with the functor of convolution with compact
supports,
\[
\sD_G(G)\times\sD_G(G)\rar{}\sD_G(G), \qquad (M,N)\longmapsto
M*N=R\mu_!(M\boxtimes N),
\]
is recalled in \S\ref{ss:der-equiv-der-convolution}. An object
$e\in\sD_G(G)$ is said to be a \emph{weak idempotent} if $e*e\cong
e$. If this holds, it is clear that the associated trace function
$t_e:G(\bF_q)\rar{}\ql$ is a central idempotent with respect to the
usual convolution on the space of $\ql$-valued functions on
$G(\bF_q)$. In particular, $t_e$ acts either as zero or as the
identity in every irreducible representation of $G(\bF_q)$.

\begin{defin}\label{d:L-packets}
Two irreducible representations, $\rho_1$ and $\rho_2$, of
$G(\bF_q)$ are said to be \emph{$\bL$-indistinguishable} if
for every weak idempotent $e\in\sD_G(G)$, the function $t_e$ acts in
the same way in $\rho_1$ and in $\rho_2$. The equivalence classes
with respect to the relation of $\bL$-indistinguishability are called \emph{Lusztig packets} of irreducible representations.
\end{defin}

From here on, for brevity, we write ``$\bL$-packet'' in place of ``Lusztig packet.''

\begin{rem}\label{r:langlands}
The conjectural notion of an $L$-packet in representation theory
of reductive groups over \emph{local} fields was introduced by
R.P.~Langlands in \cite{langlands}. It is hard to compare it with
the notion of an $\bL$-packet because technically the two definitions
are given in quite different terms. However, philosophically the two notions
are closely related. Namely, as explained by R.~Bezrukavnikov, $\bL$-packets are philosophically similar to \emph{geometric} $L$-packets, which are, in general, larger than the $L$-packets defined by Langlands\footnote{Conjecturally, $L$-packets bijectively correspond to ``Langlands parameters.''  Geometric $L$-packets should correspond to Frobenius-invariant ``geometric Langlands parameters'' (one obtains geometric Langlands parameters from the usual ones by replacing the Weil-Deligne group $W'_K$ with $\Ker (W'_K\twoheadrightarrow\bZ)$). Thus each geometric $L$-packet is a union of several ordinary $L$-packets.}.
\end{rem}

\subsection{Multiplicative local
systems}\label{ss:intro-mult-loc-sys} If $G$ is an arbitrary
connected unipotent group over $\bF_q$, it is not at all clear how
to describe all weak idempotents in the category $\sD_G(G)$. For
instance, it is not even obvious that there are any apart from the
zero object and the unit object $\e$. In \S\ref{ss:descr-L-packets}
we will state our second main result (Theorem
\ref{t:descr-L-packets}), which yields a description of $\bL$-packets
of irreducible representations of $G(\bF_q)$ in terms of more
concrete objects, the so-called ``admissible pairs''
(\S\ref{ss:intro-adm-pairs}) for $G$. This description is one of the
key ingredients in our proof of Theorem \ref{t:dim-reps-easy}. In
\S\ref{ss:fake-Heis} below, we show how it can be used to describe
all $\bL$-packets of irreducible representations of $G(\bF_q)$ when
$G$ is a fake Heisenberg group over $\bF_q$. We begin by introducing

\begin{defin}\label{d:multiplicative}
If $k$ is a field and $\ell$ is a prime different from
$\operatorname{char}k$, a \emph{nonzero} $\ql$-local system $\cL$ on
a connected algebraic group $H$ over $k$ is said to be
\emph{multiplicative} if $\mu^*(\cL)\cong\cL\boxtimes\cL$, where
$\mu:H\times_k H\rar{}H$ denotes the multiplication morphism.
\end{defin}

\begin{rem}\label{r:character-mult-loc-sys}
If $k=\bF_q$ and $\cL$ is a multiplicative $\ql$-local system on
$H$, it is clear that the ``trace function''
(\S\ref{ss:sheaves-to-functions}) $t_{\cL}$ defined by $\cL$ is a
homomorphism $H(\bF_q)\rar{}\ql^\times$. Moreover, $\cL$ can be
recovered from $t_{\cL}$ up to isomorphism. If $H$ is
\emph{commutative}, every homomorphism $H(\bF_q)\rar{}\ql^\times$
arises in this way (cf.~Proposition \ref{p:Serre-Pontryagin}). For
noncommutative $H$, this statement fails in general, even if $H$ is
unipotent (cf. the example of the fake Heisenberg groups discussed
in \cite{intro} and in \S\ref{ss:fake-Heis}).
\end{rem}

\subsection{Admissible pairs}\label{ss:intro-adm-pairs}
Let $(H,\cL)$ denote a pair consisting of a closed connected
subgroup $H\subset G$ and a multiplicative $\ql$-local system $\cL$
on $H$. The notion of what it means for this pair to be
\emph{admissible} is introduced in \S\ref{ss:def-admissible}. The
precise definition is somewhat technical, so here we will only
remark that admissibility is a certain geometric nondegeneracy
condition. (In this context, the word ``geometric'' refers to the
fact that this property depends only on the triple
$(G\tens_{\bF_q}\bF,H\tens_{\bF_q}\bF,\cL\tens_{\bF_q}\bF)$ obtained
from $(G,H,\cL)$ by base change to $\bF$, an algebraic closure of
$\bF_q$.) It should be thought of as a geometrization of the
following purely algebraic version.

\begin{defin}
Let $\Ga$ be a finite group, and consider a pair $(H,\chi)$
consisting of a subgroup $H\subset\Ga$ and a homomorphism
$\chi:H\rar{}\ql^\times$. Let $\Ga'$ be the stabilizer of the pair
$(H,\chi)$ for the conjugation action of $\Ga$. We say that the pair
$(H,\chi)$ is admissible if the following three conditions are
satisfied:
\begin{enumerate}[$($1$)$]
\item $\Ga'/H$ is commutative;
\item the map $B_\chi:(\Ga'/H)\times(\Ga'/H)\rar{}\ql^\times$
induced by
$(\ga_1,\ga_2)\longmapsto\chi(\ga_1\ga_2\ga_1^{-1}\ga_2^{-1})$
(which by (1) is well defined and bi-additive) is a perfect
pairing of finite abelian groups, i.e., induces an isomorphism
$\Ga'/H\rar{\simeq}\Hom(\Ga'/H,\ql^\times)$; and
\item for every $g\in\Ga$, $g\not\in\Ga'$, we have
$\chi\big\lvert_{H\cap H^g}\neq\chi^g\bigl\lvert_{H\cap H^g}$, where
$H^g=g^{-1}Hg$ and $\chi^g:H^g\rar{}\ql^\times$ is obtained from
$\chi$ by transport of structure.
\end{enumerate}
\end{defin}

\begin{rem}
Conditions (1) and (2) in the algebraic definition of admissibility
imply that the group $\Ga'$ has a unique irreducible representation
$\pi_\chi$ over $\ql$ which acts on $H$ by the homomorphism $\chi$.
Condition (3) further implies that the induced representation
$\Ind_{\Ga'}^\Ga\pi_\chi$ is irreducible $($in view of Mackey's
irreducibility criterion$)$. The geometric notion of admissibility
serves a somewhat similar purpose.
\end{rem}

\subsection{Explicit description of
$\bL$-packets}\label{ss:descr-L-packets} We now return to the
geometric setting. Let $G$ be a connected unipotent group over
$\bF_q$, and $(H_1,\cL_1)$, $(H_2,\cL_2)$ two pairs consisting of
closed connected subgroups $H_1,H_2\subset G$ and multiplicative
local systems $\cL_j$ on $H_j$ ($j=1,2$). We say that these pairs
are \emph{geometrically conjugate} if there exists $g\in G(\bF)$
which conjugates one of them into the other. Note that, in general,
geometric conjugacy is weaker than conjugacy by an element of
$G(\bF_q)$.

\begin{defin}\label{d:adm-pair-L-packet}
Let $\sC$ be a geometric conjugacy class of admissible pairs
$(H,\cL)$ as above for $G$. We define a set $L(\sC)$ of (isomorphism
classes of) irreducible representation of $G(\bF_q)$ over $\ql$ as
follows. We say that $\rho\in L(\sC)$ if there exists
$(H,\cL)\in\sC$ such that $\rho$ is an irreducible summand of
$\Ind_{H(\bF_q)}^{G(\bF_q)} t_\cL$.
\end{defin}

It is immediate that each of the sets $L(\sC)$ is nonempty. The
second main result of our paper claims that the sets $L(\sC)$ are
precisely the $\bL$-packets of irreducible representations of
$G(\bF_q)$. We prove it in \S\ref{ss:proof-t:descr-L-packets}. This
result, along with the definition of admissible pairs, was also
formulated by V.~Drinfeld.

\begin{thm}[Main Theorem 2]\label{t:descr-L-packets}
Let $G$ be a connected unipotent group over $\bF_q$. For every
geometric conjugacy class $\sC$ of admissible pairs for $G$, the set
$L(\sC)$ is an $\bL$-packet. Conversely, every $\bL$-packets of
irreducible representations of $G(\bF_q)$ is of the form $L(\sC)$
for some geometric conjugacy class $\sC$ of admissible pairs.
\end{thm}

\begin{cor} With the notation above,
\begin{enumerate}[$($a$)$]
\item every irreducible representation of $G(\bF_q)$ over $\ql$ lies
in $L(\sC)$ for some geometric conjugacy class $\sC$ of admissible
pairs for $G$; and
\item if $\sC_1$ and $\sC_2$ are two geometric conjugacy classes of admissible pairs for $G$, the sets
$L(\sC_1)$ and $L(\sC_2)$ are either equal or disjoint.
\end{enumerate}
\end{cor}

Note that, \emph{a priori}, neither of the statements of this
corollary is obvious.

\subsection{Example: $\bL$-packets for the fake Heisenberg
groups}\label{ss:fake-Heis} We conclude this overview with an
example which to some extent motivated the notion of an admissible
pair. If $k$ is a field of characteristic $p>2$, we define a
\emph{fake Heisenberg group} over $k$ to be a connected
noncommutative unipotent algebraic group $G$ over $k$ of exponent
$p$ and dimension $2$ (hence the word ``fake''). The reason for
imposing the restriction $p>2$ is that every $2$-dimensional
unipotent group in characteristic $0$ is commutative (which follows
from the corresponding statement for Lie algebras), and that every
group of exponent $2$ is commutative\footnote{In characteristic $2$,
there also exist connected noncommutative $2$-dimensional unipotent
groups, but they all have exponent $4$.}.

\mbr

However, if $p>2$, there are plenty of examples of fake Heisenberg
groups over $\bF_q$: see \cite{intro}.
Here we will describe the $\bL$-packets of irreducible representations
for such groups. We begin with a simple auxiliary result.

\begin{lem}\label{l:conn-sbgrps-2-dim}
If $G$ is a noncommutative connected unipotent group of dimension
$2$ over a field $k$, then the only nontrivial proper closed
connected subgroup of $G$ is its commutator, $[G,G]$. Moreover, such
a group $G$ is not easy $($Definition \ref{d:easy}$)$.
\end{lem}

\begin{proof}{Proof}
The assumptions imply that $[G,G]$ is connected and $\dim[G,G]=1$,
so that $\dim G^{ab}=1$ as well, where $G^{ab}=G/[G,G]$ is the
abelianization of $G$. Moreover, $[G,G]$ is contained in the center
of $G$. Now let $H\subset G$ be a proper closed connected subgroup,
and suppose $H\not\subset [G,G]$. Then $H$ projects epimorphically
onto $G^{ab}$, which implies that $G=H\cdot[G,G]$. Hence
$[G,G]\not\subset H$ as well. Therefore $[H,H]\subset H\cap[G,G]$ is
connected and $0$-dimensional, whence trivial. Thus $H$ is
commutative. This implies that $G$ is commutative, which is a
contradiction.

\mbr

For the second claim, note that if $g$ is a geometric point of $G$
which does not lie in the center of $G$, then $Z(g)\neq
G\tens_k\overline{k}$, whereas $[G,G]\tens_k\overline{k}\subset
Z(g)$, whence $Z(g)^\circ=[G,G]\tens_k \overline{k}$, which implies
that $g\not\in Z(g)^\circ(\overline{k})$.
\end{proof}

Let $G$ be a fake Heisenberg group over $\bF_q$, and consider a pair
$(H,\cL)$ consisting of a closed connected subgroup $H\subset G$ and
a multiplicative $\ql$-local system $\cL$ on $H$. If $H=G$, this
pair is trivially admissible; its geometric conjugacy class $\sC$
reduces to the single pair $(H,\cL)$; and the corresponding
$\bL$-packet $L(\sC)$ consists of the single $1$-dimensional
representation $t_{\cL}:G(\bF_q)\rar{}\ql^\times$.

\mbr

Usually, however, not every $\bL$-packet of irreducible
representations of $G(\bF_q)$ is of this form. For instance, if
$G(\bF_q)$ is noncommutative, it has irreducible representations of
dimension $>1$. On the other hand, in many cases where $G(\bF_q)$
\emph{is} commutative, not every $1$-dimensional representation of
$G(\bF_q)$ comes from a multiplicative local system on $G$.

\mbr

To find other $\bL$-packets of irreducible representations of
$G(\bF_q)$, we must allow $H\neq G$. It is clear that $H$ cannot be
trivial, so by Lemma \ref{l:conn-sbgrps-2-dim}, the only remaining
possibility is $H=[G,G]$. In this case $H$ is central in $G$, so
every $\ql$-local system on $H$ is automatically $G$-invariant. It
is easy to check that if $\cL$ is a multiplicative $\ql$-local
system on $H$, the pair $(H,\cL)$ is admissible for $G$ if and only
if $\cL$ is nontrivial. Moreover, in this case, the geometric
conjugacy class $\sC$ of $(H,\cL)$ also reduces to the single pair
$(H,\cL)$, and the corresponding $\bL$-packet $L(\sC)$ consists of all
irreducible representations of $G(\bF_q)$ that act by the scalar
$t_\cL$ on $[G,G](\bF_q)$.



\section{The structure of the proofs}\label{s:structure}

In this section we describe the methods we used to prove Theorems
\ref{t:dim-reps-easy} and \ref{t:descr-L-packets}, stating several
other results that are interesting in their own right along the way.

\mbr

The main technical tools used in our proofs are:
\begin{itemize}
\item the equivariant derived category $\sD_G(G)$ for a unipotent
group $G$, along with the bifunctor $(M,N)\longmapsto M*N$ of
convolution with compact supports and the collection of ``twists''
$\te_M:M\rar{\simeq}M$ defined for all $M\in\sD_G(G)$;
\item the functor of induction with compact supports
$\ind_{G'}^G:\sD_{G'}(G')\rar{}\sD_G(G)$, defined for any closed
subgroup $G'\subset G$; and
\item the notion of an admissible pair for $G$, along with an extension of Serre
duality \cite{serre} to noncommutative connected unipotent groups.
\end{itemize}
The first two of these are introduced in Sections
\ref{s:equivariant-derived} and \ref{s:induction}, respectively,
where we also establish some auxiliary results involving these
technical tools. The extension of Serre duality to the
noncommutative setting, along with some new results on the classical
Serre duality and bi-extensions of connected commutative unipotent
groups by $\qzp$, appears in the (rather extensive) Appendix.
Admissible pairs are defined in \S\ref{ss:def-admissible}. The
proofs of Theorems \ref{t:dim-reps-easy} and \ref{t:descr-L-packets}
occupy Sections \ref{s:reduction}--\ref{s:proofs}; together, they
can be split into the following sequence of steps.

\subsection{Step 1}\label{ss:step-1} Let $G$ be a connected
unipotent group over $\bF_q$. We begin by proving the one result
which explicitly relates representations of $G(\bF_q)$ with the
geometry of $G$: namely, that for every irreducible representation
$\rho$ of $G(\bF_q)$ over $\ql$, there exists a geometric conjugacy
class $\sC$ of admissible pairs for $G$ such that $\rho$ lies in
$L(\sC)$ (cf.~Definition \ref{d:adm-pair-L-packet}). Theorem
\ref{t:adm-pair-compatible} gives a slightly more precise statement.

\mbr

One of the ingredients in this step should be useful in other
situations. Namely, in Proposition \ref{p:ext-loc-sys} we formulate
a condition under which a multiplicative $\ql$-local system on a
closed connected subgroup $H$ of a connected unipotent group $G$ can
be extended to a multiplicative $\ql$-local system on all of $G$.

\subsection{Step 2}\label{ss:step-2} Next we relate admissible
pairs to $\bL$-packets. If $(H,\cL)$ is an admissible pair for a
unipotent group $G$ over an arbitrary field $k$, we consider the
object $e_{\cL}=\bK_H\tens\cL\in\sD_H(H)$, where $\bK_H\in\sD_H(H)$
is the dualizing complex of $H$. One checks easily that
$e_{\cL}*e_{\cL}\cong e_{\cL}$, i.e., $e_{\cL}$ is a weak
idempotent. If $G'$ is the stabilizer of $(H,\cL)$ for the
conjugation action of $G$, it is easy to see that the extension of
$e_{\cL}$ by zero to all of $G'$ defines an object
$e'_{\cL}\in\sD_{G'}(G')$. Of course, $e'_{\cL}$ is also a weak
idempotent. Finally, we apply the functor of induction with compact
supports, and we show that $e_{H,\cL}:=\ig e'_{\cL}$ is a
\emph{minimal weak idempotent} in $\sD_G(G)$, i.e., a nonzero weak
idempotent such that if $e\in\sD_G(G)$ is any weak idempotent, then
either $e_{H,\cL}*e=0$, or $e_{H,\cL}*e\cong e_{H,\cL}$. All these
results are proved in Sections \ref{s:heis-idemp} and
\ref{s:proofs}.

\mbr

One of the ingredients here is a more general result,
proved in \S\ref{ss:ind-weak-idemp}, which gives a condition on a given weak idempotent $e\in\sD_{G'}(G')$ under which $f=\ig(e)$
is a weak idempotent in $\sD_G(G)$ and the functor $\ig$ restricts
to an \emph{equivalence of semigroupal categories}
$e*\sD_{G'}(G')\rar{\sim}f*\sD_G(G)$. The condition is reminiscent
of Mackey's criterion for the irreducibility of an induced
representation.

\subsection{Step 3}\label{ss:step-3} We explore the relationship
between the functor $\ig$ and the operation of induction of class
functions studied in \S\ref{ss:induction-sheaves-functions} to prove
that in the situation of Step 2, if $G$ is connected, $k=\bF_q$, and
$\sC$ is the geometric conjugacy class of the admissible pair
$(H,\cL)$, then the set $L(\sC)$ of irreducible
$\ql$-representations of $G(\bF_q)$ introduced in Definition
\ref{d:adm-pair-L-packet} coincides with the set of irreducible
$\ql$-representations of $G(\bF_q)$ on which the trace function
$t_{e_{H,\cL}}:G(\bF_q)\rar{}\ql$ acts as the identity.
\begin{rem}
The functor $\ig$ is often not compatible with induction of class
functions on the nose (unless $G'$ is connected), which is why we
must work with geometric conjugacy classes of admissible pairs,
rather than $G(\bF_q)$-conjugacy classes.
\end{rem}

\mbr

After we put the previous steps together, proving Theorem
\ref{t:descr-L-packets} becomes very easy. Namely, let
$e\in\sD_G(G)$ be any weak idempotent. If $t_e\equiv 0$, then we can
discard $e$ while trying to describe $\bL$-packets. Otherwise there
exists an irreducible $\ql$-representation $\rho$ of $G(\bF_q)$ on
which $t_e$ acts nontrivially. By Step 1, there exists a geometric
conjugacy class $\sC$ of admissible pairs for $G$ such that $\rho\in
L(\sC)$. By Step 3, if $(H,\cL)\in\sC$, then $t_{e_{H,\cL}}$ acts
nontrivially on $\rho$. This implies that
$t_e*t_{e_{H,\cL}}\not\equiv 0$. A fortiori, $e*e_{H,\cL}\neq 0$ (as
convolution of functions is clearly compatible with the convolution
with compact supports of $\ell$-adic complexes). By Step 2, this
implies that $e*e_{H,\cL}\cong e_{H,\cL}$, and therefore, applying
Step 3 again, we see that $e$ acts as the identity on every
irreducible representation of $G(\bF_q)$ appearing in $L(\sC)$. This
result, together with the statement proved in Step 1, implies
Theorem \ref{t:descr-L-packets}.

\subsection{Step 4} Now let $G$ be an easy unipotent group over
$\bF_q$. In the proof of Theorem \ref{t:dim-reps-easy} we use the
result of Step 1 above (but not the results of Steps 2 and 3). Thus
let $\rho$ be an irreducible $\ql$-representation of $G(\bF_q)$, and
choose an admissible pair $(H,\cL)$ for $G$ such that $\rho$ is an
irreducible summand of $\Ind_{H(\bF_q)}^{G(\bF_q)}t_{\cL}$.

\mbr

We employ the compatibility of the functor $\ig$ with twists
(Proposition \ref{p:induction-twists}) and the triviality of twists
in $\sD_G(G)$ (Lemma \ref{l:triv-twists-easy}) to prove that if
$G'$ is the stabilizer of $(H,\cL)$ for the conjugation action of
$G$, then $G'$ is necessarily connected and the homomorphism
$(G'/H)_{per\!f}\rar{}(G'/H)_{per\!f}^*$ appearing in the definition
of an admissible pair (see \S\ref{ss:def-admissible}) is an
isomorphism (not merely an isogeny). Here, $(G'/H)_{per\!f}$ denotes
the \emph{perfectization} of the group $G'/H$ (see
\S\ref{aa:perfect}), and $(G'/H)_{per\!f}^*$ is its \emph{Serre
dual} (see \S\ref{ss:Serre-comment} and \S\ref{aa:Serre-duality}).

\mbr

From this we deduce that $G'(\bF_q)$ has a unique irreducible
$\ql$-representation $\rho'$ which acts by the scalar $t_{\cL}$ on
$H(\bF_q)$. Mackey's criterion implies that
$\Ind_{G'(\bF_q)}^{G(\bF_q)}\rho'$ is irreducible, and Frobenius
reciprocity forces $\rho\cong\Ind_{G'(\bF_q)}^{G(\bF_q)}\rho'$. In
particular, $\dim\rho=q^{\dim G-\dim G'}\cdot\dim\rho'$ (because
$G'$ is connected).

\subsection{Step 5} To complete the proof of Theorem
\ref{t:dim-reps-easy} we must demonstrate that, in the situation of
the previous step, the dimension of $\rho'$ is a power of $q$. Since
$\dim\rho'=[G'(\bF_q):H(\bF_q)]^{1/2}$, this is the same as showing
that $\dim(G'/H)$ is even. In view of the fact that the canonical
map $(G'/H)_{per\!f}\rar{}(G'/H)^*_{per\!f}$ is an isomorphism, this
follows from a more general result, Proposition
\ref{p:existence-lagr}, proved in \S\ref{aa:lagrangian}.

\begin{rem}
As we already mentioned in the Introduction, in this paper we
do not define or use the braided monoidal structure on the category
$\sD_G(G)$, without which the significance of the ``twists'' in
$\sD_G(G)$ cannot be fully appreciated (see \cite{intro}). However,
we believe that the full power of the geometric techniques should be
reserved for the theory of character sheaves.
\end{rem}

\bbr

The reader who is only interested in understanding the general ideas
behind our arguments does not have to read any further. The missing
details of the proofs sketched above are filled in the remaining
sections, which are more technical.

\subsection{On Serre duality}\label{ss:Serre-comment} We end
with a comment on the notion of a multiplicative $\ql$-local system
used in the main body of the paper and the Serre duality studied in
the Appendix. In the proofs of our main results, Serre duality
serves mostly as a tool, and if $G$ is a connected unipotent group
over a perfect field $k$ of characteristic $p>0$, we think of the
Serre dual $G^*$ of $G$ as the ``moduli space of multiplicative
$\ql$-local systems on $G$''. However, if one wishes to prove
foundational results about Serre duality, the most natural framework
(which, in particular, is independent of $\ell$) is that of
\emph{central extensions} by the discrete group $\qzp$. It would
have been inconvenient for us to choose one of these viewpoints once
and for all, and to completely discard the other one. The
relationship between them is described in
\S\ref{ss:Serre-duality-two-approaches}.



\section{The category $\sD_G(G)$ and $\bL$-packets}\label{s:equivariant-derived}

\subsection{Derived categories of constructible
$\ell$-adic complexes}\label{ss:der-cat-constr-sheaves} Fix an
arbitrary field $k$ and a prime $\ell\neq\operatorname{char}k$ (in
\S\ref{ss:sheaves-to-functions} and \S\ref{ss:idemp-L-packets}, we
take $k$ to be finite). Throughout this section we will work with
schemes of finite type over $k$. If $X$ is such a scheme, one defines the bounded derived category $D^b_c(X,\ql)$ of
constructible complexes of $\ql$-sheaves on $X$. We will
denote this category simply by $\sD(X)$, with the understanding that
$\ell$ is fixed once and for all. It is a triangulated $\ql$-linear
category. For perfect $k$ the definition of $\sD(X)$ appears in \cite{ekedahl}, and in general we define $\sD(X)=\sD(X\tens_k k^{perf})$, where $k^{perf}$ is the perfect closure of $k$.

\begin{rem}
For the purposes of this work, it would be enough to consider the case where $k$ is finite or algebraically closed. Here the definition of $\sD(X)$ is more classical \cite{deligne-weil-2,sga4,sga4.5}. However, with future applications in mind, we consider the more general case in this section.
\end{rem}

We will often use Grothendieck's ``formalism of the six
functors'' for the categories $\sD(X)$ (as well as their equivariant versions, defined in \S\ref{ss:equivariant-derived} below). For a morphism $f:X\rar{}Y$ of $k$-schemes of finite type
one has the pullback functor $f^*:\sD(Y)\rar{}\sD(X)$,
the pushforward functor $f_*:\sD(X)\rar{}\sD(Y)$,
the functor $f_!:\sD(X)\rar{}\sD(Y)$ (pushforward with compact supports),
and the functor $f^!:\sD(Y)\rar{}\sD(X)$. \emph{We always omit the letters
``$L$'' and ``$R$'' from our notation for the six functors; thus,
$f_!$ stands for $Rf_!$ and $\tens$ stands for
$\overset{L}{\tens}_{\ql}$, etc.}

\begin{rem}
In \cite{sga4.5} the functor $f_!$ is defined for separated morphisms $f$ when $k$ is finite or algebraically closed. This case would suffice for the purposes of the present work. However, the formalism we need was extended to arbitrary fields $k$ in \cite{ekedahl}, and the assumption that $f$ is separated is unnecessary \cite{Las-Ols06}.
\end{rem}

The most important result we will need is the proper base change theorem;
see Exp. XII and XVII in \cite{sga4} and Exp.~IV in
\cite{sga4.5} for the case where $k$ is finite or algebraically closed; and Theorem 6.3(iii) in \cite{ekedahl} for the general case.

\begin{thm}[Proper base change]\label{t:proper-base-change} Consider a cartesian
square
\[
\xymatrix{
  X' \ar[rr]^{g'} \ar[d]_{f'} & & X \ar[d]^f \\
  Y' \ar[rr]^g & & Y
   }
\]
of $k$-schemes of finite type. There is a natural
isomorphism of functors
\begin{equation}\label{e:PBC}
(g^*\circ f_!) \cong (f'_!\circ g^{\prime\ast}) :
\sD(X)\rar{}\sD(Y').
\end{equation}
\end{thm}

\subsection{Reminder on the sheaves-to-functions
correspondence}\label{ss:sheaves-to-functions} In this subsection we
assume that the base field is finite: $k=\bF_q$. Let $X$ be a scheme
of finite type over $\bF_q$. Given an object $M\in\sD(X)$, one can
define the corresponding function $t_M:X(\bF_q)\rar{}\ql$. Namely, a
point $x\in X(\bF_q)$ can be thought of as an $\bF_q$-morphism
$x:\Spec\bF_q\rar{}X$. Then $x^*M\in\sD(\Spec\bF_q)$, and the
cohomology sheaves $\cH^i(x^* M)\cong x^*\cH^i(M)$ are constructible
$\ell$-adic sheaves on $\Spec\bF_q$, i.e., continuous finite
dimensional representations of the absolute Galois group
$\Gal(\bF/\bF_q)$ over $\ql$. Let $F_q\in\Gal(\bF/\bF_q)$ be the
geometric Frobenius, defined as the inverse of the Frobenius
substitution $a\longmapsto a^q$. Then one defines
\[
t_M(x) = \sum_{i\in\bZ} (-1)^i\cdot \tr\bigl( F_q; \cH^i(x^*M)
\bigr).
\]
The main properties of the map $M\longmapsto t_M$ are summarized in

\begin{lem}\label{l:sheaves-functions-properties}
Let $X$ and $Y$ be schemes of finite type over $\bF_q$, and let
$f:X\rar{}Y$ be an $\bF_q$-morphism.
\begin{enumerate}[$(1)$]
\item If $N\in\sD(Y)$, then $t_{f^* N}=f^* t_N \,\overset{\text{def}}{:=}\, t_N\circ
f$. \sbr
\item If $M,K\in\sD(X)$, then $t_{M\tens K}=t_M\cdot t_K$ $($pointwise
product$)$. \sbr
\item Assume that $f$ is separated. If $M\in\sD(X)$, then $t_{f_! M}=f_! t_M$, where, by abuse of
notation, we also write $f$ for the induced map of sets
$X(\bF_q)\rar{}Y(\bF_q)$, and $(f_! t_M)(y) = \sum_{x\in f^{-1}(y)}
t_M(x)$.
\end{enumerate}
\end{lem}

Of these, (1) and (2) follow rather easily from the definitions,
while (3) is more subtle. It follows from the proper base change
theorem and the special case of (3) where $Y=\Spec\bF_q$, which is
known as the \emph{Lefschetz-Grothendieck trace formula}; see
Theorem 3.2  of ``Rapport sur la formule des traces'' in
\cite{sga4.5}.

\subsection{Equivariant derived
categories}\label{ss:equivariant-derived} We return to the situation
where the base field $k$ is arbitrary. Let $G$ be an algebraic group
over $k$, let $X$ be a scheme of finite type over $k$, and suppose
that we are given a regular left action of $G$ on $X$. We would like
to define the ``equivariant derived category'' $\sD_G(X)$.

\mbr

In general, to get the correct definition one must either adopt the
approach of Bernstein and Lunts \cite{ber-lunts} (when $G$ is
affine), or use the definition of $\ell$-adic derived categories for
Artin stacks due to Laszlo and Olsson \cite{Las-Ols06} and define
$\sD_G(X)=\sD(G\bigl\backslash X)$, where $G\bigl\backslash X$ is
the quotient stack of $X$ by $G$.

\mbr

From now on we assume that $G$ is unipotent. In this case one knows
that the naive definition of $\sD_G(X)$ (taken from \cite{intro}),
given below, already gives the correct answer. Roughly speaking,
this definition amounts to looking at the ``category of
$G$-equivariant objects in $\sD(X)$''.

\mbr

Let us write $\al:G\times X\to X$ for the action morphism and
$\pi:G\times X\to X$ for the projection. Let $\mu:G\times G\to G$ be
the product in $G$. Let $\pi_{23}:G\times G\times X\to G\times X$ be
the projection along the first factor $G$. The category $\sD_G(X)$
is defined as follows.

\begin{defin}\label{d:equiv-derived}
An {\em object} of the category $\sD_G(X)$ is a pair $(M,\phi)$,
where $M\in\sD(X)$ and $\phi:\al^*M\rar{\simeq}\pi^*M$ is an
isomorphism in $\sD(G\times_k X)$ such that
\begin{equation}\label{e:compatibility}
\pi_{23}^*(\phi)\circ(\id_G\times\al)^*(\phi) =
(\mu\times\id_X)^*(\phi),
\end{equation}
i.e., the composition of the natural isomorphisms
\[
(\id_G\times\al)^*\al^* M \cong (\mu\times\id_X)^*\al^* M \xrar{\
(\mu\times\id_X)^*(\phi)\ } (\mu\times\id_X)^*\pi^* M \cong
\pi_{23}^*\pi^* M
\]
equals the composition
\[
(\id_G\times\al)^*\al^* M \xrar{\ (\id_G\times\al)^*(\phi)\ }
(\id_G\times\al)^*\pi^* M \cong \pi_{23}^*\al^* M \xrar{\
\pi_{23}^*(\phi)\ } \pi_{23}^*\pi^* M.
\]
A {\em morphism} $(M,\phi)\rar{}(N,\psi)$ in $\sD_G(X)$ is a
morphism $\nu:M\rar{}N$ in $\sD(X)$ satisfying
$\psi\circ\al^*(\nu)=\pi^*(\nu)\circ\phi$. The {\em composition} of
morphisms in $\sD_G(X)$ is defined to be equal to their composition
in $\sD(X)$.
\end{defin}

\begin{rem}
If $G$ is a connected unipotent group, the forgetful functor $\sD_G(X)\rar{}\sD(X)$ is fully faithful.
\end{rem}

\subsection{Functors between equivariant derived
categories}\label{ss:funct-equiv-derived} In the situation of
\S\ref{ss:equivariant-derived}, let us assume that $H$ is another
unipotent group over $k$ acting on a scheme $Y$ of finite type over
$k$. Suppose we are given a homomorphism $i:G\rar{}H$ of $k$-groups
and a morphism $f:X\rar{}Y$ of $k$-schemes which is $G$-equivariant
with respect to the $G$-action on $Y$ induced by $i$. Then the functor $f^*:\sD(Y)\rar{}\sD(X)$
naturally lifts to a functor $f^*:\sD_H(Y)\rar{}\sD_G(X)$.

\mbr

In the special case where $H=G$ and $i$ is the
identity, we can also define a functor $f_!:\sD_G(X)\rar{}\sD_G(Y)$.
Indeed, we have cartesian diagrams
\[
\xymatrix{
  G\times X \ar[rr]^{\id_G\times f} \ar[d]_{\al_X} & & G\times Y \ar[d]^{\al_Y} \\
  X \ar[rr]^f & & Y
   }
 \qquad\text{and}\qquad
\xymatrix{
  G\times X \ar[rr]^{\id_G\times f} \ar[d]_{\pi_X} & & G\times Y \ar[d]^{\pi_Y} \\
  X \ar[rr]^f & & Y
   }
\]
where $\al_X$, $\al_Y$ are the action morphisms and $\pi_X$, $\pi_Y$
are the projections, so Theorem \ref{t:proper-base-change} implies
that $f_!:\sD(X)\rar{}\sD(Y)$ lifts to a functor
$f_!:\sD_G(X)\rar{}\sD_G(Y)$.

\mbr

From now on we assume that if $f$ is an equivariant morphism between
$k$-schemes of finite type equipped with a $G$-action, then $f^*$
and $f_!$ are understood as functors between the corresponding
equivariant derived categories.

\subsection{Convolution in $\sD(G)$ and
$\sD_G(G)$}\label{ss:der-equiv-der-convolution} Let $G$ be an
algebraic group over an arbitrary field $k$, and $\mu:G\times_k
G\rar{}G$ the multiplication morphism. The bifunctor
\begin{equation}\label{e:convol-derived}
\sD(G) \times \sD(G) \rar{} \sD(G), \qquad (M,N)\longmapsto
M*N=\mu_!(M\boxtimes N),
\end{equation}
is called the \emph{convolution with compact supports}. Replacing
$\mu_!$ with $\mu_*$ in the above definition would yield the
``usual'' convolution bifunctor; however, convolution with compact
supports is the only one that will be used in this article, and will
be referred to simply as ``convolution'' of constructible
$\ell$-adic complexes on $G$.

\mbr

It is easy to construct an associativity constraint for the
bifunctor $*$, and check that it makes $\sD(G)$ a monoidal category,
where the unit object $\e$ is the delta-sheaf at the identity
element of $G$, i.e., $\e=1_*\ql=1_!\ql$.

\mbr

Lemma \ref{l:sheaves-functions-properties} implies that the bifunctor $*$ is compatible with convolution of
functions via the sheaves-to-functions correspondence. Namely, for a
finite group $\Ga$, let us define the convolution of two functions
$f_1,f_2:\Ga\rar{}\ql$ by the formula
$(f_1*f_2)(g)=\sum_{\ga\in\Ga}f_1(\ga)f_2(\ga^{-1}g)$. Then, for any
algebraic group $G$ over $\bF_q$ and any $M,N\in\sD(G)$, we have
$t_{M*N}=t_M*t_N$ as functions on $G(\bF_q)$.

\mbr

Next, suppose that $G$ is a unipotent algebraic group over $k$.
Unless otherwise explicitly stated, whenever we consider a
$G$-action on itself, we will always mean the \emph{conjugation}
action. We also have the induced action of $G$ on $G\times_k G$ (by
simultaneous conjugation), and the multiplication morphism
$\mu:G\times_k G\rar{}G$ is $G$-equivariant. It follows (see
\S\ref{ss:funct-equiv-derived}) that \eqref{e:convol-derived} can be
upgraded to a bifunctor
\begin{equation}\label{e:convol-equiv-derived}
\sD_G(G) \times \sD_G(G) \rar{} \sD_G(G), \qquad (M,N)\longmapsto
M*N=\mu_!(M\boxtimes N),
\end{equation}
which we also call convolution with compact supports.

\mbr

Just as in the non-equivariant case,
\eqref{e:convol-equiv-derived} can be upgraded to a monoidal
structure on the category $\sD_G(G)$. Moreover, this category has a
natural braiding, defined explicitly in \cite{intro}. We will only need a weaker assertion:

\begin{lem}\label{l:central-equivariant-derived}
There exist functorial isomorphisms
$\be_{M,N} : M*N\rar{\simeq}N*M$ for all $M,N\in\sD_G(G)$.
\end{lem}

\begin{proof}{Proof}
Consider the commutative diagram
\[
\xymatrix{
  G\times G \ar[d]_{\tau} \ar[r]^{\xi} &  G\times G \ar[d]^{\mu} \\
  G\times G \ar[r]^{\mu} & G,
   }
\]
where $\tau(g,h):=(h,g)$ and $\xi(g,h):=(g,g^{-1}hg)$. We have
$M*N=\mu_!(M\boxtimes N)$, and the above diagram shows that
$N*M=(\mu\tau )_!(M\boxtimes N)=\mu_!\xi_!(M\boxtimes N)$. We define
$\be_{M,N}:\mu_!(M\boxtimes N)\iso\mu_!\xi_!(M\boxtimes N)$ by
$\be_{M,N}:=\mu_!(f)$, where $f:M\boxtimes N\iso\xi_!(M\boxtimes N)$
comes from the $G$-equivariant structure on $N$.
\end{proof}

\subsection{Semigroupal categories}\label{ss:semigroupal}
The notion of a \emph{semigroupal} category is obtained from the
notion of a monoidal category by discarding all the axioms that
involve the unit object. Thus a semigroupal category is a triple
$(\cM,\tens,\al)$, where $\cM$ is a category,
$\tens:\cM\times\cM\rar{}\cM$ is a bifunctor, and $\al$ is an
associativity constraint for $\tens$, i.e., a collection of
trifunctorial isomorphisms $\al_{X,Y,Z}:(X\tens Y)\tens
Z\rar{\simeq}X\tens(Y\tens Z)$ for all triples of objects
$X,Y,Z\in\cM$ satisfying a standard coherence condition.

\mbr

The reason we need this notion is that even though the categories
$\sD(G)$ and $\sD_G(G)$ introduced in
\S\ref{ss:der-equiv-der-convolution} are monoidal, we will have the
occasion to consider certain semigroupal subcategories of $\sD_G(G)$
that, at least \emph{a priori}, may not be monoidal (cf.~Remark \ref{r:why-semigroupal}).

\mbr

The notion of a (weak, strong or strict) \emph{semigroupal functor}
between semigroupal categories is also obtained from the notion of a
monoidal functor in the obvious way. Thus if $(\cM,\tens,\al)$ and
$(\cN,\tens',\al')$ are semigroupal categories, a functor
$F:\cM\rar{}\cN$ is said to be \emph{strict semigroupal} if $F$
commutes with the semigroupal structures ``on the nose,'' i.e.,
$F(X\tens Y)=F(X)\tens' F(Y)$ for every pair of objects $X$, $Y$ of
$\cM$; $F(f\tens g)=F(f)\tens' F(g)$ for every pair of morphisms
$f$, $g$ in $\cM$; and $F(\al_{X,Y,Z})=\al'_{F(X),F(Y),F(Z)}$ for
every triple of objects $X,Y,Z$ of $\cM$.

\mbr

On the other hand, a \emph{weak semigroupal structure} on a functor
$\Phi:\cM\rar{}\cN$ is a collection of bifunctorial morphisms
$\vp_{X,Y}:\Phi(X)\tens'\Phi(Y)\rar{}\Phi(X\tens Y)$ for all
$X,Y\in\cM$, which are compatible with the associativity constraints
in the obvious sense. The structure is said to be \emph{strong} if
every $\vp_{X,Y}$ is an isomorphism.

\mbr

An \emph{additive semigroupal category} is a semigroupal category
$(\cM,\tens,\al)$ such that $\cM$ is an additive category, and the
bifunctor $\tens$ is bi-additive.

\subsection{Weak idempotents}\label{ss:weak-idemp}
Let $\cM=(\cM,\tens,\al)$ be a semigroupal category. An object
$e\in\cM$ is said to be a \emph{weak idempotent} if $e\tens e\cong
e$. Observe that this notion depends only on the bifunctor $\tens$
and not on the associativity constraint $\al$.

\mbr

If $e\in\cM$ is a weak idempotent, the corresponding \emph{Hecke
subcategory} is defined as the full subcategory $e\cM e\subset\cM$
consisting of all objects $N\in\cM$ such that $N\cong e\tens N\tens
e$. Equivalently, $e\cM e$ can be described as the essential image
of the functor $\cM\rar{}\cM$ given by $M\longmapsto (e\tens M)\tens
e$, which explains the notation.

\mbr

The Hecke subcategory $e\cM e\subset\cM$ is stable under
$\tens$, so it becomes a semigroupal category in its own right. If
$\cM$ is additive, so is $e\cM e$.
\begin{rem}\label{r:why-semigroupal}
Even if $\cM$ is a \emph{monoidal} category, one cannot expect $e\cM
e$ to be a monoidal category in general. Indeed, even though $e\tens
N\cong N\cong N\tens e$ for all $N\in e\cM e$, there is no guarantee
that the functor $N\longmapsto e\tens N$ is an auto-equivalence of
$e\cM e$; if it is not, then $e\cM e$ cannot have a unit
object\footnote{This is one of the reasons why we chose the term
``\emph{weak} idempotent''. The notion of a \emph{closed idempotent}
in a monoidal category, defined in \cite{intro}, is much more rigid;
in particular, if $e$ is any closed idempotent in a monoidal
category $\cM$, then $e\cM e$ is monoidal as
well.}.
\end{rem}

\begin{defin}
A semigroupal category $\cM$ is \emph{weakly symmetric} if $M\tens N\cong
N\tens M$ for all pairs of objects $M,N\in\cM$.
\end{defin}

For example, if a semigroupal category admits a braiding, then it is weakly symmetric. The category $(\sD_G(G),*)$ is weakly symmetric by virtue of Lemma \ref{l:central-equivariant-derived}. Observe that if $\cM$ is a weakly symmetric semigroupal category, then for any weak
idempotent $e\in\cM$, we have $e\cM e=e\cM=\cM e$.

\begin{defin}\label{d:minimal-idempotent}
If $\cM$ is an additive weakly symmetric semigroupal category, a
weak idempotent $e\in\cM$ is said to be \emph{minimal} if $e\neq 0$
and, for any weak idempotent $e'\in\cM$, we have either $e\tens
e'=0$, or $e\tens e'\cong e$.
\end{defin}

\begin{rem}
If $\cM$ is an additive weakly symmetric semigroupal category, a weak idempotent $e\in\cM$ is minimal if and only if the Hecke subcategory $e\cM$ contains exactly two weak idempotents $($up to isomorphism$)$, namely, $0$ and $e$.
\end{rem}

\subsection{Idempotents and $\bL$-packets}\label{ss:idemp-L-packets}
In this subsection we again assume that the base field is finite:
$k=\bF_q$. Let us fix a connected unipotent group $G$ over $\bF_q$.
Recall from Definition \ref{d:L-packets} that two irreducible
representations, $\rho_1$ and $\rho_2$, of $G(\bF_q)$ over $\ql$,
are said to be \emph{$\bL$-indistinguishable} if for every weak
idempotent $e\in\sD_G(G)$, the function $t_e$ acts in the same way
in $\rho_1$ and $\rho_2$. In this subsection we will explain that in
this definition one can restrict attention to a special class of
weak idempotents.

\begin{defin}\label{d:geom-min-idemp}
If $k$ is any field and $U$ is a unipotent group over $k$, a weak
idempotent $e\in\sD_U(U)$ is said to be \emph{geometrically minimal}
if for every algebraic extension $k'$ of $k$, the induced weak
idempotent $e'=e\tens_k k'$ in $\sD_{U'}(U')$, where $U'=U\tens_k
k'$, is minimal in the sense of Definition
\ref{d:minimal-idempotent}.
\end{defin}

Every geometrically minimal weak idempotent in $\sD_U(U)$
is minimal, but the converse need not be true. The next result is
proved in \S\ref{ss:proof-p:L-packets-min-idemp}.

\begin{prop}\label{p:L-packets-min-idemp}
Let $G$ be a connected unipotent group over $\bF_q$, and let
$\rho_1$, $\rho_2$ be two irreducible representations of
$G(\bF_q)$ over $\ql$. The following are equivalent.
\begin{enumerate}[$($i$)$]
\item The representations $\rho_1$ and $\rho_2$ are
$\bL$-indistinguishable.
\item For every minimal weak idempotent $e\in\sD_G(G)$, the function $t_e$
acts in the same way in $\rho_1$ and in $\rho_2$.
\item For every geometrically minimal weak idempotent $e\in\sD_G(G)$, the function $t_e$
acts in the same way in $\rho_1$ and in $\rho_2$.
\end{enumerate}
\end{prop}

\begin{rems}\label{r:idempotents-L-packets}
\begin{enumerate}[$($1$)$]
\item Note that if $e_1,e_2\in\sD_G(G)$ are non-isomorphic minimal
weak idempotents, then $e_1*e_2=0$, so $t_{e_1}*t_{e_2}=0$. Now if
$e\in\sD_G(G)$ is a geometrically minimal weak idempotent such that
$t_e\not\equiv 0$, we can define $L(e)$ as the set of irreducible
representations of $G(\bF_q)$ on which $t_e$ acts as the identity,
and it follows (from Definition \ref{d:L-packets} and the last
observation) that $L(e)$ is an $\bL$-packet. Proposition
\ref{p:L-packets-min-idemp} implies that, conversely, \emph{every}
$\bL$-packet of irreducible representations of $G(\bF_q)$ is of this
form. \sbr
\item It is shown in \cite{charsheaves} that if $e\in\sD_G(G)$ is any
geometrically minimal weak idempotent, then $t_e\not\equiv 0$. In particular, one obtains a bijection between
$\bL$-packets of irreducible representations of $G(\bF_q)$ and
isomorphism classes of geometrically minimal weak idempotents in
$\sD_G(G)$. However, the proofs of these facts use some of the methods developed in \cite{foundations} $($as well as additional techniques$)$, and are beyond the scope of the present article. \sbr
\item On the other hand, Proposition \ref{p:L-packets-min-idemp}
easily implies that if $e\in\sD_G(G)$ is a minimal weak idempotent
which is \emph{not geometrically minimal}, then $t_e\equiv 0$.
\end{enumerate}
\end{rems}

\subsection{Twists in the category
$\sD_G(G)$}\label{ss:twists-groups} A structure on the equivariant derived category $\sD_G(G)$ that plays an important role
in the proof of Theorem \ref{t:dim-reps-easy} is a canonical
automorphism of the identity functor, whose construction we now recall. Fix a unipotent group $G$
over $k$,
let $c:G\times G\to G$ be the conjugation action morphism
$c(g,h)=ghg^{-1}$, let $p_2:G\times G\to G$ denote the second
projection, and write $\De:G\to G\times G$ for the diagonal. Then
$c\circ\De=\id_G=p_2\circ\De$. For each $M\in\sD_G(G)$, the
$G$-equivariant structure on $M$ yields an isomorphism $c^*M\rar{\simeq} p_2^*M$. Pulling it back by $\De$, we obtain an isomorphism $\te_M:M=\De^*c^*M\rar{\simeq}\De^*p_2^*M=M$.

\begin{defin}\label{d:twists-group}
One calls $\te_M$ the \emph{twist automorphism} of $M$, or
the \emph{balancing isomorphism}. The collection
$\bigl\{\te_M \st M\in\sD_G(G)\bigr\}$ defines an automorphism of the
identity functor on $\sD_G(G)$, which we simply denote by $\te$
if no confusion can arise.
\end{defin}

The following fact will be used in our proof of Theorem
\ref{t:dim-reps-easy}.

\begin{lem}\label{l:triv-twists-easy}
Let $G$ be an easy unipotent group over a field $k$ of
characteristic $p>0$. For every object $M\in\sD_G(G)$, the twist
automorphism $\te_M$ of $M$ is trivial.
\end{lem}

\begin{proof}{Proof $($cf.~\cite{intro}$)$} We may and do assume
that $k$ is algebraically closed. Consider the usual (non-perverse)
$t$-structure on $\sD(G)$ whose heart is the category $Sh_c(G,\ql)$
of constructible $\ql$-sheaves on $G$. If $M\in\sD(G)$, we
will write $\cH^i(M)$ for the cohomology sheaves of $M$ with respect
to this $t$-structure ($i\in\bZ$).

\mbr

Fix $M\in\sD_G(G)$ and $i\in\bZ$. For every $x\in G(k)$, we have the
induced action of the centralizer $Z_G(x)$ of $x$ in $G$ on the stalk $\cH^i(M)_x$, and by continuity, the neutral component
$Z_G(x)^\circ\subset Z_G(x)$ acts trivially on $\cH^i(M)_x$. In
particular, since $G$ is easy, $x$
acts trivially on $\cH^i(M)_x$, which shows that
$\te_{\cH^i(M)}=\id_{\cH^i(M)}$. This implies that $\te_M$
is a \emph{unipotent} automorphism of $M$. On the other hand, since
$G$ is unipotent, it has exponent $p^n$ for some $n\in\bN$, and it
follows that $\bigl(\te_M\bigr)^{p^n}=\id_M$. Finally, we
conclude that $\te_M=\id_M$, as desired.
\end{proof}



\section{Induction with compact supports}\label{s:induction}

\subsection{Setup}\label{ss:setup-induction} Throughout this section
$G$ is a unipotent algebraic group over a field $k$, $G'\subset G$ is a
closed subgroup, and $\ell$ is a prime different from
$\operatorname{char}k$. Both $G$ and $G'$ are allowed to be
disconnected. We will use the notation introduced in Section
\ref{s:equivariant-derived} above. Our goal is to define the functor
of ``induction with compact supports''
\begin{equation}\label{e:ind}
\ind_{G'}^G : \sD_{G'}(G') \rar{} \sD_G(G)
\end{equation}
and establish its basic properties, such as the existence of
a weak semigroupal structure (cf.~\S\ref{ss:semigroupal}) on
this functor and its compatibility with twists (\S\ref{ss:ind-twists}).
In \S\ref{ss:ind-weak-idemp} we study the restriction of the functor
\eqref{e:ind} to a functor between suitable Hecke subcategories of
$\sD_{G'}(G')$ and of $\sD_G(G)$.

\subsection{Definition of
$\ind_{G'}^G$}\label{ss:def-ind} In this subsection we will define
the functor \eqref{e:ind}.

\subsubsection{Motivation}\label{sss:induction-motivation}
To motivate the definition of \eqref{e:ind}, we first rewrite the
formula for the induced character (in the setting of representations
of finite groups) in a suggestive way. Let $\Ga$ be a finite group,
and $\Ga'\subset\Ga$ a subgroup. Consider the (free) right action of
$\Ga'$ on the product $\Ga\times\Ga'$ given by
$(g,g')\cdot\ga=(g\ga,\ga^{-1}g'\ga)$, and let
$\widetilde{\Ga}=(\Ga\times\Ga')/\Ga'$ be the set of orbits for this
action. The left $\Ga$-action on $\Ga\times\Ga'$ given by
$\ga:(g,g')\longmapsto(\ga g,g')$ descends to a left $\Ga$-action on
$\widetilde{\Ga}$. We have a natural $\Ga'$-equivariant injection
$i:\Ga'\into\widetilde{\Ga}$ induced by $g'\longmapsto(1,g')$, and a
natural $\Ga$-equivariant map $\pi:\widetilde{\Ga}\rar{}\Ga$ induced
by $(g,g')\longmapsto gg'g^{-1}$ (as always, $\Ga'$ and $\Ga$ act on
themselves by conjugation).

\mbr

We will use the following notation. If $X$ is any set, $\Fun(X)$
denotes the vector space of all functions $X\rar{}\ql$. If $H$ is an
(abstract) group acting on $X$, we write $\Fun(X)^H\subset\Fun(X)$
for the subspace of $H$-invariant functions. If $\phi:X\rar{}Y$ is a
map of sets, we have the pullback map $\phi^*:\Fun(Y)\rar{}\Fun(X)$
given by $\phi^*(f)=f\circ\phi$. Finally, if $\phi$ has finite
fibers (in particular, if $X$ itself is finite), we can also define
a linear map $\phi_!:\Fun(X)\rar{}\Fun(Y)$ by the formula
\[
(\phi_! f)(y) = \sum_{x\in\phi^{-1}(y)} f(x).
\]

\mbr

With this notation, one can easily verify the following statements.
\begin{enumerate}[(a)]
\item The map $i^*:\Fun(\widetilde{\Ga})^\Ga\rar{}\Fun(\Ga')^{\Ga'}$
is an isomorphism. \sbr
\item Let $\iga:\Fun(\Ga')^{\Ga'}\rar{}\Fun(\Ga)^\Ga$ be defined by $\iga=\pi_!\circ(i^*)^{-1}$. If $\rho$ is any finite dimensional representation of $\Ga'$ over $\ql$ and $\chi\in\Fun(\Ga')^{\Ga'}$
is its character, then the character of the representation
$\Ind_{\Ga'}^\Ga\rho$ of $\Ga$ equals $\iga(\chi)$.
\end{enumerate}

\subsubsection{Auxiliary constructions}\label{sss:induction-auxiliary}
We will define the functor $\ig$ by imitating the formula presented
in \S\ref{sss:induction-motivation}. We have a free right action of
$G'$ on $G\times G'$, given by
$(g,g')\cdot\ga=(g\ga,\ga^{-1}g'\ga)$, so we can form the quotient
$\widetilde{G}=(G\times G')/G'$ (it exists as a scheme, for
instance, because $G\times G'$ is affine), and $G$ acts on $\Gt$ on
the left. Similarly, we can consider the $G'$-equivariant injection
$i:G'\into\widetilde{G}$ induced by $g'\longmapsto(1,g')$, and the
$G$-equivariant morphism $\pi:\widetilde{G}\rar{}G$ induced by
$(g,g')\longmapsto gg'g^{-1}$, where $G$ and $G'$ act on themselves
by conjugation. Applying the constructions of
\S\ref{ss:funct-equiv-derived}, we obtain functors between the
equivariant derived categories
\[
\sD_{G'}(G') \xleftarrow{\ \ i^*\ \ } \sD_G(\Gt) \xrightarrow{\ \
\pi_!\ \ } \sD_G(G).
\]
The geometric analogue of statement (a) in
\S\ref{sss:induction-motivation} is the following
\begin{lem}\label{l:pullback-equivalence}
The functor $i^*:\sD_G(\Gt)\rar{}\sD_{G'}(G')$ is an equivalence of
categories.
\end{lem}

This result is proved in \S\ref{sss:proof-l:pullback-equivalence}
below.

\subsubsection{The main definition}\label{sss:induction-definition} In the situation of
\S\ref{sss:induction-auxiliary}, let us choose a quasi-inverse to
the functor $i^*$, and denote it by $(i^*)^{-1}$, by a slight abuse
of notation.

\begin{defin}\label{d:induction}
The functor $\ind_{G'}^G:\sD_{G'}(G')\rar{}\sD_G(G)$ of
\emph{induction with compact supports} is defined as the composition
\[
\sD_{G'}(G') \xrar{\ \ (i^*)^{-1}\ \ } \sD_G(\Gt) \xrar{\ \ \pi_!\ \
} \sD_G(G).
\]
\end{defin}

\begin{rems}\label{r:induction}
\begin{enumerate}[$(1)$]
\item Strictly speaking, the definition we gave depends on the
choice of $(i^*)^{-1}$. However, different choices lead to
isomorphic functors $\ind_{G'}^G$, and since we only use induction
as a technical tool, we prefer to ignore this issue. \sbr
\item Along with the functor $\ind_{G'}^G$, one can introduce an
induction functor
\[
\Ind_{G'}^G=\pi_*\circ(i^*)^{-1}\ :\
\sD_{G'}(G')\rar{}\sD_G(G).
\]
We will only need this functor in the proof of Proposition \ref{p:induction-twists}.
\end{enumerate}
\end{rems}

\subsubsection{Proof of Lemma
\ref{l:pullback-equivalence}}\label{sss:proof-l:pullback-equivalence}
The result would have been more or less obvious, had we used the
definition of equivariant derived categories in terms of quotient
stacks (cf.~\S\ref{ss:equivariant-derived}). Indeed, the map
$g'\longmapsto(1,g')$ induces an isomorphism
$G'\rar{\simeq}G\bigl\backslash(G\times G')$, and hence an
isomorphism $G'\bigl/(\Ad G')\rar{\simeq}G\bigl\backslash(G\times
G')\bigl/G'=G\bigl\backslash\Gt$.
However, since we used an \emph{ad hoc} definition of the
equivariant derived category, we will give a proof of Lemma
\ref{l:pullback-equivalence} that only uses that definition. The
argument is based on two results:

\begin{lem}\label{l:free-quotient-pullback-equivalence}
Let $U$ be a unipotent group over $k$, let $N\subset U$ be a closed
normal subgroup, let $X$ be a scheme of finite type over $k$ with a
left $U$-action, and let $\phi:X\rar{}Y$ be a morphism of
$k$-schemes which makes $X$ an $N$-torsor over $Y$ $($a fortiori,
the induced action of $N$ on $X$ is free$)$. The pullback functor
$\phi^*$ can be upgraded to an equivalence of categories
\[
\phi^* : \sD_{U/N}(Y) \rar{} \sD_U(X).
\]
\end{lem}

\begin{lem}\label{l:free-quotient-pullback-quasi-inverse}
With the notation of Lemma
\ref{l:free-quotient-pullback-equivalence}, assume, in addition,
that $N$ admits a complement in $U$ $($i.e., a closed subgroup
$H\subset U$ which maps isomorphically onto $U/N${}$)$, and that
$\phi$ admits an $H$-equivariant section $\sg:Y\rar{}X$. Then the
functor
\[
\sg^* : \sD_U(X) \rar{} \sD_{U/N}(Y)=\sD_H(Y),
\]
understood as the composition of the forgetful functor
$\sD_U(X)\rar{}\sD_H(X)$ and the pullback via $\sg$, is a
quasi-inverse to the functor $\phi^*:\sD_{U/N}(Y)\rar{}\sD_U(X)$.
\end{lem}

Let us now prove Lemma \ref{l:pullback-equivalence}. Recall that
$\Gt$ is defined as the quotient of $G\times G'$ by the right
$G'$-action defined by $(g,g')\cdot\ga = (g\ga,\ga^{-1}g'\ga)$.
Since we will also need to consider left $G$-actions, we prefer to
turn this action into a left $G'$-action as well, given by
$\ga:(g,g')\longmapsto(g\ga^{-1},\ga g'\ga^{-1})$ (merely for
notational convenience).

\mbr

 Let us write $\fq:G\times G'\rar{}\Gt$ for the quotient
morphism. We define a left action of $G\times G'$ on $G\times G'$ by
$(h,\ga):(g,g')\longmapsto(hg\ga^{-1},\ga g'\ga^{-1})$. Applying
Lemma \ref{l:free-quotient-pullback-equivalence} to $U=G\times G'$
and $N=\{1\}\times G'\subset U$, we obtain an equivalence
\[
\fq^* : \sD_G(\Gt) \rar{\sim} \sD_{G\times G'}(G\times G').
\]

\mbr

 On the other hand, let $\fp':G\times G'\rar{}G'$ denote the
second projection. It can be viewed as a quotient map for the
induced action of $G$ on $G\times G'$, where we embed $G\into
G\times G'$ via $g\longmapsto(g,1)$. Of course, the quotient group
$(G\times G')/(G\times\{1\})$ is naturally identified with $G'$. Let
$\De:G'\into G\times G'$ denote the diagonal embedding. Then
$\De(G')$ is a complement to $G\times\{1\}$ in $G\times G'$.
Moreover, the map $j:G'\rar{}G\times G'$ defined by
$g'\longmapsto(1,g')$ is a $\De(G')$-equivariant section of $\fp'$.

\mbr

 Applying Lemma \ref{l:free-quotient-pullback-quasi-inverse}, we
see that the functors
\[
\fp^{\prime\ast}:\sD_{G'}(G')\rar{}\sD_{G\times G'}(G\times G')
\qquad\text{and}\qquad j^*:\sD_{G\times G'}(G\times
G')\rar{}\sD_{G'}(G')
\]
are equivalences of categories that are quasi-inverse to each other.

\mbr

 Finally, since the composition of $\De:G'\into G\times G'$ and
the natural projection $G\times G'\rar{}G$ is equal to the inclusion
map $G'\rar{}G$, and since $i=\fq\circ j$ by definition, we see that
the functor $i^*:\sD_G(\Gt)\rar{}\sD_{G'}(G')$ is isomorphic to the
composition
\[
\sD_G(\Gt) \xrar{\ \ \fq^*\ \ } \sD_{G\times G'}(G\times G') \xrar{\
\ j^*\ \ } \sD_{G'}(G').
\]
We just showed that $\fq^*$ and $j^*$ are equivalences, whence so is
$i^*$. \qed

\subsection{An alternative viewpoint on induction
functors}\label{ss:induction-alternative} Suppose
$Y$ is a scheme of finite type over $k$ equipped with a \emph{transitive} left $G$-action. We
let $G$ act on itself by conjugation, as usual, and consider the
induced diagonal action of $G$ on $G\times Y$. Write
$Z=\bigl\{(g,y)\st g\cdot y=y\bigr\}\subset G\times Y$ and observe
that $Z$ is $G$-stable. Given $y\in Y(k)$, we let $G^y$ denote the
stabilizer of $y$ in $G$, and consider the inclusion morphism
$j_y:G^y\into Z$ given by $g\mapsto(g,y)$.

\begin{prop}\label{p:induction-alternative}
\begin{enumerate}[$($a$)$]
\item For every $y\in Y(k)$, the pullback
$j_y^*:\sD_G(Z)\rar{}\sD_{G^y}(G^y)$ is an equivalence of
categories.
 \sbr
\item If $\pr_1:Z\rar{}G$ is the first projection, the functors
\[
\pr_{1*}\circ (j_y^*)^{-1}, \quad \pr_{1!}\circ (j_y^*)^{-1} \quad :
\quad \sD_{G^y}(G^y) \rar{} \sD_G(G)
\]
are isomorphic to the induction functors $\Ind_{G^y}^G$ and
$\ind_{G^y}^G$, respectively.
\end{enumerate}
\end{prop}

\begin{proof}{Proof}
Fix $y\in Y(k)$, write $G'=G^y$, and let $\widetilde{G}=(G\times G')/G'$ be defined as before. The morphism $G\times G'\rar{}G\times Y$ given by $(g,g')\mapsto(gg'g^{-1},g\cdot y)$ has image in $Z$ and induces a $G$-equivariant isomorphism $\widetilde{G}\rar{\simeq}Z$, which identifies $j_y$ with $i:G'\into\widetilde{G}$ and $\pr_1$ with $\pi:\widetilde{G}\rar{}G$. The proposition follows.
\end{proof}

\subsection{Useful notation}\label{ss:induction-useful-notation}
In this subsection we collect some of the notation introduced in
\S\ref{sss:induction-auxiliary} and in the proof of Lemma
\ref{l:pullback-equivalence}, given in
\S\ref{sss:proof-l:pullback-equivalence} above. It is convenient to
put all the maps we defined together into the following diagram:
\[
\xymatrix{
  &  G\times G' \ar@/_/[ddl]_{\fp'} \ar[dd]^{\fq} \ar[ddr]^{\fc} & \\
  & & \\
 G' \ar@/_/[uur]_j \ar[r]_i & \Gt \ar[r]_\pi & G
   }
\]
Here, $\fq$ is the quotient map for the $G'$-action, $\fp'$ is the
second projection, $j$ is the natural inclusion given by
$j(g')=(1,g')$, and $i=\fq\circ j$. Also, $\fc$ is the conjugation
map $(g,g')\longmapsto gg'g^{-1}$, and $\pi$ is the unique morphism
satisfying $\fc=\pi\circ\fq$. Finally, let us agree, from now on, to
denote the chosen quasi-inverse of the functor $i^*$ by
\[
\sD_{G'}(G') \ni M \longmapsto \Mt \in\sD_G(\Gt) .
\]

\subsection{Weak semigroupal structure on $\ind_{G'}^G$}\label{ss:weak-semigroupal-ind} In this
subsection we will define functorial morphisms
\begin{equation}\label{e:weak-semigr}
\vp_{M,N} : \ind_{G'}^G(M)*\ind_{G'}^G(N) \rar{} \ind_{G'}^G(M*N)
\end{equation}
for all $M,N\in\sD_{G'}(G')$, where the convolution on the left
(respectively, on the right) is computed on $G$ (respectively, on
$G'$). In \S\ref{ss:when-isom} we prove that under a suitable
condition on $M$ and $N$, the arrow $\vp_{M,N}$ is an isomorphism.
Of course, there is no reason for it to be an isomorphism in general
(in the setting of finite groups, induction of class functions
usually does not commute with convolution).

\subsubsection{Preparations}\label{sss:weak-semigr-preparations}
We keep the notation of \S\ref{ss:induction-useful-notation}. We
have an obvious morphism $\Gt\rar{}G/G'$ induced by the first
projection $G\times G'\rar{}G$. Form the fiber product
\[Z=\Gt\underset{G/G'}{\times}\Gt.\] Thus $Z$ is a closed subscheme of
$\Gt\times\Gt$, and the morphism $i\times i:G'\times
G'\rar{}\Gt\times\Gt$ factors through $Z$.
Further, let $\mu:G\times G\rar{}G$ and $\mu':G'\times G'\rar{}G'$
denote the respective multiplication morphisms. The next result is straightforward.
\begin{lem}\label{l:Z-group-scheme}
There exists a morphism $\mt:Z\rar{}\Gt$ such that
\[
\mt(y_1,y_2) = \bigl[ (g_1,h_1\cdot g_1^{-1}g_2\cdot h_2\cdot
g_2^{-1}g_1) \bigr] \qquad \forall\,(y_1,y_2)\in
Z\subset\Gt\times\Gt,
\]
where $(g_j,h_j)\in G\times G'$ are representatives of the
$G'$-orbits $y_j\in\Gt$ $(j=1,2)$, and $[(g,h)]$ denotes the
$G'$-orbit of a point $(g,h)\in G\times G'$. Furthermore, the square
\[
\xymatrix{
  G'\times G' \ar@{^{(}->}[rr]^{\ \ i\times i}\ar[d]_{\mu'} & & Z
  \ar[d]^{\mt} \\
  G' \ar@{^{(}->}[rr]^{i} & & \Gt
   }
\]
commutes and is cartesian, and the square
\[
\xymatrix{
  Z \ar[rr]^{\pi\times\pi\ \ \ } \ar[d]_{\mt} & & G\times G \ar[d]^\mu \\
  \Gt \ar[rr]^\pi & & G
   }
\]
commutes. Also, $Z$ is stable under the diagonal action of $G$,
and $\mt$ is $G$-equivariant.
\end{lem}

\subsubsection{Definition of the weak semigroupal
structure}\label{sss:weak-semigr-construction} Let us choose
$M,N\in\sD_{G'}(G')$. By definition, $(i^*\times
i^*)(\Mt\boxtimes\Nt)\cong M\boxtimes N$, whence, by the proper base
change theorem, we have functorial isomorphisms
\[
i^*\bigl(\mt_!(\Mt\boxtimes\Nt)\bigl\lvert_Z\bigr) \cong
\mu'_!(M\boxtimes N)\cong M*N,
\]
and thus we have functorial isomorphisms
\[
\ig(M*N) \cong \pi_!\mt_!(\Mt\boxtimes\Nt)\bigl\lvert_Z \cong \mu_!
(\pi\times\pi)_! (\Mt\boxtimes\Nt)\bigl\lvert_Z.
\]
If $f:Z\into\Gt\times\Gt$ denotes the inclusion morphism, we see
that the adjunction morphism
$\Mt\boxtimes\Nt\rar{}f_!f^*(\Mt\boxtimes\Nt)$ induces a natural
morphism
\[
(\ig M)*(\ig N) \cong\mu_!(\pi\times\pi)_!(M\boxtimes N) \rar{}
\ig(M*N).
\]
This is the desired morphism \eqref{e:weak-semigr}.

\subsection{Some auxiliary results} The following facts will be used several times in the rest of the section.

\begin{lem}\label{l:distinguished-triangle}
Let $X$ be a scheme of finite type over $k$, let $U\subset X$ be an
open subset, let $Z=X\setminus U$ be equipped with the reduced
induced subscheme structure, and let
\[
\xymatrix{
 U \ar@{^{(}->}[rr]^j & & X & & Z \ar@{_{(}->}[ll]_i
   }
\]
be the natural inclusions. For every $\cF\in\sD(X)$, there is a
distinguished triangle
\[
j_! j^! \cF \rar{} \cF \rar{} i_* i^* \cF \rar{} j_! j^! \cF[1],
\]
functorial in $\cF$, where the morphisms $j_!j^! \cF\rar{}\cF$ and
$\cF\rar{}i_* i^*\cF$ are induced via adjunction by the identity
morphisms $j^!\cF\rar{}j^!\cF$ and $i^*\cF\rar{}i^*\cF$.
\end{lem}

Let us write, as usual, $\cF\bigl\lvert_U=j^*\cF$ and
$\cF\bigl\lvert_Z=i^*\cF$. Since $j^!=j^*$ and $i_*=i_!$, the
distinguished triangle of Lemma \ref{l:distinguished-triangle} can
also be rewritten as
\begin{equation}\label{e:distinguished-triangle}
j_! \bigl(\cF\bigl\lvert_U\bigr) \rar{} \cF \rar{} i_!
\bigl(\cF\bigl\lvert_Z\bigr) \rar{} j_!
\bigl(\cF\bigl\lvert_U\bigr)[1].
\end{equation}

\begin{lem}\label{l:push-pull-summand}
Let $H$ be a possibly disconnected unipotent group over a field $k$,
and let $h:X\rar{}Y$ be an $H$-torsor, where $Y$ is a scheme of
finite type over $k$. For every $M\in\sD(Y)$, consider the canonical
adjunction morphism $\eps_M:h_!h^!M\rar{}M$. \sbr
\begin{enumerate}[$($a$)$]
\item If $H$ is connected, then $\eps_M$ is an isomorphism for all $M\in\sD(Y)$. \sbr
\item In general, $\eps_M$ has a natural splitting, i.e., there exist functorial
morphisms $s_M:M\rar{}h_!h^!M$ for all $M\in\sD(Y)$, such that
$\eps_M\circ s_M=\id_M$.
\end{enumerate}
\end{lem}
\begin{proof}{Proof}
(a) First, we may clearly assume that $k$ is algebraically closed
(using base change to an algebraic closure of $k$). Second, using
base change by the smooth surjective morphism $X\rar{h}Y$, we may
assume that $X$ is a trivial $H$-torsor over $Y$. Since $k$ is
algebraically closed and $H$ is a connected unipotent group over
$k$, it follows that $H$ has a filtration by normal connected
subgroups with successive subquotients isomorphic to the additive
group $\bG_a$ over $k$. Thus we may also assume that $H=\bG_a$. In
this case, the result reduces to the standard computation of the
cohomology with compact supports for an affine line over $k$.

\mbr

(b) In view of (a), we may assume that the neutral connected
component of $H$ is trivial, i.e., that $H$ is a finite \'etale
group scheme over $k$. In this case $h_*=h_!$, $h^!=h^*$, so that
$h^!$ is also left adjoint to $h_!$, and we obtain a canonical
adjunction morphism $\eta_M:M\rar{}h_!h^!M$. Let $\kbar$ be an
algebraic closure of $k$, and let $n=\abs{H(\kbar)}$. We claim that
the composition $\eps_M\circ\eta_M$ equals the multiplication by
$n$; assuming this claim, we can define $s_M=n^{-1}\cdot\eta_M$,
because $\ql$ has characteristic $0$, and then $s_M$ is the desired
splitting.
To prove the last claim, we again extend the base field to $\kbar$
and thus assume that $k=\kbar$. Then $H$ is a discrete group of
order $n$, and the claim becomes trivial.
\end{proof}

In practice, we will apply the following corollary of the lemma.
Note that $h$ is a smooth morphism of relative dimension $d=\dim H$,
so that there is a natural isomorphism of functors $h^!\cong
h^*[2d](d)$. Thus the next result is immediate.

\begin{cor}\label{c:push-pull-summand}
In the situation of Lemma \ref{l:push-pull-summand}, let $d=\dim H$.
If $H$ is connected $($respectively, in general$)$, every
$M\in\sD(Y)$ is naturally isomorphic to $($respectively, is naturally isomorphic to a direct
summand of$)$ $h_!h^*M[2d](d)$.
\end{cor}

\subsection{A case where \eqref{e:weak-semigr} is an
isomorphism}\label{ss:when-isom} In this subsection we establish a
sufficient condition for \eqref{e:weak-semigr} to be an isomorphism.
First we introduce more notation. If $M$ is an object of $\sD(G')$
or $\sD_{G'}(G')$, we will denote by $\Mb$ the object of $\sD(G)$
obtained from $M$ by extension by zero. It is clear that if
$M,N\in\sD(G')$, then
\begin{equation}\label{e:convol-ext-by-0}
 \overline{M}*\overline{N}\cong\overline{M*N}, \qquad
 \text{functorially in } M,N.
\end{equation}

\mbr

Next, we choose an algebraic closure, $\kbar$, of $k$. We can
consider the algebraic group $G\tens_k\kbar$ over $\kbar$. By a
slight abuse of notation, given an object of $\sD(G)$, we will
denote the corresponding object of $\sD(G\tens_k\kbar)$ by the same
letter. If $x\in G(\kbar)$, we denote by $\de_x$ the corresponding
delta-sheaf on $G\tens_k\kbar$.

\begin{prop}\label{p:induction-convolution}
If $M,N\in\sD_{G'}(G')$ are such that $\Mb*\de_x*\Nb=0$, as objects
of $\sD(G\tens_k\kbar)$, for all $x\in G(\kbar)\setminus G'(\kbar)$,
then \eqref{e:weak-semigr} is an isomorphism.
\end{prop}

\begin{proof}{Proof}
Using base change from $k$ to $\kbar$, we may and do assume that $k$
is algebraically closed. Let $U=(\Gt\times\Gt)\setminus Z$, an open
subset of $\Gt\times\Gt$. We will use the following shorthand
notation: $(\Mt\boxtimes\Nt)_U$ is the extension of
$(\Mt\boxtimes\Nt)\bigl\lvert_U$ to $\Gt\times\Gt$ by zero outside
of $U$. Applying the distinguished triangle
\eqref{e:distinguished-triangle} to our situation (where
$X=\Gt\times\Gt$), we see that it is enough to check that
$\mu_!(\pi\times\pi)_!(\Mt\boxtimes\Nt)_U=0$.

\mbr

We still keep the notation of \S\ref{ss:induction-useful-notation}.
Consider the morphism
\[
\fq\times \fq:G\times G'\times G\times G'\rar{}\Gt\times\Gt,
\]
which is a torsor under $G'\times G'$. According to Corollary
\ref{c:push-pull-summand}, if $d'=\dim G'$, then
$(\Mt\boxtimes\Nt)_U$ is a direct summand of $(\fq\times
\fq)_!(\fq\times \fq)^*(\Mt\boxtimes\Nt)_U[4d'](2d')$, so to
complete the proof of the proposition, it will suffice to show that
\[
\mu_!(\pi\times\pi)_!(\fq\times \fq)_!(\fq\times
\fq)^*(\Mt\boxtimes\Nt)_U=0.
\]

\mbr

We have $(\fq\times \fq)^*(\Mt\boxtimes\Nt)_U=(\fq^*\Mt\boxtimes
\fq^*\Nt)_{(\fq\times \fq)^{-1}(U)}$, where the meaning of the
subscript $(\fq\times \fq)^{-1}(U)$ is similar to that of the
subscript $U$. According to the proof of Lemma
\ref{l:pullback-equivalence} given in
\S\ref{sss:proof-l:pullback-equivalence}, we have $\fq^*\Mt\cong
\fp^{\prime\ast}M$, where $\fp':G\times G'\rar{}G'$ is the
projection onto the second factor. Similarly, $\fq^*\Nt\cong
\fp^{\prime\ast}N$. Thus we are reduced to showing that
\begin{equation}\label{e:enough}
\bigl[ \mu\circ(\pi\times\pi)\circ(\fq\times \fq) \bigr]_! \bigl(
\fp^{\prime\ast}M\boxtimes \fp^{\prime\ast}N \bigr)_{(\fq\times
\fq)^{-1}(U)}=0.
\end{equation}

\mbr

 In fact, we will prove a stronger statement. Namely, consider
the morphism
\[
\Phi = \bigl[ \mu\circ(\pi\times\pi)\circ(\fq\times \fq) \bigr]
\times \pr_1 \times \xi \ :\  G\times G'\times G\times G' \rar{}
G\times G\times G,
\]
where $\pr_1:G\times G'\times G\times G'\rar{}G$ is the first
projection and $\xi:G\times G'\times G\times G'\rar{}G$ is given by
$(g_1,g'_1,g_2,g'_2)\longmapsto g_1^{-1}g_2$. We will prove that
\begin{equation}\label{e:enough-stronger}
\Phi_! \bigl( \fp^{\prime\ast}M\boxtimes \fp^{\prime\ast}N
\bigr)_{(\fq\times \fq)^{-1}(U)}=0,
\end{equation}
which will of course imply \eqref{e:enough}.

\mbr

 By definition, it is easy to check that $(\fq\times
\fq)^{-1}(U)=\Phi^{-1}\bigl(G\times G\times (G\setminus G')\bigr)$.
Hence it suffices to prove the vanishing of the stalk of
$\Phi_!(\fp^{\prime\ast}M\boxtimes \fp^{\prime\ast}N)$ at any given
point $(g,g_1,x)\in G(k)\times G(k)\times\bigl(G(k)\setminus
G'(k)\bigr)$. As usual, we apply the proper base change theorem. The
fiber $\Phi^{-1}(g,g_1,x)$ is naturally identified with the closed
subset
\[
\bigl\{ (h_1,h_2)\in G'\times G' \st h_1 x h_2 = g_1^{-1}gg_1x
\bigr\} \subset G'\times G'
\]
via the morphism $(h_1,h_2)\longmapsto (g_1,h_1,g_1 x,h_2)$. (This
is simply because the equation $g_1 h_1 g_1^{-1} \cdot (g_1
x)h_2(g_1 x)^{-1}=g$ is equivalent to $h_1 x h_2 = g_1^{-1}gg_1x$.)
Hence the stalk of $\Phi_!(\fp^{\prime\ast}M\boxtimes
\fp^{\prime\ast}N)$ at $(g,g_1,x)$ is quasi-isomorphic to the stalk
of the convolution $\Mb*\de_x*\Nb$ at $g_1^{-1}gg_1x$. Since $x\in
G(k)\setminus G'(k)$, we have $\Mb*\de_x*\Nb=0$ by assumption, which
implies \eqref{e:enough-stronger} and completes the proof of
Proposition \ref{p:induction-convolution}.
\end{proof}

\subsection{Induction of weak idempotents}\label{ss:ind-weak-idemp}
In this subsection we keep all the notation introduced earlier
(notably, at the beginning of \S\ref{ss:when-isom}).

\subsubsection{Statement of the main theorem}
Let $e\in\sD_{G'}(G')$ be a weak idempotent, and recall that the
corresponding Hecke subcategory of $\sD_{G'}(G')$ is denoted by
$e\sD_{G'}(G')$, because $\sD_{G'}(G')$ is weakly symmetric by Lemma \ref{l:central-equivariant-derived} (see \S\ref{ss:weak-idemp} for the all relevant
terminology).

\mbr

Assume that $\overline{e}*\de_x*\overline{e}=0$ for all $x\in
G(\kbar)\setminus G'(\kbar)$. It follows from Proposition
\ref{p:induction-convolution} that $f:=\ig(e)$ is a weak idempotent
in $\sD_G(G)$. Moreover, if $M\in e\sD_{G'}(G')$, then $M\cong e*M$,
which implies that $\overline{e}*\de_x*\overline{M}=0$ for all $x\in
G(\kbar)\setminus G'(\kbar)$ (in view of \eqref{e:convol-ext-by-0}),
and Proposition \ref{p:induction-convolution} shows that $\ig(M)\in
f\sD_G(G)$.

\begin{thm}\label{t:induction-Hecke}
\begin{enumerate}[$($a$)$]
\item In this situation, the functor
\begin{equation}\label{e:induction-Hecke}
\ig\Bigl\lvert_{e\sD_{G'}(G')} \,:\, e\sD_{G'}(G') \rar{} f\sD_G(G)
\end{equation}
is strong semigroupal $($with respect to the semigroupal structure introduced in \S\ref{ss:weak-semigroupal-ind}$)$ and induces a bijection on isomorphism classes of objects. \sbr
\item If the functor $M\longmapsto e*M$ is isomorphic to the
identity functor on $e\sD_{G'}(G')$, the functor
\eqref{e:induction-Hecke} is faithful. \sbr
\item If the functors $M\longmapsto e*M$ and $N\longmapsto f*N$ are
isomorphic to the identity functors on $e\sD_{G'}(G')$ and
$f\sD_G(G)$, respectively, then \eqref{e:induction-Hecke} is an
equivalence of categories, a quasi-inverse to which is provided by
the functor
\[
f\sD_G(G)\rar{}e\sD_{G'}(G'), \qquad N \longmapsto
e*\bigl(N\bigl\lvert_{G'}\bigr).
\]
\end{enumerate}
\end{thm}

\subsubsection{An immediate consequence} Let us note at once the following

\begin{cor}\label{c:induction-min-idemp}
If $e$ is a minimal $($resp., geometrically minimal$)$ weak
idempotent in $\sD_{G'}(G')$ and the other assumptions are in force,
then $f=\ig(e)$ is a minimal $($resp., geometrically
minimal$)$ weak idempotent in $\sD_G(G)$.
\end{cor}

It suffices to check that if $e$ is minimal, then so is $f$, as all
the hypotheses of the corollary are obviously invariant under base
change to algebraic extensions of $k$. To this end, observe that a
weak idempotent $f\in\sD_G(G)$ is minimal if and only if the
semigroupal category $f\sD_G(G)$ contains no weak idempotents other
than $0$ and $f$. By Proposition \ref{p:induction-convolution} and
Theorem \ref{t:induction-Hecke}, the functor
\eqref{e:induction-Hecke} induces a bijection between the set of
isomorphism classes of weak idempotents in $e\sD_{G'}(G')$ and the
set of isomorphism classes of weak idempotents in $f\sD_G(G)$,
whence the claim.

\subsubsection{Reduction of Theorem \ref{t:induction-Hecke} to two auxiliary propositions}
From now on we fix a weak idempotent $e\in\sD_{G'}(G')$ satisfying
$\overline{e}*\de_x*\overline{e}=0$ for all $x\in G(\kbar)\setminus
G'(\kbar)$. If $M,N\in e\sD_{G'}(G')$, then $M\cong M*e$ and $N\cong e*N$, which implies that $\overline{M}*\de_x*\overline{N}$ for all $x\in G(\kbar)\setminus G'(\kbar)$. Thus the first assertion of Theorem \ref{t:induction-Hecke}(a) results from Proposition \ref{p:induction-convolution}.

\mbr

Put $f=\ig(e)\in\sD_G(G)$. The following two propositions are proved below.

\begin{prop}\label{p:restrict-induce}
For each $N\in\sD_G(G)$, there is an isomorphism
\[f*N\xrar{\ \ \ \simeq\ \ \ }\ig\bigl(e*(N\bigl\lvert_{G'})\bigr),\] functorial with
respect to $N$.
\end{prop}

\begin{prop}\label{p:induce-restrict}
For each $M\in e\sD_{G'}(G')$, there is an isomorphism
\[
e * \bigl( (\ig M)\bigl\lvert_{G'} \bigr) \xrar{\ \ \ \simeq\ \ \ }
e * M,
\]
functorial with respect to $M$.
\end{prop}

These propositions clearly imply part (c) and the second assertion of part (a) of Theorem
\ref{t:induction-Hecke} (by restricting attention to the objects
$N\in f\sD_G(G)$). To see that they also imply part (b) of the
theorem, observe that by Proposition \ref{p:induce-restrict}, the
functor $M\longmapsto e*M$ on $e\sD_{G'}(G')$ is isomorphic to the
composition
\[
e\sD_{G'}(G') \xrar{\ \ \ig\ \ } \sD_G(G) \xrar{\text{restriction}}
\sD_{G'}(G') \xrar{\ \ e*-\ \ } e\sD_{G'}(G').
\]
If the composition is isomorphic to the identity functor on
$e\sD_{G'}(G')$, then the first term in the composition,
$\ig\bigl\lvert_{e\sD_{G'}(G')}$, has to be faithful.

\subsubsection{Proof of Proposition \ref{p:restrict-induce}}
The argument follows a pattern similar to the one used in the proof
of Proposition \ref{p:induction-convolution}. By definition, we have
\begin{equation}\label{e:f*N}
f*N = (\ig e)*N = \mu_!(\pi\times\id)_!(\et\boxtimes N),
\end{equation}
where we are using the morphisms
\[
\Gt\times G\xrar{\ \ \pi\times\id\ \ } G\times G \xrar{\ \ \mu\ \ }
G.
\]

\mbr

 Consider the closed subset
\[
Z' = \Bigl\{ \bigl( [(g,h)],\ga \bigr) \in \Gt\times G
\,\Bigl\lvert\, g^{-1}\ga g\in G' \Bigr\} \subset \Gt\times G.
\]
It is easy to check that $Z'$ is well defined and is stable under
the diagonal action of $G$ on $\Gt\times G$ (where, as always, $G$
acts on itself by conjugation). Moreover, let us recall the closed
subset $Z\subset\Gt\times\Gt$ introduced in
\S\ref{sss:weak-semigr-preparations}:
\[
Z = \Bigl\{ \bigl( [(g_1,h_1)], [(g_2,h_2)] \bigr) \in \Gt\times \Gt
\,\Bigl\lvert\, g_1^{-1}g_2\in G' \Bigr\}.
\]
It is clear that the morphism
$\id\times\pi:\Gt\times\Gt\rar{}\Gt\times G$ takes $Z$ into $Z'$.

\mbr

 Now consider the (clearly $G$-equivariant) morphism
\[
\nu : Z' \rar{} \Gt, \qquad \bigl( [(g,h)],\ga \bigr) \longmapsto
\bigl[ (g,hg^{-1}\ga g) \bigr].
\]
It is easy to check that $\nu$ is well defined, and we obtain a
commutative diagram
\[
\xymatrix{
 G' \times G' \ar[d]_{\mu'} \ar@{^{(}->}[rr]^{\ \ i\times i} & & Z
 \ar[d]^{\mt}
 \ar[rr]^{\id\times\pi} & & Z' \ar[dll]_{\nu} \\
 G' \ar@{^{(}->}[rr]^{i} & & \Gt & &
   }
\]
Furthermore, it is also easy to check that the following square is
cartesian:
\[
\xymatrix{
 G' \times G' \ar[d]_{\mu'} \ar@{^{(}->}[rr]^{\ \ i\times(\pi\circ i)} & & Z'
 \ar[d]^{\nu}
  \\
 G' \ar@{^{(}->}[rr]^{i} & & \Gt
   }
\]

\mbr

 Now we use the same argument as before. Put $U'=(\Gt\times
G)\setminus Z'$, write $(\et\boxtimes N)_{Z'}$ for the extension of
$(\et\boxtimes N)\bigl\lvert_{Z'}$ to $\Gt\times G$ by zero outside
of $Z'$, and define $(\et\boxtimes N)_{U'}$ similarly. Applying the
proper base change theorem to the cartesian square above, we obtain
functorial isomorphisms
\[
i^*\nu_!\bigl( (\et\boxtimes N)\bigl\lvert_{Z'} \bigr) \cong
\mu'_!\bigl(e\boxtimes (N\bigl\lvert_{G'})\bigr) \cong
e*(N\bigl\lvert_{G'}),
\]
whence $\ig\bigl(e*(N\bigl\lvert_{G'})\bigr)\cong\pi_!\nu_!\bigl(
(\et\boxtimes N)\bigl\lvert_{Z'} \bigr)$. We also have a commutative
diagram
\[
\xymatrix{
  Z' \ar[rr]^{\pi\times\id\ \ \ } \ar[d]_{\nu} & & G\times G \ar[d]^\mu \\
  \Gt \ar[rr]^\pi & & G
   }
\]
which implies that
\[
\ig\bigl(e*(N\bigl\lvert_{G'})\bigr) \cong \mu_!(\pi\times\id)_!
\bigl( (\et\boxtimes N)_{Z'} \bigr).
\]

\mbr

 In view of \eqref{e:f*N}, we see that the adjunction morphism
$(\et\boxtimes N)\rar{}(\et\boxtimes N)_{Z'}$ yields a morphism
$f*N\rar{}\ig\bigl(e*(N\bigl\lvert_{G'})\bigr)$, functorial with
respect to $N$. As before, to complete the proof of Proposition
\ref{p:restrict-induce}, it is enough to show that if
$\overline{e}*\de_x*\overline{e}=0$ for all $x\in G(\kbar)\setminus
G'(\kbar)$, then $\mu_!(\pi\times\id)_! \bigl( (\et\boxtimes N)_{U'}
\bigr)=0$.

\mbr

Henceforth, we may and do assume that $k$ is algebraically closed.
By Corollary \ref{c:push-pull-summand}, it is enough to prove that
$\mu_!(\pi\times\id)_!(\fq\times\id)_! \bigl[
(\fq\times\id)^*(\widetilde{e}\boxtimes N)_{U'} \bigr] = 0$, which
is equivalent to
\begin{equation}\label{e:enough-2}
\mu_!(\pi\times\id)_!(\fq\times\id)_! \Bigl[ \bigl(
(\fp^{\prime\ast}e) \boxtimes N\bigr)_{(\fq\times\id)^{-1}(U')}
\Bigr] = 0.
\end{equation}
(As in \S\ref{ss:induction-useful-notation}, we write $\fq:G\times
G'\rar{}\Gt$ for the quotient map and $\fp':G\times G'\rar{}G'$ for
the projection onto the second factor; and we are using the proof of
Lemma \ref{l:pullback-equivalence} given in
\S\ref{sss:proof-l:pullback-equivalence} to conclude that
$\fq^*\widetilde{e}\cong \fp^{\prime\ast}e$.)

\mbr

 Note that the morphism
\[
\mu\circ(\pi\times\id)\circ(\fq\times\id) : G\times G'\times G
\rar{} G
\]
is given by $(g,h,\ga)\longmapsto ghg^{-1}\ga$. Let us consider the
morphism
\[
\Phi' : G\times G'\times G \rar{} G\times G, \qquad
(g,h,\ga)\longmapsto (ghg^{-1}\ga,g).
\]
To establish \eqref{e:enough-2}, it suffices to prove that
\begin{equation}\label{e:enough-stronger-2}
\Phi'_! \Bigl[ \bigl( (\fp^{\prime\ast}e) \boxtimes
N\bigr)_{(\fq\times\id)^{-1}(U')} \Bigr] = 0.
\end{equation}
We also observe that
\[
(\fq\times\id)^{-1}(U') = \bigl\{ (g,h,\ga)\in G\times G'\times G
\st g^{-1}\ga g\in G\setminus G' \bigr\}.
\]

\mbr

 We will use the proper base change theorem to compute the stalk
of the left hand side of \eqref{e:enough-stronger-2} at a point
$(x,g)\in G(k)\times G(k)$. The fiber $\Phi^{\prime-1}(x,g)$ is
identified with the closed subset $\bigl\{ (h,\ga)\in G'\times G \st
ghg^{-1}\ga=x \bigr\}\subset G'\times G$. The equation
$ghg^{-1}\ga=x$ can be rewritten as $h\cdot g^{-1}\ga g=g^{-1}xg$,
so we see that $\Phi^{\prime-1}(x,g)\cap (\fq\times\id)^{-1}(U)$ can
be identified with the closed subset $ W = \bigl\{ (h,\ga')\in
G'\times (G\setminus G') \st h\ga' = g^{-1}xg \bigr\} $ via the
morphism $w:W\rar{}G\times G'\times G$ given by $(h,\ga')\longmapsto
(g,h,g\ga'g^{-1})$.

\mbr

 Since $N$ is $G$-equivariant, the pullback
$\la^*((\fp^{\prime\ast}e)\boxtimes N)$ is naturally identified with
$(e\boxtimes N)\bigl\lvert_W$, and thus, by the proper base change
theorem, $R\Ga_c\bigl(W,\la^*((\fp^{\prime\ast}e)\boxtimes N)\bigr)$
is naturally identified with the stalk at $g^{-1}xg$ of the
convolution $\overline{e}*N_{G/G'}$, where the meaning of the
subscript $G\setminus G'$ is as before: $N_{G\setminus G'}$ is the
extension of $N\bigl\lvert_{G\setminus G'}$ to $G$ by zero outside
of $G\setminus G'$. Hence we are reduced to the following
\begin{lem}\label{l:sublemma}
Under the assumptions of Proposition \ref{p:restrict-induce}, we
have
\[
\overline{e}*N_{G\setminus G'} = 0 \qquad \text{for all }
N\in\sD_G(G).
\]
\end{lem}

\mbr

 To prove Lemma \ref{l:sublemma}, we use the distinguished
triangle \eqref{e:distinguished-triangle} with $M=N$, $X=G$ and
$Z=G'$. We see that it suffices to show that the natural morphism
$\overline{e}*N\rar{}\overline{e}*N_{G'}$ is an isomorphism. But
$\overline{e}\cong\overline{e}*\overline{e}$, and since
$N\in\sD_G(G)$, we see that
$\overline{e}*N\cong\overline{e}*N*\overline{e}$, functorially in
$N$ (see Lemma \ref{l:central-equivariant-derived}). On the other
hand, it is clear that
$\overline{e}*N_{G'}=\overline{e*(N\bigl\lvert_{G'})}$, and since
$e$ is a weak idempotent in $\sD_{G'}(G')$, we also have $e*M\cong
e*M*e$ functorially with respect to $M\in\sD(G')$ (applying Lemma
\ref{l:central-equivariant-derived} to $G'$ in place of $G$). Thus
we are reduced to showing that the morphism
\[
\overline{e}*N*\overline{e} \rar{} \overline{e}*N_{G'}*\overline{e},
\]
induced by the adjunction morphism $N\rar{}N_{G'}$, is an
isomorphism. Applying the distinguished triangle
\eqref{e:distinguished-triangle} once again, we see that it is
enough to show that
\begin{equation}\label{e:enough-3}
\overline{e}*N_{G\setminus G'}*\overline{e} = 0 \qquad \text{for all
} N\in\sD_G(G).
\end{equation}

\mbr

 Finally, to prove \eqref{e:enough-3}, note that
$\overline{e}*N_{G\setminus G'}*\overline{e}=\mu'_{3!}(e\boxtimes
N\bigl\lvert_{G\setminus G'}\boxtimes e)$, where
$\mu'_3:G'\times(G\setminus G')\times G'\rar{}G$ is given by
$(g_1,g_2,g_3)\longmapsto g_1 g_2 g_3$. Consider the map
\[
\la : G'\times(G\setminus G')\times G'\rar{}G\times (G\setminus G'),
\qquad (g_1,g_2,g_3)\longmapsto (g_1 g_2 g_3, g_2).
\]
By the proper base change theorem, the stalk of
$\la_!\bigl(e\boxtimes N\bigl\lvert_{G\setminus G'}\boxtimes
e\bigr)$ at a point $(g,x)\in G(k)\times \bigl(G(k)\setminus
G'(k)\bigr)$ is isomorphic to $N_x\tens
(\overline{e}*\de_x*\overline{e})_g$, where $N_x$ is the stalk of
$N$ at $x$. But $\overline{e}*\de_x*\overline{e}=0$ by assumption,
so $\la_!(e\boxtimes N\bigl\lvert_{G\setminus G'}\boxtimes e)=0$.
This forces \eqref{e:enough-3}, completing the proof of Lemma
\ref{l:sublemma} and of Proposition \ref{p:restrict-induce}.
\qed

\subsubsection{Proof of Proposition \ref{p:induce-restrict}} Once
again, the argument is very similar to the ones used in the proofs
of Propositions \ref{p:induction-convolution} and
\ref{p:restrict-induce}. The morphism $i:G'\rar{}\Gt$ is a closed
immersion; let $U''\subset\Gt$ denote the complement of its image.
As usual, we have an exact triangle
$\Mt_{U''}\rar{}\Mt\rar{}\Mt_{i(G')}\rar{}\Mt_{U''}[1]$, where the
meaning of the subscripts $U''$ and $i(G')$ is as before. In
addition, we have $\Mt_{i(G')}\cong i_!M$ by definition, and
therefore $\pi_!\Mt_{i(G')}\cong\Mb$. Thus we obtain a natural
morphism
\[
\overline{e}*\ig(M) = \mu_!(\overline{e}\boxtimes\pi_!\Mt) \rar{}
\mu_!(\overline{e}\boxtimes\pi_!\Mt_{i(G')}) \cong
\overline{e}*\overline{M}\cong\overline{e*M}.
\]
Restricting it to $G'$ yields a morphism
\[
e*\bigl( (\ig M)\bigl\lvert_{G'} \bigr) \rar{} e*M,
\]
functorial in $M\in e\sD_{G'}(G')$, and we would like to show that
it is an isomorphism.

\mbr

 As before, it is enough to prove that
\begin{equation}\label{e:enough-4}
\overline{e}*\bigl(\pi_!\Mt_{U''}\bigr)=0.
\end{equation}
In turn, to establish this equality, it is enough (by Corollary
\ref{c:push-pull-summand}) to check that
\begin{equation}\label{e:**}
\overline{e}* \pi_!\fq_! \bigl(\fq^*(\Mt_{U''})\bigr)=0.
\end{equation}
where $\fq:G\times G'\rar{}\Gt$ is the quotient map.

\mbr

 As always, we may and do assume that $k=\kbar$. Now
$\fq^{-1}(U'')=(G\setminus G')\times G'$, and
$\fq^*(\Mt_{U''})=(\fp^{\prime\ast}M)_{((G\setminus G')\times G')}$,
where $\fp':G\times G'\rar{}G'$ is the projection onto the second
factor. Thus the left hand side of \eqref{e:**} can be rewritten as
$\Psi_!(e\boxtimes\ql\boxtimes M)$, where we define
\[
\Psi:G'\times(G\setminus G')\times G'\rar{} G \qquad\text{by}\qquad
(h_1,g,h_2)\longmapsto h_1\cdot gh_2g^{-1},
\]
and $\ql$ denotes the constant rank $1$ local system on $G\setminus
G'$.

\mbr

 Let us consider instead the morphism
\[
\Phi'':G'\times(G\setminus G')\times G'\rar{} G\times (G\setminus
G'), \qquad (h_1,g,h_2)\longmapsto (h_1gh_2g^{-1},g).
\]
The fiber of $\Phi''$ over $(x,g)\in G(k)\times (G(k)\setminus
G'(k))$ is naturally identified with
\[
W' = \bigl\{ (h_1,h_2)\in G'\times G' \st h_1 g h_2 = xg \bigr\}
\]
via the morphism
\[
w':W'\rar{}G'\times(G\setminus G')\times G', \qquad
(h_1,h_2)\longmapsto(h_1,g,h_2).
\]
By the proper base change theorem, we have
\[
\Phi''_!(e\boxtimes\ql\boxtimes M)_{(x,g)} \cong
R\Ga_c\bigl(W',w^{\prime\ast}(e\boxtimes\ql\boxtimes M) \bigr) \cong
(\overline{e}*\de_g*\Mb)_{xg}.
\]
The latter stalk is zero because $g\in G(k)\setminus G'(k)$ and
$\Mb\cong\overline{e}*\Mb$. Thus we have proved that
$\Phi''_!(e\boxtimes\ql\boxtimes M)=0$. \emph{A fortiori},
$\Psi_!(e\boxtimes\ql\boxtimes M)=0$, which is equivalent to
\eqref{e:**}, which in turn implies \eqref{e:enough-4} and completes
the proof of Proposition \ref{p:induce-restrict}.

\subsection{Compatibility of induction with
twists}\label{ss:ind-twists}
Our final goal in this section is the following

\begin{prop}\label{p:induction-twists}
For every $M\in\sD_{G'}(G')$, we have $\ig(\te'_M)=\te_{\ig(M)}$ as
automorphisms of $\ig(M)$, where $\te'$ and $\te$ are the twists in
the categories $\sD_{G'}(G')$ and $\sD_G(G)$, respectively,
introduced in Definition \ref{d:twists-group}.
\end{prop}

\begin{lem}\label{l:induction-adjunction-properties}
Let $j:G'\into G$ denote the inclusion map and put $d=\dim(G/G')$. Then $\Ig$ is right adjoint to the restriction functor $j^*:\sD_G(G)\rar{}\sD_{G'}(G')$, and $\ig$ is left adjoint to the functor $j^![2d](d):\sD_G(G)\rar{}\sD_{G'}(G')$.
\end{lem}

\begin{proof}{Proof}
Observe that $j=\pi\circ i$ with the notation of \S\ref{sss:induction-auxiliary}. Next apply the definitions of $\Ig$ and $\ig$ (see Definition \ref{d:induction} and Remark \ref{r:induction}(2)) and use the fact that $\pi_*$ and $\pi_!$ are, respectively, right and left adjoint to $\pi^*$ and $\pi^!$.
\end{proof}

To formulate the next lemma we let $\bD_G:\sD_G(G)\rar{}\sD_G(G)$ denote the Verdier duality functor and write $\bD_G^-=\iota^*\circ\bD_G=\bD_G\circ\iota^*:\sD_G(G)\rar{}\sD_G(G)$ (following \cite{intro}), where $\iota:G\rar{}G$ is given by $g\mapsto g^{-1}$.

\begin{lem}\label{l:ind-Ind-Verdier}
The functors $\bD^-_G\circ\Ig\circ\bD^-_{G'}$ and
$\ig[2d](d)$ are isomorphic.
\end{lem}

\begin{proof}{Proof}
This follows from Lemma \ref{l:induction-adjunction-properties} using the fact that $\bD_G^-$ and $\bD_{G'}^-$ are anti-auto-equivalences together with the isomorphism $\bD_{G'}^-\circ(j^![2d](d))\circ\bD_G^-\cong j^*$.
\end{proof}

\begin{proof}{Proof of Proposition \ref{p:induction-twists}}
By definition, $\te'_{j^*N}=j^*(\te_N)$ for all $N\in\sD_G(G)$. Lemma \ref{l:induction-adjunction-properties} formally implies that $\Ig(\te'_M)=\te_{\Ig(M)}$ for all $M\in\sD_{G'}(G')$. It is shown in \cite[Prop. 7.2]{tanmay} that $\bD^-_G(\te_N)=\te_{\bD^-_G N}$ for all $N\in\sD_G(G)$ (and a similar statement holds for $\bD^-_{G'}$). Now Lemma \ref{l:ind-Ind-Verdier} finishes the proof.
\end{proof}



\section{Inner forms of algebraic groups and
$G$-schemes}\label{s:inner-forms}

The material of this section will be used to study the relationship between the induction functor introduced in \S\ref{s:induction} above and the operation of induction of class functions on finite groups (see \S\ref{ss:induction-sheaves-functions} below). It is also a necessary ingredient in the formulation of the relationship between character sheaves on a \emph{disconnected} unipotent group $G$ over $\bF_q$ and irreducible representations of $G(\bF_q)$; cf.~\cite{charsheaves}.

\subsection{Notation}
We fix an algebraic closure $\bF$ of a finite field of characteristic $p>0$. If $q$ is a power of $p$, we write $\bF_q$ for the unique subfield of $\bF$ consisting of $q$ elements. Given a scheme $X$ over $\bF_q$, we write $\Fr_q$ for the Frobenius endomorphism of $X\tens_{\bF_q}\bF$ (it is obtained by extension of scalars from the absolute Frobenius $\Phi_q:X\rar{}X$).

\mbr

Suppose $\Ga$ is an abstract group and $\vp:\Ga\rar{\simeq}\Ga$ is an automorphism. We can use $\vp$ to define an action of $\bZ$ on $\Ga$, and hence obtain the pointed set $H^1(\bZ,\Ga)$. Concretely, $H^1(\bZ,\Ga)$ can be identified with the set of \emph{$\vp$-conjugacy classes} in $\Ga$, the latter being the orbits of the $\Ga$-action on itself defined by $\ga:g\mapsto\vp(\ga)g\ga^{-1}$.

\subsection{Galois cohomology and torsors}\label{ss:galois-cohomology-torsors}
If $G$ is an algebraic group over $\bF_q$, the first Galois cohomology $H^1(\bF_q,G)$ is the pointed set of isomorphism classes of right $G$-torsors.
We can consider the action of $\bZ$ on $G(\bF)$ such that
$1\in\bZ$ acts via $\Fr_q$ and form the pointed set $H^1(\bZ,G(\bF))$ as above. The following result is standard (part (b) is due to Serge Lang \cite{lang}).

\begin{lem}\label{l:galois-cohomology-torsors}
\begin{enumerate}[$($a$)$]
\item Let $P$ be a right $G$-torsor, choose $p\in P(\bF)$, and
let $g\in G(\bF)$ be the unique element such that $p=\Fr_q(p)\cdot
g$. Then the $\Fr_q$-conjugacy class of $g$ in $G(\bF)$ is
independent of $p$, and the map $[P]\longmapsto[g]$ gives a
bijection
\[
H^1(\bF_q,G) \rar{\simeq} H^1(\bZ,G(\bF)).
\]
 \sbr
\item Let $G^\circ$ denote the neutral connected component of
$G$, and let $\Pi=G/G^\circ$. The natural map
$H^1(\bZ,G(\bF))\rar{}H^1(\bZ,\Pi(\bF))$ is a bijection.
 \sbr
\item Suppose that $G$ is a closed subgroup of an
algebraic group $U$ over $\bF_q$, form the quotient $Y=U/G$,
and let $\pi:U\rar{}Y$ denote the quotient morphism. The map
$y\longmapsto \pi^{-1}(y)$ induces a bijection between the set of
$U(\bF_q)$-orbits in $Y(\bF_q)$ and the kernel of the natural
map $H^1(\bF_q,G)\rar{}H^1(\bF_q,U)$.
\end{enumerate}
\end{lem}

\begin{rems}
\begin{enumerate}[$($1$)$]
\item We recall that the kernel of a pointed map between pointed sets
$(S_1,s_1)\rar{f}(S_2,s_2)$ is defined as the subset
$f^{-1}(s_2)\subset S_1$.
 \sbr
\item The action of $U$ on $Y$ is by left translations.
 \sbr
\item If $y\in
Y(\bF_q)$, then $\pi^{-1}(y)$ is a closed subvariety of $U$
defined over $\bF_q$, and the action of $G$ on $U$ by right
multiplication makes $\pi^{-1}(y)$ a right $G$-torsor.
\end{enumerate}
\end{rems}

\subsection{Inner forms of algebraic
groups}\label{ss:inner-forms-algebraic-groups} We continue working
in the setup of \S\ref{ss:galois-cohomology-torsors}. Given $\al\in
H^1(\bF_q,G)$, we would like to define an inner form $G^\al$ of
$G$ determined by $\al$. Let $P$ be a right $G$-torsor whose
isomorphism class equals $\al$. Briefly, $G^\al$ is the group of
automorphisms of $P$ that commute with the right $G$-action. To
define $G^\al$ more formally, we consider a functor, which we
denote by $G^{P}$, from the category of $\bF_q$-schemes to the
category of groups, constructed as follows.

\mbr

Let $S$ be any $\bF_q$-scheme. We can view $P\times S$ as an
$S$-scheme, and we have a right action of $G$ on $P\times S$ by
$S$-scheme automorphisms. Then $G^{P}(S)$ is defined as the
group of $S$-scheme automorphisms of $P\times S$ that commute with
the $G$-action.

\begin{lem}
The functor $G^{P}$ is representable by an algebraic group over
$\bF_q$. Moreover, $G^{P}\tens_{\bF_q}\bF\cong
G\tens_{\bF_q}\bF$ as algebraic groups over $\bF$.
\end{lem}

\begin{proof}{Proof}
In the case where $P$ is a trivial torsor (i.e.,
$P(\bF_q)\neq\varnothing$), one checks that $G^{P}$ is
representable by $G$ itself. In general, we have
$P(\bF_{q^n})\neq\varnothing$ for some $n\geq 1$. Thus $G^{P}$
is representable by $G\tens_{\bF_q}\bF_{q^n}$ after base change to
$\bF_{q^n}$, and Galois descent implies that $G^{P}$ is
representable over $\bF_q$.
\end{proof}

\begin{rem}
If $P'$ is another right $G$-torsor that is isomorphic to $P$,
then a choice of an isomorphism $P\rar{\simeq}P'$ induces an
isomorphism $G^{P}\rar{\simeq}G^{P'}$. Moreover, by
definition, any two isomorphisms $P\rar{}P'$ differ by an
element of $G^{P}(\bF_q)$. Consequently, we have an isomorphism
between $G^{P}$ and $G^{P'}$ that is unique up to inner
automorphisms.
\end{rem}

\begin{defin}\label{d:inner-form-subgroup}
Let $G$ be an algebraic group over $\bF_q$, and let $\al\in
H^1(\bF_q,G)$. For a representative $P$ of the isomorphism class
$\al$, we will write, somewhat imprecisely, $G^\al=G^{P}$. We
call $G^\al$ the \emph{inner form of $G$ defined by $\al$}.
\end{defin}

\begin{rem}
In view of the previous remark, we see that the set of conjugacy
classes in the group $G^\al(\bF_q)$ is determined
\emph{canonically} by $\al$.
Similarly, we can speak of \emph{the} set of irreducible
characters of the group $G^\al(\bF_q)$.
\end{rem}

\begin{rem}
The reader may prefer the following more concrete description of the
inner form $G^{P}$. Fix $p\in P(\bF)$, and let $g\in G(\bF)$ be the unique
element such that $p=\Fr_q(p)\cdot g$ $($so the $\Fr_q$-conjugacy
class of $g$ in $G(\bF)$ is the element of $H^1(\bZ,G(\bF))$
corresponding to the isomorphism class of $P${}$)$. Then there exists
an isomorphism
$G^{P}\tens_{\bF_q}\bF\rar{\simeq}G\tens_{\bF_q}\bF$ such
that the Frobenius endomorphism for $G^{P}$ becomes identified
with the endomorphism $x\mapsto g^{-1}\Fr_q(x)g$ of $G$.
\end{rem}

\subsection{Inner forms of
$G$-schemes}\label{ss:inner-forms-G-varieties} We remain in the
setup of \S\ref{ss:galois-cohomology-torsors}. Given $\al\in
H^1(\bF_q,G)$ and an $\bF_q$-scheme of finite type $X$ equipped with a
left $G$-action, we would like to define an inner form $X^\al$
of $X$ so that the corresponding inner form $G^\al$ of $G$
acts on $X^\al$. Once again, let $P$ be a representative of
$\al$, and define $G^{P}$ as above. Note that by construction,
$G^{P}$ acts on $P$ on the left; in fact, $P$ is a left
$G^{P}$-torsor. We also consider the free left action of $G$
on the product $P\times X$ given by $g\cdot(p,x)=(p\cdot
g^{-1},g\cdot x)$, and we form the quotient
$X^{P}:=G\setminus(P\times X)$. The actions of $G$ and
$G^{P}$ on $P\times X$ commute (here, $G^{P}$ acts on
$X$ trivially), so we obtain an induced action of $G^{P}$ on
$X^{P}$.

\begin{defin}
We write $X^\al=X^{P}$ $($somewhat imprecisely$)$, and we call
$X^\al$ $($together with the left action of $G^\al$ constructed above$)$ the \emph{inner form of the $G$-scheme $X$ defined by the cohomology class $\al$}.
\end{defin}

The next fact follows directly from the definitions.

\begin{lem}\label{l:inner-forms-compatibility}
Let $G$ be an algebraic group over $\bF_q$, let $P$ be a right
$G$-torsor, and let $X=G$ equipped with the conjugation action
of $G$. Then $X^{P}$ is naturally isomorphic to $G^{P}$,
also equipped with the conjugation action of $G^{P}$.
\end{lem}

\subsection{Transport of equivariant
complexes}\label{ss:transport-equivariant-complexes} In this section
we assume that $G$ is a unipotent\footnote{This assumption is
imposed only because we decided to work with the ``naive''
definition of an equivariant derived category.} algebraic group over
$\bF_q$. Given $\al\in H^1(\bF_q,G)$ and a scheme $X$ of finite
type over $\bF_q$ equipped with a left $G$-action, our goal is to
define a canonical ``transport functor'' (in fact, an equivalence of
categories) $\sD_{G}(X)\rar{\sim}\sD_{G^\al}(X^\al)$.

\mbr

As usual, we choose a representative $P$ of $X$. Let
$\pr_2:P\times X\rar{}X$ denote the second projection, and let
$q:P\times X\rar{}X^{P}$ denote the quotient morphism for
the free left $G$-action defined in
\S\ref{ss:inner-forms-G-varieties}. As we already remarked, the
product $G\times G^{P}$ acts on $P\times X$ on the left;
moreover, $\pr_2$ is the quotient map for the action of $G^{P}$,
which is also free. Thus both pullback functors
\[
\sD_{G}(X) \xrightarrow{\ \pr_2^* \ } \sD_{G\times
G^{P}}(P\times X) \xleftarrow{\ q^*\ }
\sD_{G^{P}}(X^{P})
\]
are equivalences of categories.

\begin{defin}\label{d:transport-functor}
The composition
$q^*\circ(\pr_2^*)^{-1}:\sD_{G}(X)\rar{\sim}\sD_{G^\al}(X^\al)$
is called the functor of \emph{transport of equivariant complexes},
and is denoted by $M\longmapsto M^\al$.
\end{defin}

\begin{rem}
If an object $M\in\sD_{G}(X)$ comes from a $G$-equivariant
local system on $X$, then $M^\al$ is also a
$G^\al$-equivariant local system on $X^\al$.
\end{rem}

As a corollary of Lemma \ref{l:inner-forms-compatibility}, we now
also have the construction of a transport functor
$\sD_{G}(G)\rar{\sim}\sD_{G^\al}(G^\al)$, which is again
denoted by $M\longmapsto M^\al$.

\subsection{Alternative descriptions}\label{ss:alternative-inner}
In this subsection we present a slightly different viewpoint on the
constructions introduced in
\S\S\ref{ss:inner-forms-algebraic-groups}--\ref{ss:transport-equivariant-complexes}.
It has the advantage of being somewhat more concrete, although it is less evident that the constructions
appearing in this subsection are independent of the choices involved
in them.

\begin{prop}\label{p:alternative-inner}
Let $G$ be a closed subgroup of an algebraic group $U$ over
$\bF_q$. Define $Y=U/G$, equipped with the left $U$-action
by translations, let $\pi:U\rar{}Y$ denote the quotient map,
write $\overline{1}=\pi(1)$, and put $Z=\bigl\{(u,y)\st u\cdot
y=y\bigr\}\subset U\times Y$. We consider the diagonal action of
$U$ on $U\times Y$, where the action on the first factor is by
conjugation, and remark that $Z$ is stable under this action.
 \mbr
Finally, for each $y\in Y(\bF_q)$, let $\al(y)\in H^1(\bF_q,G)$
denote the isomorphism class of the right $G$-torsor $\pi^{-1}(y)$
$($cf.~Lemma \ref{l:galois-cohomology-torsors}$($c$))$.
 \sbr
\begin{enumerate}[$($a$)$]
\item For every $y\in Y(\bF_q)$, the stabilizer, $U^y$, of $y$ in $U$ is
isomorphic to the inner form $G^{\al(y)}$ of $G$ defined by the
cohomology class $\al(y)$.
\item Let $X$ be a scheme of finite type over $\bF_q$ equipped
with a left $G$-action, and let $\widetilde{X}=(U\times
X)/G$, where the right $G$-action on $U\times X$ is given
by $(u,x)\cdot g=(ug,g^{-1}\cdot x)$. Write
$p:\widetilde{X}\rar{}Y$ for the induced morphism. For every
$y\in Y(\bF_q)$, the fiber $p^{-1}(y)$ is isomorphic
to\footnote{Observe that $p^{-1}(y)$ is stable under $U^y\subset
U$; thus we have a left action of $U^y$ on $p^{-1}(y)$.} the
inner form $X^{\al(y)}$ in a way compatible with the isomorphism of
part $($a$)$.
\item Assume that $U$ is unipotent. For every $y\in Y(\bF_q)$,
the inclusion $j_y:U^y\into Z$, given by $g\mapsto (g,y)$,
induces an equivalence
$j_y^*:\sD_{U}(Z)\rar{\sim}\sD_{U^y}(U^y)$ $($as usual,
$U^y$ acts on itself by conjugation$)$. Furthermore, the
composition
\[
j_y^*\circ (j^*_{\overline{1}})^{-1} : \sD_{G}(G) \rar{\sim}
\sD_{U^y}(U^y) \simeq \sD_{G^{\al(y)}}(G^{\al(y)})
\]
is isomorphic to the transport functor introduced in
\S\ref{ss:transport-equivariant-complexes}.
\item Again, assume that $U$ is unipotent, and let the notation be
as in part $($b$)$. Given $y\in Y(\bF_q)$, write $\Xt^y=p^{-1}(y)$,
and let $i_y:\Xt^y\into\widetilde{X}$ denote the inclusion. Then
$i_y^*:\sD_{U}(\widetilde{X})\rar{}\sD_{U^y}(\Xt^y)$ is an
equivalence, and the composition
\[
i_y^*\circ (i^*_{\overline{1}})^{-1} : \sD_{G}(X) \rar{\sim}
\sD_{U^y}(\Xt^y) \simeq \sD_{G^{\al(y)}}(X^{\al(y)})
\]
is isomorphic to the transport functor of Definition
\ref{d:transport-functor}.
\end{enumerate}
\end{prop}

\subsection{Relation between $\ig$ and induction of class
functions}\label{ss:induction-sheaves-functions}

\begin{prop}\label{p:induction-sheaves-functions}
Let $G$ be a unipotent group over $\bF_q$, let $G'\subset G$ be a
closed subgroup, and let $M\in\sD_{G'}(G')$. Then
\begin{equation}\label{e:induction-sheaves-functions}
t_{\ig M} = \sum_{\al\in\Ker(H^1(\bF_q,G')\rar{}H^1(\bF_q,G))}
\ind_{G^{\prime\al}(\bF_q)}^{G(\bF_q)} t_{M^\al}.
\end{equation}
\end{prop}

\begin{rems}
\begin{enumerate}[$($1$)$]
\item For every $\al\in\Ker(H^1(\bF_q,G')\rar{}H^1(\bF_q,G))$,
we realize $G^{\prime\al}$ as a subgroup of $G$ using
Proposition \ref{p:alternative-inner}(a).
 \sbr
\item The notation on the right hand side of
\eqref{e:induction-sheaves-functions} is as follows: if
$\Ga'\subset\Ga$ are finite groups, then
$\iga:\Fun(\Ga')^{\Ga'}\rar{}\Fun(\Ga)^\Ga$ denotes the usual
induction map from class functions on $\Ga'$ to class functions on
$\Ga$ $($cf.~\S\ref{sss:induction-motivation}$)$.
 \sbr
\item As a special case of Proposition \ref{p:induction-sheaves-functions}, we observe that if $G'$ is connected, then $H^1(\bF_q,G')$ is trivial, so the sum in \eqref{e:induction-sheaves-functions} reduces to $\ind_{G'(\bF_q)}^{G(\bF_q)}t_M$. Hence for connected $G'$, the proposition states that $\ig$ is compatible with induction of class functions $($via the sheaves-to-functions correspondence$)$ ``on the nose.''
\end{enumerate}
\end{rems}

\begin{proof}{Proof of Proposition \ref{p:induction-sheaves-functions}}
Form the quotient $\Gt=(G\times G')/G'$, where the right action of $G'$ on $G\times G'$ is given by $(g,g')\cdot\ga=(g\ga,\ga^{-1}g'\ga)$, and equip it with the $G$-action induced by the left translation action of $G$ on the first factor in $G\times G'$.

\mbr

The conjugation map $G\times G'\rar{}G$ (given by $(g,g')\mapsto gg'g^{-1}$) induces a $G$-equivariant morphism $\pi:\Gt\rar{}G$, and the map $i:G'\into G\times G'$ given by $g'\mapsto(1,g')$ induces a $G'$-equivariant morphism $i:G'\rar{}\Gt$.

\mbr

Fix $M\in\sD_{G'}(G')$. By Lemma \ref{l:pullback-equivalence}, there is a (unique up to isomorphism) object $\Mt\in\sD_G(\Gt)$ such that $i^*\Mt\cong M$. We put $N=\pi_!(\Mt)\in\sD_G(G)$, so that, by definition, $N\cong\ig M$. By Lemma \ref{l:sheaves-functions-properties}(c), $t_N=\pi_!(t_{\Mt})$, so to prove \eqref{e:induction-sheaves-functions} we need to calculate the function $t_{\Mt}:\Gt(\bF_q)\rar{}\ql$.

\mbr

To this end, define $Y=G/G'$ and equip it with the translation action of $G$.
The first projection $G\times G'\rar{}G$ induces a $G$-equivariant morphism $p:\Gt\rar{}Y$.
Then $\Gt(\bF_q)$ is the disjoint union of the sets of $\bF_q$-points of the fibers $p^{-1}(y)$, where $y$ ranges over $Y(\bF_q)$. For each $y\in Y(\bF_q)$, write $G^y\subset G$ for the stabilizer of $y$ in $G$ and observe that the fiber $p^{-1}(y)\subset\Gt$ is $G^y$-stable.

\mbr

The next result is straightforward.

\begin{lem}\label{l:straightforward}
Let $y\in Y(\bF_q)$ and choose $g\in G(\bF)$ that maps onto $y$. Then
\begin{enumerate}[$($a$)$]
\item $G^y=gG'g^{-1}$;
\item the map $G^y\rar{}G\times G'$ given by $\ga\mapsto(g,g^{-1}\ga g)$ induces a $G^y$-equivariant inclusion $i_y:G^y\into\Gt$ $($where $G^y$ acts on itself by conjugation$)$;
\item $i_y$ induces an isomorphism $G^y\rar{\simeq}p^{-1}(y)$, which is independent of the choice of $g$; and
\item the composition $\pi\circ i_y:G^y\rar{}G$ is equal to the natural inclusion $G^y\into G$.
\end{enumerate}
\end{lem}

We can now complete the proof of \eqref{e:induction-sheaves-functions}. For each $y\in Y(\bF_q)$, write $\al(y)\in H^1(\bF_q,G')$ for the isomorphism class of the right $G'$-torsor $q^{-1}(y)\subset G$, where $q:G\rar{}Y$ is the quotient map. If $M^y=i_y^*\Mt\in\sD_{G^y}(G^y)$, then by Lemma \ref{l:straightforward}(c) and Proposition \ref{p:alternative-inner}(d), we can identify $M^y$ with $M^{\al(y)}\in\sD_{G'^{\al(y)}}(G'^{\al(y)})$, where $G'^{\al(y)}$ is identified with $G^y$ using Proposition \ref{p:alternative-inner}(a). Since $t_N=\pi_!(t_{\Mt})$, Lemma \ref{l:straightforward}(d) shows that
\begin{equation}\label{e:big-sum}
t_{\ig M} = t_N = \sum_{y\in Y(\bF_q)} \overline{t_{M^y}},
\end{equation}
where $\overline{t_{M^y}}$ denotes the function on $G(\bF_q)$ obtained from $t_{M^y}:G^y(\bF_q)\rar{}G(\bF_q)$ via extension by zero. Now suppose $\cO\subset Y(\bF_q)$ is a single $G(\bF_q)$-orbit and set $\al(\cO)=\al(y)$ for any $y\in\cO$ (note that $\al(y)$ does not depend on the choice of $y\in\cO$). It is then easy to see that
\begin{equation}\label{e:small-sum}
\sum_{y\in\cO} \overline{t_{M^y}} = \ind_{G'^{\al}(\bF_q)}^{G(\bF_q)} t_{M^\al}, \qquad \text{where } \al=\al(\cO).
\end{equation}
As $\cO$ ranges over all $G(\bF_q)$-orbits in $Y(\bF_q)$, the corresponding cohomology class $\al(\cO)\in H^1(\bF_q,G')$ ranges over $\Ker(H^1(\bF_q,G')\rar{}H^1(\bF_q,G))$ by Lemma \ref{l:galois-cohomology-torsors}(c). Combining this observation with \eqref{e:big-sum}--\eqref{e:small-sum} yields \eqref{e:induction-sheaves-functions}.
\end{proof}



\section{Geometric reduction process}\label{s:reduction}

\subsection{Overview} The main goal of this section is to prove the
following result:
\begin{thm}\label{t:adm-pair-compatible}
Let $G$ be a $($possibly disconnected$)$ unipotent group over
$\bF_q$, let $\rho$ be a nonzero representation of $G(\bF_q)$ over
$\ql$, and let $(A,\cN)$ be a pair consisting of a normal connected
subgroup $A\subset G$ and a $G$-invariant multiplicative $\ql$-local
system $\cN$ on $A$ such that the restriction of $\rho$ to
$A(\bF_q)$ is scalar, given by the $1$-dimensional character
$t_{\cN}:A(\bF_q)\rar{}\ql^\times$. Then there exists an admissible
pair $(H,\cL)$ for $G$ such that $A\subset H$,
$\cN\cong\cL\bigl\lvert_A$, and the restriction of $\rho$ to
$H(\bF_q)$ has as a direct summand the $1$-dimensional
representation defined by $t_{\cL}:H(\bF_q)\rar{}\ql^\times$.
\end{thm}

We note that, in particular, we can take $G$ to be connected, $\rho$
to be irreducible, and the pair $(A,\cN)$ to be trivial. In this
case, in view of the Frobenius reciprocity, the proposition implies
the existence of an admissible pair $(H,\cL)$ for $G$ such that
$\rho$ is a direct summand of $\Ind_{H(\bF_q)}^{G(\bF_q)}t_{\cL}$,
which proves the claim of \S\ref{ss:step-1} above.

\mbr

The reason for stating Theorem \ref{t:adm-pair-compatible} the way
we did is that this approach makes it easier to give an inductive
proof of the proposition. The title of this section (``Geometric
reduction process'') is motivated by an analogy between (the proof
of) Theorem \ref{t:adm-pair-compatible} and the (algebraic)
reduction process described in one of the appendices to
\cite{intro}, where it is proved that every irreducible
representation of a finite nilpotent group $\Ga$ can be
\emph{canonically} realized as a representation induced from a
``Heisenberg representation'' (\emph{op.~cit.}) of a subgroup of
$\Ga$.

\mbr

The notion of an admissible pair (for a unipotent group over an
arbitrary field of characteristic $p>0$) is defined in
\S\ref{ss:def-admissible}. However, it is more convenient to
formulate this definition in the framework of Serre duality
developed in the Appendix, rather than in the language of
multiplicative local systems. Therefore we first explain the
relationship between these two approaches to Serre duality in
\S\ref{ss:Serre-duality-two-approaches}. We then state an auxiliary
result in \S\ref{ss:ext-mult} (it is equivalent to one
of the results proved in the Appendix) and use it to prove Theorem
\ref{t:adm-pair-compatible} in
\S\S\ref{ss:geom-heis-reps}--\ref{ss:proof-adm-pair-compatible}.

\subsection{Two approaches to Serre
duality}\label{ss:Serre-duality-two-approaches} Let us first fix an
arbitrary field $k$ and a connected algebraic group $G$ over $k$. If
$A$ is an abstract abelian group, we will view $A$ as a discrete
group scheme over $k$. Thus we have the notion of a central
extension of $G$ by $A$, as well as the notion of a
\emph{multiplicative $A$-torsor} on $G$ (defined by an obvious
analogy with Definition \ref{d:multiplicative}). Let us define a
\emph{rigidification} of an $A$ torsor $\cE$ on $G$ to be a
trivialization of the pullback $1^*\cE$, where $1:\Spec k\rar{}G$ is
the unit morphism, and let us define a \emph{rigidified} $A$-torsor
on $G$ to be an $A$-torsor on $G$ equipped with a chosen
rigidification. Since $G$ is connected and $A$ is discrete, it is
easy to see that rigidified $A$-torsors on $G$ form a discrete
groupoid (i.e., a category with no non-identity morphisms). On the
other hand, plain $A$-torsors on $G$ form a groupoid where the group
of automorphisms of every object is isomorphic to $A$.

\mbr

The following result is proved in \cite{masoud}.
\begin{lem}\label{l:mult-tors-central-exts}
\begin{enumerate}[$($a$)$]
\item Every multiplicative $A$-torsor on $G$ admits a rigidification,
and the forgetful functor induces a bijection between the set of
isomorphism classes of rigidified multiplicative $A$-torsors on $G$
and that of plain ones. \sbr
\item The natural forgetful functor from the groupoid
of central extensions of $G$ by $A$ to the groupoid of rigidified
multiplicative $A$-torsors on $G$ is an equivalence.
\end{enumerate}
\end{lem}

\mbr

Now we assume that $k$ has characteristic $p>0$ and $G$ is a
connected \emph{unipotent} group over $k$. Fix a prime $\ell\neq p$.
Our goal is to relate multiplicative $\ql$-local systems on $G$ to
central extensions of $G$ by the discrete group $\qzp$.

\mbr

Let us fix a homomorphism $\psi:(\bQ_p,+)\rar{}\ql^\times$ with
kernel equal to $\bZ_p$. Then $\psi$ identifies $\qzp$ with the
group $\mu_{p^\infty}(\ql)$ of roots of unity in $\ql^\times$ whose
order is a power of $p$. Given a central extension
$1\rar{}\qzp\rar{}\Gt\rar{}G\rar{}1$, we can view $\Gt$ as a
multiplicative $\qzp$-torsor on $G$, and using the homomorphism
$\psi$, we obtain the induced multiplicative $\ql$-local system on
$G$, which we denote $\Gt_\psi$.

\begin{lem}\label{l:mult-loc-sys-unip-centr-ext-qzp}
The map $\Gt\mapsto\Gt_\psi$ constructed above is an isomorphism between the group $H^2(G,\qzp)$ of isomorphism classes of central extensions of $G$ by $\qzp$ and the
group of isomorphism classes of multiplicative $\ql$-local systems
on $G$.
\end{lem}

\begin{proof}{Proof}
In view of Lemma \ref{l:mult-tors-central-exts}, it is enough to
show that if $\cL$ is an arbitrary multiplicative $\ql$-local system
on $G$, then $\cL$ is induced from a $\qzp$-torsor on $G$ via
$\psi$. To this end, we will show that if
$f:\pi_1(G)\rar{}\ql^\times$ is the homomorphism corresponding to
$\cL$, then the image of $f$ is finite and is contained in the
subgroup $\mu_{p^\infty}(\ql)$, where $\pi_1(G)$ is the algebraic
fundamental group of $G$, see \cite{sga1}.

\mbr

Choose an algebraic closure $\overline{k}$ of $k$, write
$\Gb=G\tens_k\overline{k}$, let $1:\Spec k\rar{}G$ denote the unit
morphism as before, and let $\overline{1}$ denote the corresponding
$\overline{k}$-point of either $G$ or $\Gb$. By \cite[Thm.~IX.6.1]{sga1}, we have a short exact sequence of groups
\[
1 \rar{} \pi_1(\Gb,\overline{1}) \rar{} \pi_1(G,\overline{1}) \rar{}
\operatorname{Gal}(\overline{k}/k) \rar{} 1,
\]
which is split by the homomorphism
$\operatorname{Gal}(\overline{k}/k) \rar{} \pi_1(G,\overline{1})$
induced by the morphism $1:\Spec k\rar{}G$. By the definition of
multiplicativity, the $\ql$-local system $1^*\cL$ on $\Spec k$
satisfies $(1^*\cL)\tens(1^*\cL)\cong 1^*\cL$, whence $1^*\cL$ is
trivial. In other words, the composition
$\Gal(\overline{k}/k)\rar{1_*}\pi_1(G,\overline{1})\rar{f}\ql^\times$
is trivial. Thus $f$ is determined by its restriction to
$\pi_1(\Gb,\overline{1})$, which we will also denote by $f$.

\mbr

By definition, $f$ factors through a continuous homomorphism
$\pi_1(\Gb,\overline{1})\rar{}K^\times$, where $K$ is a finite
extension of $\bQ_\ell$ contained in $\ql$. Moreover, by
compactness, the image of $f$ must lie in $\cO_K^\times$, where
$\cO_K\subset K$ is the ring of integers of $K$. The structure of
$\cO_K^\times$ is known; in particular, if $\fm_K$ denotes the
unique maximal ideal of $\cO_K$, then
$(\cO_K^\times)/(1+\fm_K)\cong(\cO_K/\fm_K)^\times$ is finite, and
$1+\fm_K$ has a descending filtration by closed subgroup with
successive quotients isomorphic to the additive group of the residue
$\cO_K/\fm_K$, which is a finite field of characteristic $\ell$. It
follows that the subgroup
$\mu_{p^\infty}(K)=K\cap\mu_{p^\infty}(\ql)\subset\cO_K^\times$ is
finite, and the quotient $\cO_K^\times/\mu_{p^\infty}(K)$ is a
profinite group whose order is relatively prime to $p$.

\mbr

On the other hand, since $G$ is connected,
$\Gb$ is isomorphic to an affine space over $\overline{k}$, so its
algebraic fundamental group has no nontrivial quotients of order
prime to $p$. Thus the image of $f$ lies in $\mu_{p^\infty}(K)$,
completing the proof.
\end{proof}

\subsection{Definition of an admissible
pair}\label{ss:def-admissible} The notion of an admissible pair is a
\emph{geometric} one; thus we will first formulate it for an
algebraically closed base field, and then for an arbitrary one.
Moreover, it is more convenient to begin by working in the framework
of Serre duality developed in the Appendix.

\subsubsection*{Normalizer of a pair $(H,\chi)$}
Let us fix a perfect field $k$ of characteristic $p>0$ and a
unipotent\footnote{The definition applies equally well (with obvious
simplifications) in the case where $G$ is a \emph{perfect} unipotent
group over $k$ (see \S\ref{aa:perfect-unipotent}).} algebraic group
$G$ over $k$. From now on, by default, all subgroups of group
schemes are assumed to be closed. Consider a pair $(H,\chi)$
consisting of a \emph{connected} subgroup $H\subset G$ and a central
extension $1\rar{}\qzp\rar{}\Ht\rar{\chi}H\rar{}1$. We can view
$\chi$ as a $k$-point of the Serre dual $H^*$ of $H$, see
\S\ref{aa:noncomm-Serre}, which is a perfect (possibly disconnected)
commutative unipotent group over $k$ by Proposition
\ref{p:serre-noncomm}. Let $N_G(H)$ denote the normalizer of $H$ in
$G$. Since the Serre dual $H^*$ is defined by a universal property,
we obtain an induced regular action of the perfectization
$N_G(H)_{per\!f}$ on $H^*$ by $k$-group scheme automorphisms. We let
$G'_{per\!f}\subset N_G(H)_{per\!f}$ denote the stabilizer of $\chi$
under this action. The notation is explained by the fact that
$G'_{per\!f}$ is the perfectization of a uniquely determined closed
subgroup of $G$, which we denote by $G'$ and (unambiguously) call
the ``normalizer of $(H,\chi)$ in $G$''. We remark that $G'$ may be
disconnected even if $G$ is connected.

\begin{defin}\label{d:adm-pair-centr-ext-alg-closed}
Let $k$ be an algebraically closed field of characteristic $p>0$,
let $G$ be a unipotent algebraic group (or perfect unipotent group)
over $k$, and let $(H,\chi)$ be a pair consisting of a connected
subgroup $H\subset G$ and an element $\chi\in H^*(k)$. We say that
this pair is \emph{admissible for $G$} if the following three
conditions are satisfied.
\begin{enumerate}[(1)]
\item Let $G'$ be the normalizer of $(H,\chi)$ in $G$, defined in
the previous paragraph. Then the quotient group $G^{\prime\circ}/H$
is commutative, i.e., $[G^{\prime\circ},G^{\prime\circ}]\subset H$.
\item The homomorphism
$(G^{\prime\circ}/H)_{per\!f}\rar{}(G^{\prime\circ}/H)_{per\!f}^*$
constructed in \S\ref{aa:construction-alt}, which is well defined in
our situation in view of condition (1), is an \emph{isogeny}.
\item Given $g\in G(k)$, write $H^g=g^{-1}Hg$, and
let $\chi^g\in(H^g)^*(k)$ be obtained from $\chi$ by transport of
structure via $H^g\rar{\simeq}H$, $h\longmapsto ghg^{-1}$. If
$g\not\in G'(k)$, then
\[
\chi\bigl\lvert_{(H\cap H^g)^\circ} \not\cong
\chi^g\bigl\lvert_{(H\cap H^g)^\circ}
\]
\end{enumerate}
\end{defin}

\begin{defin}\label{d:adm-pair-centr-ext}
Let $k$ be an arbitrary field of characteristic $p>0$, let $G$ be a
unipotent algebraic group over $k$, let $H\subset G$ be a connected
subgroup, and let $\chi$ be a central extension of $H$ by $\qzp$.
The pair $(H,\chi)$ is said to be \emph{admissible for $G$} if the
pair $(H\tens_k\overline{k},\chi\tens_k\overline{k})$ obtained from
$(H,\chi)$ by base change to an algebraic closure $\overline{k}$ of
$k$ is admissible for $G\tens_k\overline{k}$.
\end{defin}

\subsubsection*{Admissible pairs in the context of multiplicative
local systems} Now let $k$ be a field of characteristic $p>0$, let
$\ell$ be a prime different from $p$, and choose a homomorphism
$\psi:(\bQ_p,+)\rar{}\ql^\times$ with kernel $\bZ_p$, as in
\S\ref{ss:Serre-duality-two-approaches}. Fix a unipotent algebraic
group $G$ over $k$, a connected subgroup $H\subset G$, and a
multiplicative $\ql$-local system $\cL$ on $H$ (see Definition
\ref{d:multiplicative}). By Lemma
\ref{l:mult-loc-sys-unip-centr-ext-qzp}, $\cL$ is associated to a
unique (up to isomorphism) central extension $\chi$ of $H$ by $\qzp$
via the homomorphism $\psi$.

\begin{defin}\label{d:adm-pair-mult-loc-sys}
We say that the pair $(H,\cL)$ is \emph{admissible for $G$} if the
pair $(H,\chi)$ is admissible for $G$ in the sense of Definition
\ref{d:adm-pair-centr-ext}. In the case where $k$ is perfect, the
\emph{normalizer} of the pair $(H,\cL)$ in $G$ is defined as the
normalizer of $(H,\chi)$ in $G$.
\end{defin}

It is not hard to see that this definition does not depend on the
choice of $\psi$.

\subsection{Extension of multiplicative local
systems}\label{ss:ext-mult} The next result is used in the proof of
Theorem \ref{t:adm-pair-compatible}. However, we also find it to be
interesting in its own right. It is a natural geometrization of the
fact that if $\Ga$ is a group, $H\subset\Ga$ is a subgroup such that
$[\Ga,\Ga]\subset H$, and $\chi:H\rar{}\ql^\times$ is a
homomorphism, then $\chi$ extends to a homomorphism
$\Ga\rar{}\ql^\times$ if and only if
$\chi\bigl\lvert_{[\Ga,\Ga]}\equiv 1$ (a simple exercise).

\begin{prop}\label{p:ext-loc-sys}
Let $G$ be a connected unipotent group over an arbitrary field $k$
of characteristic $p>0$, let $\ell$ be a prime different from $p$,
let $H\subset G$ be a closed connected subgroup such that
$[G,G]\subset H$, and let $\cL$ be a multiplicative $\ql$-local
system on $H$. Then there exists a multiplicative $\ql$-local system
$\cL'$ on $G$ with $\cL'\bigl\lvert_H\cong\cL$ if and only if the
pullback $\com^*\cL$ is a trivial $\ql$-local system on $G\times G$,
where $\com:G\times G\rar{}H$ is the commutator morphism,
$\com(g_1,g_2)=g_1g_2g_1^{-1}g_2^{-1}$.
\end{prop}

In view of Lemma \ref{l:mult-loc-sys-unip-centr-ext-qzp}, this
result is essentially equivalent to Proposition \ref{p:lifts},
proved in \S\ref{aa:lift-central-exts} of the Appendix.

\begin{rem}
Naively, one might have replaced the condition that $\com^*\cL$ is
trivial by the stronger requirement that $\cL\bigl\lvert_{[G,G]}$ is
a trivial $\ql$-local system on $[G,G]$. However, the latter
condition is \emph{not} necessary. Indeed, as explained in
\cite{masoud}, there are examples of connected unipotent groups $G$
for which there is a multiplicative $\ql$-local system $\cL$ on $G$
with $\cL\bigl\lvert_{[G,G]}$ being nontrivial. $($We leave it to the
reader to check that the fake Heisenberg groups defined in
\S\ref{ss:fake-Heis} are among such examples.$)$
\end{rem}

\subsection{A special case of Theorem
\ref{t:adm-pair-compatible}}\label{ss:geom-heis-reps} In this
subsection we prove Theorem \ref{t:adm-pair-compatible} in the
special case where $\rho$ is irreducible and $[G,G^\circ]\subset A$.
Using the construction explained in \S\ref{aa:construction-alt}, we
see that $\cN$ induces a homomorphism of perfect $\bF_q$-groups
$\phi_{\cN}:(G/A)_{perf}\rar{}(G^\circ/A)_{perf}^*$. Let $H$ be a
maximal (with respect to inclusion) connected subgroup of $G$ with
the property that $A\subset H$ and the composition
\[
(H/A)_{per\!f} \hookrightarrow (G/A)_{per\!f} \xrar{\ \ \phi_{\cN}\
\ } (G^\circ/A)^*_{per\!f} \onto (H/A)^*_{per\!f}
\]
is trivial (the subscript ``perf'' is defined in
\S\ref{aa:perfect}). By the definition of $\phi_{\cN}$, this implies
that the pullback of $\cN$ by the commutator map $H\times H\rar{}A$
is trivial. By Proposition \ref{p:ext-loc-sys}, there is a
multiplicative $\ql$-local system $\cL$ on $H$ with
$\cN\cong\cL\bigl\lvert_A$.

\mbr

We first claim that $\cL$ can be chosen so that
$t_{\cL}:H(\bF_q)\rar{}\ql^\times$ is a direct summand of the
restriction of $\rho$ to $H(\bF_q)$. Indeed, since the homomorphism
\[
(H/A)(\bF_q)=(H/A)_{per\!f}(\bF_q)\rar{}(H/A)_{per\!f}^*(\bF_q)=(H/A)^*(\bF_q)
\]
induced by $\cN$ is trivial, we see, in particular, that
$t_{\cN}:A(\bF_q)\rar{}\ql^\times$ is trivial on
$[H(\bF_q),H(\bF_q)]$. Now let $V$ denote the representation space
of $\rho$, so that $A(\bF_q)$ acts on $V$ through the character
$t_{\cN}$. Then we see that $\rho(H(\bF_q))\subset \Aut(V)$ is a
commutative subgroup; in particular, by Schur's lemma, there exists
a character $\nu:H(\bF_q)\rar{}\ql^\times$ which is contained in the
restriction of $\rho$ to $H(\bF_q)$. \emph{A fortiori}, $\nu$ and
$t_{\cL}$ agree on $A(\bF_q)$, and hence $\nu\cdot t_{\cL}^{-1}$
comes from a character $(H/A)(\bF_q)\rar{}\ql^\times$. But $H/A$ is
a connected commutative algebraic group over $\bF_q$, so, as we
mentioned in Remark \ref{r:character-mult-loc-sys}, there exists a
multiplicative $\ql$-local system $\cE$ on $H/A$ such that $\nu\cdot
t_{\cL}^{-1}=t_{\widetilde{\cE}}$, where $\widetilde{\cE}$ is the
pullback of $\cE$ to $H$. In other words,
$\nu=t_{\cL\tens\widetilde{\cE}}$. But the restriction of
$\widetilde{\cE}$ to $A$ is trivial by construction, so we may
replace $\cL$ with $\cL\tens\widetilde{\cE}$ without loss of
generality.

\mbr

It remains to verify that $(H,\cL)$ is
admissible. Since $[G,H]\subset[G,G^\circ]\subset A\subset H$, we
see that $H$ is normal in $G$, so condition (3) in the
definition of admissibility is automatic. Condition (1) holds
because $H\supset A$ and $G^\circ/A$ is central in $G/A$. Finally,
condition (2) holds by the maximality requirement in the choice of
$H$.

\subsection{Proof of Theorem
\ref{t:adm-pair-compatible}}\label{ss:proof-adm-pair-compatible} Let
us complete the proof of Theorem \ref{t:adm-pair-compatible} in
general. We will use simultaneous induction on $\dim G$ and on the
length of the \'etale group $\pi_0(G)=G/G^\circ$ (i.e., the number
of elements of $\pi_0(G)(\bF)$). By the result of
\S\ref{ss:geom-heis-reps}, we may assume that
$[G,G^\circ]\not\subset A$. We may also assume that $\rho$ is
irreducible.

\mbr

Let $Z\subset G^\circ$ denote the preimage in $G$ of the neutral
component of the center of $G/A$. By assumption, $Z\neq G^\circ$, so
$Z$ is a proper connected subgroup of $G^\circ$. As in
\S\ref{ss:geom-heis-reps}, ${\cN}$ induces a $k$-group morphism
$\phi_{\cN}:(G/A)_{per\!f}\rar{}(Z/A)^*_{per\!f}$. Since $\dim
Z<\dim G$, the restriction of this morphism to
$(Z/A)_{per\!f}\rar{}(Z/A)_{per\!f}^*$ has positive dimensional
kernel\footnote{If $K$ is the neutral connected component of the
kernel of $\phi_{\cN}$, then $K$ is normal in $G/A$ and we obtain a
decreasing sequence $K$, $[G/A,K]$, $[G/A,[G/A,K]],\dotsc$ of normal
connected subgroups of $G/A$. The last nontrivial term of this
sequence is contained in $Z/A$ by the definition of $Z$.}.
Hence there is a connected subgroup $B\subset Z$ such that
$A\subsetneq B$ and the composition
$
(B/A)_{perf}\hookrightarrow(G/A)_{perf}\to(Z/A)_{perf}^*\onto(B/A)_{perf}^*
$
is trivial. As in \S\ref{ss:geom-heis-reps}, we see that there
exists a multiplicative $\ql$-local system $\cN'$ on $B$ such that
$\cN\cong\cN'\bigl\lvert_A$ and $t_{\cN'}$ is a summand of $\rho\bigl\lvert_{B(\bF_q)}$.

\mbr

On the other hand, $B$ is normal in $G$ since $A\subset
B\subset Z$, and hence $\cN'$ is not $G$-invariant by the maximality
of $(A,\cN)$. Let $G_1$ denote the normalizer of $\cN'$ in $G$. Then
$G_1$ is a proper subgroup of $G$, so either $\dim G_1<\dim G$, or
$\abs{\pi_0(G_1)(\bF)}<\abs{\pi_0(G)(\bF)}$. In either case, if we
let $\rho_1$ be the restriction of $\rho$ to $G_1(\bF_q)$, we may
assume that Theorem \ref{t:adm-pair-compatible} holds
for $\rho_1$ and the pair $(B,\cN')$.

\mbr

Let $(H,\cL)$ denote a pair consisting of a connected subgroup
$H\subset G_1$ and a multiplicative $\ql$-local system $\cL$ on $H$,
which satisfies the conclusion of Theorem
\ref{t:adm-pair-compatible} for the $4$-tuple $(G_1,\rho_1,B,\cN')$.
We assert that it also satisfies the conclusion of this theorem for
$(G,\rho,A,{\cN})$. To prove this assertion, we only need to check
that $(H,\cL)$ is admissible with respect to $G$.

\mbr

Let $G'$ denote the normalizer of the pair $(H,\cL)$ in $G$.
We have $G'\subset G_1$. Indeed, if $g\in G(\bF)$,
$g\not\in G_1(\bF)$, then, by construction, $g$ does not fix $\cN'$,
and hence, \emph{a fortiori}, it cannot fix $(H,\cL)$ (because $B$
is normal in $G$ and $\cL\bigl\lvert_B\cong\cN'$).

\mbr

Since $G'\subset G_1$, conditions (1) and (2) in
the definition of admissibility for $(H,\cL)$ hold with respect to
$G$ because they hold with respect to $G_1$. To verify condition
(3), let $g\in G(\bF)$, $g\not\in G'(\bF)$. If $g\in G_1(\bF)$,
there is nothing to do because $(H,\cL)$ is admissible with respect
to $G_1$. If $g\not\in G_1(\bF)$, then, since $B\subset(H\cap
H^g)^\circ$ and $g$ does not fix $\cN'$, it follows that the
restrictions of $\cL$ and $\cL^g$ to $(H\cap H^g)^\circ$ cannot be
isomorphic, completing the induction step in the proof of Theorem
\ref{t:adm-pair-compatible}.



\section{Analysis of Heisenberg idempotents}\label{s:heis-idemp}

In this section we study a certain special type of geometrically
minimal weak idempotents (cf.~Definition \ref{d:geom-min-idemp}) in
the equivariant derived categories of unipotent algebraic groups.
The main result of the section is Proposition \ref{p:heis-idemp}.

\subsection{Setup}\label{ss:heis-idemp-setup}
Throughout this section, we fix a field $k$ of characteristic $p>0$,
let $U$ be a possibly disconnected unipotent group over $k$, and let
$(N,\cL)$ be an admissible pair for $U$ in the sense of Definition
\ref{d:adm-pair-mult-loc-sys}, such that its normalizer in $U$ is
all of $U$ (so that the third condition in the definition of
admissibility is vacuous). In particular, $N$ is a normal closed
connected subgroup of $U$, and $\cL$ is a multiplicative $\ql$-local
system on $N$ which is invariant under the conjugation action of
$U$. In this subsection we construct certain objects associated to
the data $U,N,\cL$.

\subsubsection{Construction of $e_\cL$ and
$e'_\cL$}\label{sss:heis-idemp-construction} It follows from our
assumptions that $\cL$ has a natural $U$-equivariant structure (because $N$ is connected). Let
$\bK_N$ denote the dualizing complex of $N$; then $\bK_N$ has a
natural $U$-equivariant structure as well, since $N$ is normal in
$U$. It follows that we can define $e_{\cL}=\cL\tens\bK_N$ as an
object of $\sD_U(N)$. Let $e'_{\cL}$ denote the object of $\sD_U(U)$
obtained from $e_{\cL}$ via extension by zero.

\subsubsection{Construction of a morphism
$\e\rar{}e'_\cL$}\label{sss:heis-idemp-closed} Let $1:\Spec
k\rar{}U$ be the unit morphism, and let $\e=1_!\ql$ be the
delta-sheaf at $1$, equipped with the ``trivial'' $U$-equivariant
structure. Recall that $\e$ is a unit object in the monoidal
category $\sD_U(U)$ under convolution. Of course, we can equally
well think of $\e$ as an object of $\sD_N(N)$. If $p:N\rar{}\Spec k$
is the structure morphism, then $\bK_N=p^!\ql$, so we get a
canonical identification $\ql\rar{\simeq}1^!\bK_N$. By adjunction,
we get a canonical morphism $\e\rar{}\bK_N$. On the other hand,
since the stalk $\cL_1$ of $\cL$ at $1$ has a natural
trivialization, we obtain an isomorphism
$\e\tens\cL^\vee\rar{\simeq}\e$, where $\cL^\vee=\hom(\cL,\ql)$ is
the dual local system on $\cL$. Composing the two morphisms we just constructed, we obtain a natural
morphism $\e\tens\cL^\vee\rar{}\bK_N$, which induces a morphism
$\e\rar{}e_{\cL}$ in $\sD_U(N)$, and hence a morphism
$\e\rar{}e'_{\cL}$ in $\sD_U(U)$.

\subsubsection{Construction of a homomorphism $\vp_\cL:(U^\circ/N)_{perf}\rar{}(U^\circ/N)_{perf}^*$}
\label{sss:heis-idemp-pairing} Before stating the main result of the
section, we need one last construction. Let us fix a homomorphism
$\psi:(\bQ_p,+)\rar{}\ql^\times$ with kernel $\bZ_p$. By Lemma
\ref{l:mult-loc-sys-unip-centr-ext-qzp}, the multiplicative local
system $\cL$ on $N$ is induced from a central extension $\Nt$ of $N$
by $\qzp$ via $\psi$. Let $k_{perf}$ be the perfect closure of $k$
(see \cite{greenberg}); it is the maximal purely inseparable
algebraic extension of $k$, and hence is determined up to a unique
$k$-isomorphism. As recalled in \S\ref{aa:perfect}, for every
$k$-scheme $X$, we can construct its \emph{perfectization},
$X_{perf}$, which is a scheme over $k_{perf}$. In particular, we
obtain the induced central extension $\Nt_{perf}$ of $N_{perf}$ by
$\qzp$, which is $U_{perf}$-invariant. Using the construction
explained in \S\ref{aa:construction-alt} with $Z=U_{perf}^\circ$,
the neutral connected component of $U$ (recall that $U^\circ/N$ is
commutative, so this choice of $Z$ is allowed), we obtain a
homomorphism of perfect unipotent groups
$U^\circ_{perf}/N_{perf}\rar{}\bigl(U^\circ_{perf}/N_{perf}\bigr)^*$
over $k_{perf}$, i.e., a homomorphism
\[
\vp_\cL:(U^\circ/N)_{perf}\rar{}(U^\circ/N)_{perf}^*.
\]
By the definition of admissibility, $\vp_\cL$ is an isogeny.

\subsection{Statement of the main result}\label{ss:heis-idemp-main}

\begin{prop}\label{p:heis-idemp} Let $U$, $N$, $\cL$, $e_\cL$, $e'_\cL$
and $\vp_\cL$ be as above.
\begin{enumerate}[$($a$)$]
\item The morphism $\e\rar{}e'_\cL$ constructed in
\S\ref{sss:heis-idemp-closed} becomes an isomorphism after
convolving with $e'_\cL$. A fortiori, $e'_\cL$ is a weak idempotent
in $\sD_U(U)$. \sbr
\item In fact, $e'_\cL$ is a geometrically minimal weak idempotent
in $\sD_U(U)$ $($Def.~\ref{d:geom-min-idemp}$)$. \sbr
\item Let $\te$ denote the twist automorphism of the identity
functor on $\sD_{U\tens_k\kbar}(U\tens_k\kbar)$, introduced in
\S\ref{ss:twists-groups}. If the restriction of $\te$ to the Hecke
subcategory
\[
e'_\cL\sD_{U\tens_k\kbar}(U\tens_k\kbar) \,\subset\,
\sD_{U\tens_k\kbar}(U\tens_k\kbar)
\]
is trivial, then $U$ is connected and $\vp_\cL$ is an isomorphism.
\end{enumerate}
\end{prop}

This proposition is the last ingredient in the proofs of the main
results of our work, stated in Section \ref{s:results}. The rest of
the section is devoted to its proof. If $U$ is as above, a weak
idempotent in $\sD_U(U)$ isomorphic to one of the form $e'_\cL$ will
be called a \emph{Heisenberg idempotent}\footnote{The notion of a
Heisenberg idempotent is the geometric analogue of the notion of a
Heisenberg representation of a finite group, introduced in
\cite{intro}.}, which explains the title of the section.

\begin{rem}
One can show that the converse of Proposition \ref{p:heis-idemp}(c)
holds as well (see \cite{tanmay}), but we do not need this fact.
\end{rem}

\subsection{Proof of Proposition
\ref{p:heis-idemp}(a)}\label{ss:proof-p:heis-idemp-a} It is
enough to show that the morphism $\e\rar{}e_\cL$ in $\sD_N(N)$
becomes an isomorphism after convolving with $e_\cL$. Without loss
of generality, we may and do assume that $k$ is algebraically
closed. Then $e_\cL\cong\cL[2\dim N]$, so it suffices to prove that
$\e\rar{}e_\cL$ becomes an isomorphism after convolving with $\cL$.

\mbr

Fix $g\in N(k)$, and let $\rho_g:N\rar{}N$ be defined by
$n\longmapsto n^{-1}g$. By the proper base change theorem, the
induced morphism on the stalks, $(\e*\cL)_g\rar{}(e_\cL*\cL)_g$, is
the same as the morphism obtained by applying the functor
$R\Ga_c(N,-)$ to the induced morphism
$\e\tens(\rho_g^*\cL)\rar{}e_\cL\tens(\rho_g^*\cL)$. However,
$\rho_g^*\cL$ is naturally isomorphic to $\cL^\vee$, because $\cL$ is
multiplicative, and the morphism
$\e\tens(\rho_g^*\cL)\rar{}e_\cL\tens(\rho_g^*\cL)$ becomes the
canonical morphism $\e\rar{}\bK_N$. Applying the functor
$R\Ga_c(N,-)$, we recover the adjunction morphism
$\ql\rar{}p_!p^!\ql$ (where $p:N\rar{}\Spec k$ is the structure
morphism), which is an isomorphism because $N$ is a
connected unipotent group over $k$, and hence is isomorphic to an
affine space over $k$.

\subsection{Proof of Proposition \ref{p:heis-idemp}(b)}
Without loss of generality, we may and do assume that $k$ is algebraically closed. Then we must prove that $e'_{\cL}$ is a minimal weak idempotent in $\sD_U(U)$, which is equivalent to showing that the Hecke subcategory $e'_{\cL}\sD_U(U)$ contains no weak idempotents apart from $0$ and $e'_{\cL}$.

\mbr

The category $e'_{\cL}\sD_U(U)$ is studied in\footnote{The results of \cite{tanmay} use the construction of the arrow $\e\rar{}e'_{\cL}$ presented in \S\ref{sss:heis-idemp-closed}, but are otherwise independent of the current section.} \cite{tanmay}. Let us recall the results of \emph{op.~cit.} that will be used in the current proof.

\begin{thm}[\cite{tanmay}, Theorem 1.5]\label{t:tanmay}
Let $\cM\subset e'_{\cL}\sD_U(U)$ be the full subcategory consisting of objects $M$ for which $M[-\dim N]$ is a perverse\footnote{See \cite{bbd}; we only consider the \emph{middle perversity} in this article.} sheaf on $U$.
 \sbr
\begin{enumerate}[$($a$)$]
\item The natural functor $D^b(\cM)\rar{}e'_{\cL}\sD_U(U)$ is an equivalence of categories.
 \sbr
\item The subcategory $\cM$ is closed under convolution, and is a $($semisimple$)$ fusion category with unit object $e'_{\cL}$.
 \sbr
\item There exists a ribbon structure on the fusion category $\cM$, which makes $\cM$ a modular category, and is such that the corresponding twist $($balancing$)$ is equal to the canonical automorphism $\te$ of the identity functor introduced in \S\ref{ss:twists-groups}.
\end{enumerate}
\end{thm}

\begin{rems}\label{r:fusion-modular-remarks}
\begin{enumerate}[$($1$)$]
\item The word ``semisimple'' in the formulation of part (b) is only added for emphasis. Our use of the term ``fusion category'' agrees with that of \cite{ENO}. Thus part (b) means that $\cM$ is a semisimple $\ql$-linear monoidal category over $\ql$, which is rigid, has finitely many simple objects and finite dimensional $\Hom$-spaces, and is such that the unit object is simple.
 \sbr
\item We do not require the precise definitions of the terms ``ribbon structure'' or ``modular category''; see, e.g., \cite{BK}. All we need is the fact that a modular fusion category where the twist is trivial has only one simple object (which is necessarily the unit object).
\end{enumerate}
\end{rems}

We see that Proposition \ref{p:heis-idemp}(b) follows from the more general

\begin{lem}\label{l:idempotents-fusion-category}
Let $\cM$ be a weakly symmetric fusion category. The bounded derived category $D^b(\cM)$ $($equipped with the induced tensor product$)$ has no weak idempotents other than $0$ and the unit object.
\end{lem}

\begin{proof}{Proof}
Let $\tens$ and $\e$ denote the monoidal functor and the unit object of $\cM$.
In the proof we will repeatedly use the observation that if $X,Y\in\cM$ are nonzero, then $X\tens Y$ is also nonzero. The reason is that if $X^*$ is a right dual\footnote{We are using the terminology of \cite{BK,ENO}.} of $X$, then $X^*\tens X$ contains $\e$ as a direct summand, and hence $X^*\tens X\tens Y$ contains $Y$ as a direct summand.

\mbr

First let us check that every weak idempotent in $D^b(\cM)$ comes from a simple object of $\cM$. Write $\cM^{gr}$ for the category of bounded graded objects of $\cM$; in other words, $\cM^{gr}$ is the category of bounded complexes over $\cM$ in which all the differentials are equal to $0$. The cohomology functor $H^\bullet:D^b(\cM)\rar{}\cM^{gr}$ is an equivalence of categories because $\cM$ is semisimple, and it is moreover a monoidal equivalence by the K\"unneth formula. The comment in the previous paragraph implies that the length function $\ell:\cM^{gr}\rar{}\bZ_{\geq 0}$, which assigns to an object $X\in\cM^{gr}$ the sum of lengths of all the components of $X$, satisfies $\ell(X\tens Y)\geq\ell(X)\cdot\ell(Y)$. It follows that every weak idempotent in $\cM^{gr}$ has length $1$, and hence it must be a simple object concentrated in a single degree $k$. Now we are forced to have $k+k=k$, and hence $k=0$, as desired.

\mbr

Now let $X\in\cM$ be a nonzero weak idempotent. Then the right dual $X^*$ is also a weak idempotent, and hence so is $X^*\tens X$ (since $\cM$ is weakly symmetric). We already saw that $X^*\tens X$ must be simple, and thus $X^*\tens X\cong\e$. So $X$ is invertible, and since $X\tens X\cong X$, we see that $X\cong\e$, proving the lemma.
\end{proof}

\subsection{Proof of Proposition \ref{p:heis-idemp}(c)}
By remark \ref{r:fusion-modular-remarks}(2), the hypothesis of Proposition \ref{p:heis-idemp}(c) implies that the category $\cM\subset e'_{\cL}\sD_U(U)$ defined in Theorem \ref{t:tanmay} has only one simple object, namely, $e'_{\cL}$ itself. Write $\Ga=U/U^\circ$, where $U^\circ\subset U$ is the neutral connected component, and let $\cM_0\subset e'_{\cL}\sD_{U^\circ}(U)$ be the full subcategory consisting of objects $M$ for which $M[-\dim N]$ is a perverse sheaf on $U$. The natural action of $\Ga$ on $\sD_{U^\circ}(U)$ induces an action of $\Ga$ on $\cM_0$, and by \cite[Lemma 1.4]{tanmay}, the $\Ga$-equivariantization $\cM_0^\Ga$ is equivalent to $\cM$. By \cite[Theorem 1.3]{tanmay}, $\cM_0$ is also a fusion category, so we see that $e'_{\cL}$ is the only simple object of $\cM_0$ (indeed, every simple object of $\cM_0$ can be realized as a direct summand of a simple object of $\cM_0^\Ga$). On the other hand, all simple objects of $\cM_0$ are described in \cite[\S4.1]{tanmay}, and that description implies that $U$ is connected and that $\vp_{\cL}$ is an isomorphism.



\section{The proofs of the main results}\label{s:proofs}

In this section we put together all the preliminary results obtained
in Sections \ref{s:equivariant-derived}--\ref{s:heis-idemp} to prove
Theorem \ref{t:dim-reps-easy}, Theorem \ref{t:descr-L-packets} and
Proposition \ref{p:L-packets-min-idemp}.

\subsection{The key result} Throughout this section we work with a
fixed connected unipotent group $G$ over a field $k$ of
characteristic $p>0$. For the most part, we will take $k=\bF_q$, but
it is convenient to formulate one of the results in the more general
setting. Below we will state a result (Proposition \ref{p:key}) to
which all the other results to be proved in this section are easily
reduced.

\mbr

It is convenient to introduce the following notation. If
$k=\bF_q$ and $e\in\sD_G(G)$ is any weak idempotent, let us write
$L(e)$ for the set of isomorphism classes of irreducible
representations $\rho$ of $G(\bF_q)$ over $\ql$ in which $t_e$ acts
by the identity operator. If $e$ is minimal (Definition
\ref{d:minimal-idempotent}), it follows from Definition
\ref{d:L-packets} and Remark \ref{r:idempotents-L-packets}(1) that
$L(e)$ is either empty (when $t_e\equiv 0$) or an $\bL$-packet (when
$t_e\not\equiv 0$).

\mbr

Let $(H,\cL)$ be an admissible pair for $G$, and
let $G'$ be its normalizer in $G$ $($see Definition
\ref{d:adm-pair-mult-loc-sys}$)$. Let $e'_{\cL}\in\sD_{G'}(G')$ be
the object obtained by applying the construction of
\S\ref{sss:heis-idemp-construction} with $G'$ and $H$ in place of
$U$ and $N$, and let $e_{H,\cL}=\ig e'_\cL$.

\begin{prop}\label{p:key}
With this notation,
\begin{enumerate}[$($a$)$]
\item $e_{H,\cL}$ is a geometrically minimal weak idempotent in
$\sD_G(G)$; and \sbr
\item if $\sC$ denotes the geometric conjugacy class of $(H,\cL)$
$($see \S\ref{ss:descr-L-packets}$)$ and $k=\bF_q$, then
\[L(\sC)=L(e_{H,\cL}),\] where $L(\sC)$ is introduced in Definition
\ref{d:adm-pair-L-packet}.
\end{enumerate}
\end{prop}

This proposition is proved in
\S\S\ref{ss:proof-p:key-a}--\ref{ss:proof-p:key-b}. First we explain
how it implies all the other results to be proved in this section.
In
\S\S\ref{ss:proof-t:descr-L-packets}--\ref{ss:proof-t:dim-reps-easy},
we assume that $k=\bF_q$.

\subsection{Proof of Theorem
\ref{t:descr-L-packets}}\label{ss:proof-t:descr-L-packets} Proposition \ref{p:key} implies that if $\sC$ is a
geometric conjugacy class of admissible pairs for $G$, then $L(\sC)$
is an $\bL$-packet of irreducible representations of $G(\bF_q)$.

\mbr

Conversely, consider an $\bL$-packet $\sP$ of irreducible
representations of $G(\bF_q)$ over $\ql$, and choose $\rho\in\sP$.
By Theorem \ref{t:adm-pair-compatible} (and Frobenius reciprocity),
there exists a geometric conjugacy class $\sC$ of admissible pairs
for $G$ such that $\rho\in L(\sC)$. Then
$L(\sC)\cap\sP\neq\varnothing$, and since $L(\sC)$ is also an
$\bL$-packet by the previous paragraph, we see that $L(\sC)=\sP$,
proving Theorem \ref{t:descr-L-packets}. \qed

\subsection{Proof of Proposition
\ref{p:L-packets-min-idemp}}\label{ss:proof-p:L-packets-min-idemp}
Let us fix two irreducible representations, $\rho_1$ and $\rho_2$,
of $G(\bF_q)$ over $\ql$. We tautologically have
$(i)\Longrightarrow(ii)\Longrightarrow(iii)$ in the statement of
Proposition \ref{p:L-packets-min-idemp}, so we only need to show
that $(iii)\Longrightarrow(i)$. Assume that $(iii)$ holds. By the
arguments above, there exists an admissible pair $(H,\cL)$ for $G$
such that $\rho_1\in L(e_{H,\cL})$. By Proposition \ref{p:key},
$e_{H,\cL}$ is a geometrically minimal weak idempotent in
$\sD_G(G)$. Now $t_{e_{H,\cL}}$ acts as the identity in $\rho_1$,
and hence, by assumption, it also acts as the identity in $\rho_2$.
This means that $\rho_1$ and $\rho_2$ both lie in $L(e_{H,\cL})$,
which is a single $\bL$-packet, and the proof is complete. \qed

\subsection{Proof of Theorem
\ref{t:dim-reps-easy}}\label{ss:proof-t:dim-reps-easy} In this
subsection we assume that $G$ is an \emph{easy} unipotent group over
$\bF_q$. Let $\rho$ be an irreducible representation of $G(\bF_q)$
over $\ql$. We must prove that the dimension of $\rho$ is a power of
$q$.

\subsubsection{Step 1}\label{sss:dimensions-step-1}
By the arguments above, there exists an admissible pair $(H,\cL)$
for $G$ such that $\rho$ is a direct summand of
$\Ind_{H(\bF_q)}^{G(\bF_q)}t_{\cL}$. Let $G'$ be the normalizer of
$(H,\cL)$ in $G$ (Definition \ref{d:adm-pair-mult-loc-sys}). We
first show, using Proposition \ref{p:heis-idemp}, that $G'$ is
connected and the dimension of $G'/H$ is even.

\mbr

Let $e'_\cL\in\sD_{G'}(G')$ be the extension of $\cL\tens\bK_H$ by
zero to $G'$, as before. From Proposition \ref{p:heis-idemp}(a) and
Lemma \ref{l:closed-idemp} below, it follows that the functor
$M\longmapsto e'_\cL*M$ is isomorphic to the identity functor on the
Hecke subcategory $e'_\cL\sD_{G'}(G')\subset\sD_{G'}(G')$. Now
Theorem \ref{t:induction-Hecke}(b) implies that the restriction
\[
\ig\bigl\lvert_{e'_\cL\sD_{G'}(G')} : e'_\cL\sD_{G'}(G') \rar{}
\sD_G(G)
\]
is a \emph{faithful} functor. By Lemma \ref{l:triv-twists-easy},
the twist automorphism of the identity functor on $\sD_G(G)$ is
trivial, because $G$ is easy. As $\ig$ is compatible with twists
(Proposition \ref{p:induction-twists}), the restriction of the twist
on $\sD_{G'}(G')$ to the Hecke subcategory $e'_\cL\sD_{G'}(G')$ is
trivial as well. These statements continue to hold after base change
from $\bF_q$ to $\bF$. By Proposition \ref{p:heis-idemp}(c), $G'$ is
connected, and the homomorphism
$\vp_\cL:(G'/H)_{perf}\rar{}(G'/H)^*_{perf}$ induced by $\cL$ is an
isomorphism. Since $\vp_\cL$ obviously arises from a skewsymmetric
bi-extension of $G'/H$ by $\qzp$ (cf.~Remark
\ref{r:varphi-skew-symmetric}; see also \S\ref{aa:construction-alt}
for the construction of $\vp_\cL$, and \S\ref{aa:symm-skew-symm} for
the terminology), it follows from Proposition
\ref{p:existence-lagr}(b) that $G'/H$ is even-dimensional.

\mbr

We now pause to state and prove the lemma used in the previous
paragraph.

\begin{lem}\label{l:closed-idemp}
Let $\cM$ be a monoidal category with monoidal bifunctor $\tens$ and
unit object $\e$, and consider an arrow $\e\rar{}e$ in $\cM$ that
becomes an isomorphism after tensoring with $e$ on the right. Then
the functor $X\longmapsto e\tens X$ is isomorphic to the identity
functor on the subcategory $e\cM\subset\cM$.
\end{lem}

\begin{rem}\label{r:closed-idemp}
Here, the notation is similar to that used in \S\ref{ss:weak-idemp},
namely, $e\cM$ is the essential image of the functor
$\cM\rar{}\cM$ given by $X\longmapsto e\tens X$. With the assumption
of the lemma, it is obvious that $e$ is a weak idempotent in $\cM$
in the sense of \S\ref{ss:weak-idemp}. However, the existence of an
arrow $\e\rar{}e$ satisfying the assumption of Lemma
\ref{l:closed-idemp} is a much stronger condition than merely
requiring $e$ to be a weak idempotent. (In \cite{intro}, arrows
$\e\rar{}e$ that become isomorphisms after tensoring with $e$ on
either side are called \emph{closed idempotents}.) In particular, we
do not expect the conclusion of Lemma \ref{l:closed-idemp} to hold
for an arbitrary weak idempotent $e\in\cM$.
\end{rem}

\begin{proof}{Proof of Lemma \ref{l:closed-idemp}}
We use the fact that $\tens$ is equipped with an
associativity constraint. If $X\in e\cM$, then $X\cong e\tens X$,
because $e\cong e\tens e$. Hence for any $X\in e\cM$, the arrow
$\e\rar{}e$ becomes an isomorphism after we apply the functor
$Y\longmapsto Y\tens X$ to it. This (together with the unit
constraint for $\tens$) gives us a functorial collection of
isomorphisms $X\rar{\simeq}e\tens X$ for all $X\in e\cM$, as
desired.
\end{proof}

\subsubsection{Step 2}\label{sss:dimensions-step-2} Now we complete the proof of Theorem
\ref{t:dim-reps-easy}. Consider the commutator morphism
$\com:G'\times G'\rar{}H$, $(g_1,g_2)\longmapsto g_1 g_2 g_1^{-1}
g_2^{-1}$, and form the pullback local system $\cL'=\com^*\cL$ on
$G'\times G'$. Since the map
$\vp_\cL:(G'/H)_{perf}\rar{}(G'/H)_{perf}^*$ induced by $\cL$ is an
isomorphism, it is easy to deduce from Proposition
\ref{p:Serre-Pontryagin} that the trace function $t_{\cL'} :
G'(\bF_q)\times G'(\bF_q) \rar{} \ql^\times$ descends to a perfect
pairing
\[
B_{\cL} : (G'/H)(\bF_q)\times (G'/H)(\bF_q) \rar{} \ql^\times,
\]
i.e., a bimultiplicative map that induces an isomorphism
\[
(G'/H)(\bF_q) \xrar{\ \ \simeq\ \ }
\Hom\bigl((G'/H)(\bF_q),\ql^\times\bigr).
\]
Next, the definition of $\vp_\cL$ implies that $B_\cL$
is equal to the map induced by the commutator pairing defined by the
character $t_\cL:H(\bF_q)\rar{}\ql$, namely,
\[
G'(\bF_q) \times G'(\bF_q) \xrar{\ \ \com\ \ } H(\bF_q) \xrar{\ \
t_\cL\ \ } \ql^\times.
\]
It is well known (see, e.g., the appendix on Heisenberg
representations in \cite{intro}) that the nondegeneracy of $B_\cL$
implies that $G'(\bF_q)$ has a unique irreducible representation,
call it $\rho'$, which acts on $H(\bF_q)$ by the scalar $t_\cL$.
Moreover, $\rho'$ has dimension $[G'(\bF_q):H(\bF_q)]^{1/2}$, which
is a power of $q$ by the first step of the proof. Furthermore, by
the Frobenius reciprocity, the irreducible representation of
$G(\bF_q)$ with which we started, $\rho$, is a direct summand of
$\Ind_{G'(\bF_q)}^{G(\bF_q)}\rho'$. However, the definition of an
admissible pair, together with Mackey's irreducibility criterion,
imply that $\Ind_{G'(\bF_q)}^{G(\bF_q)}\rho'$ is irreducible. Thus
$\rho\cong\Ind_{G'(\bF_q)}^{G(\bF_q)}\rho'$, whence (as $G'$ is
connected)
\[
\dim\rho = [G(\bF_q):G'(\bF_q)]\cdot\dim\rho = q^{\dim G-\dim
G'}\cdot [G'(\bF_q):H(\bF_q)]^{1/2},
\]
which is a power of $q$, completing the proof of Theorem
\ref{t:dim-reps-easy}.

\subsection{Proof of Proposition \ref{p:key}(a)}\label{ss:proof-p:key-a}
In this subsection $k$ is allowed to be an arbitrary field of
characteristic $p>0$. We will use the notation introduced at the
beginning of \S\ref{ss:when-isom}. In view of Corollary
\ref{c:induction-min-idemp}, it suffices to check that for every
$g\in G(\kbar)\setminus G'(\kbar)$, we have
$\overline{e'_\cL}*\de_g*\overline{e'_\cL}=0$. Since the notion of
an admissible pair is stable under base change from $k$ to $\kbar$,
we may as well \emph{assume that $k$ is algebraically closed}. Let
us fix $g\in G(k)\setminus G'(k)$.

\mbr

Consider the morphism
\[
m_g:H\times H\rar{}G, \qquad (h_1,h_2)\longmapsto h_1 g h_2.
\]
By definition,
$\overline{e'_\cL}*\de_g*\overline{e'_\cL}=m_{g!}(e_\cL\boxtimes
e_\cL)$. To complete the proof, it suffices to show that for every
$x\in G(k)$, the stalk $m_{g!}(e_\cL\boxtimes e_\cL)_x=0$. Up to
cohomological shift\footnote{Recall that $\bK_H\cong\ql[2\dim
H](\dim H)$, and Tate twists are trivial since $k=\kbar$.}, this is
the same as proving that $m_{g!}(\cL\boxtimes\cL)_x=0$. By the
proper base change theorem, this is equivalent to
$R\Ga_c(m_g^{-1}(x),\cL\boxtimes\cL)=0$.

\mbr

 Let us fix $x\in G(k)$. If $m_g^{-1}(x)=\varnothing$, there is
nothing to check. Otherwise, fix a $k$-point $(h_1,h_2)$ of
$m_g^{-1}(x)$. Then $m_g^{-1}(x)$ can be identified
with $H\cap gHg^{-1}$ via the map $w:H\cap gHg^{-1}\rar{} H\times H$
given by $w(h)=(h_1 h,g^{-1}h^{-1}gh_2)$.

\mbr

 The (isomorphism class of the) local system $\cL$ on $H$
is invariant under left and right translations (this follows
from the multiplicativity of $\cL$). Thus
\[
w^*(\cL\boxtimes\cL) \cong \cL\Bigl\lvert_{H\cap gHg^{-1}} \tens
^g\cL^\vee\Bigl\lvert_{H\cap gHg^{-1}},
\]
where $^g\cL$ denotes the multiplicative $\ql$-local system on
$gHg^{-1}$ obtained from $\cL$ by transport of structure via
$h\longmapsto ghg^{-1}$, and $^g\cL^\vee$ is its dual local system.

\mbr

 By the definition of admissibility, we are reduced to the
following well known
\begin{lem}\label{l:coho-nontriv-mult-loc-syst}
Let $A$ be an algebraic group over a field $k$, and let $\cL$ be a
multiplicative $\ql$-local system on $A$ such that
$\cL\bigl\lvert_{A^\circ}$ is nontrivial. Then $R\Ga_c(A,\cL)=0$.
\end{lem}

\begin{proof}{Proof}
It suffices to show that $f_!\cL=0$, where $f:A\rar{}\pi_0(A)$ is
the natural quotient morphism and $\pi_0(A)=A/A^\circ$. Since $\cL$
is multiplicative, it in turn suffices to show that
$R\Ga_c(A^\circ,\cL\bigl\lvert_{A^\circ})=0$. Thus we may assume,
without loss of generality, that $A$ is connected and $\cL$ is as
before.

\mbr

 The following diagram is clearly cartesian:
\[
\xymatrix{
  A\times A \ar[rr]^\mu \ar[d]_{\pr_1} & & A \ar[d]^\pi \\
  A \ar[rr]^\pi & & \Spec k
   }
\]
where $\pi:A\rar{}\Spec k$ is the structure morphism and
$\pr_1:A\times A\rar{}A$ is the projection onto the first factor. By
the proper base change theorem,
\[
\pi^*R\Ga_c(A,\cL)\cong \pr_{1!}\mu^*\cL\cong
\pr_{1!}(\cL\boxtimes\cL)\cong R\Ga_c(A,\cL)\tens\cL.
\]
Since $\cL$ is a nontrivial local system on $A$, this clearly forces
$R\Ga_c(A,\cL)=0$.
\end{proof}

\subsection{Proof of Proposition \ref{p:key}(b)}\label{ss:proof-p:key-b}
We now take $k=\bF_q$ and recall that $G'\subset G$ denotes the normalizer of the given admissible pair $(H,\cL)$ and $\sC$ denotes the geometric conjugacy class of $(H,\cL)$. We can identify $\sC$ with $(G/G')(\bF_q)$. Note also that $H^1(\bF_q,G)$ is trivial because $G$ is connected, so we can choose representatives $\bigl\{(H_\al,\cL_\al)\bigr\}_{\al\in H^1(\bF_q,G')}$ of the $G(\bF_q)$-orbits in $\sC$ such that the normalizers of $(H_\al,\cL_\al)$ are the inner forms $G'^\al\subset G$ of $G'$ (see Definition \ref{d:inner-form-subgroup} and Proposition \ref{p:alternative-inner}).

\mbr

By Definition \ref{d:adm-pair-L-packet}, the set $L(\sC)$ consists of all irreducible representations $\rho$ of $G(\bF_q)$ such that the function\footnote{Here we view $t_{\cL_\al}$ as a conjugation-invariant function on the finite group $G'^\al(\bF_q)$.} $t_{\cL_\al}$ acts nontrivially in the representation $\rho\bigl\lvert_{G'^{\al}(\bF_q)}$ for some $\al\in H^1(\bF_q,G')$. On the other hand,
Proposition \ref{p:induction-sheaves-functions} shows that $t_{e_{H,\cL}}=t_{\ig e'_\cL}$ is equal to the
character of the representation $\bigoplus\limits_{\al\in H^1(\bF_q,G')}
\Ind_{G'^{\al}(\bF_q)}^{G(\bF_q)} t_{\cL_\al}$ of $G(\bF_q)$. In view of Frobenius reciprocity and the definition of $L(e_{H,\cL})$, this implies that $L(\sC)=L(e_{H,\cL})$ and completes the proof of Proposition \ref{p:key}(b).



\appendix

\section*{Appendix A: Serre duality and bi-extensions}

\setcounter{section}{1}

In this appendix, which can (for the most part) be read
independently of the rest of the paper, we recall the classical
Serre duality theory \cite{serre,begueri} for connected commutative
unipotent groups, explain how to extend this theory to the case
where the commutativity assumption is dropped (following a
suggestion of Drinfeld), and establish a number of technical results
on Serre duality and skewsymmetric bi-extensions that are used in
the main body of the text. Our presentation closely follows
\cite{drinfeld-lectures}, and we verify a few of the statements
conjectured there.

\subsection{Prologue}\label{aa:app-prologue} If $G$ is
an algebraic group over a field $k$ and $\ell$ is a prime different
from $\operatorname{char}k$, we recall that a $\ql$-local system
$\cL$ on $G$ is said to be \emph{multiplicative} if
$\mu^*(\cL)\cong\cL\boxtimes\cL$, where $\mu:G\times_k G\rar{}G$ is
the multiplication morphism. This notion is a natural geometrization
of the notion of a homomorphism $\Ga\rar{}\ql^\times$, where $\Ga$ is an
abstract group. In the purely algebraic setting, the set of all such
homomorphisms is itself an abelian group, and this observation is
useful in the character theory of finite groups. It is natural to
ask whether this statement has a geometric analogue.

\mbr

In particular, we would like to construct a ``moduli space'' of
multiplicative $\ql$-local systems on $G$. We assume that $G$ is
connected: otherwise local systems on $G$ have nontrivial
automorphisms, and there is no convenient way to ``rigidify'' them. Moreover, if we want this moduli space to be
something resembling an algebraic group as well, it is not hard to
see \cite{intro} that $G$ must be unipotent. Next, if $G$ is a
unipotent group over a field of characteristic $0$, then every local
system on $G$ is constant, so we will assume that
$\operatorname{char}k=p>0$.

\mbr

In this case fix an injection of groups
$\psi:\qzp\into\ql^\times$. It identifies $\qzp$ with the group of
roots of unity in $\ql^\times$ whose order is a power of $p$, and
one easily checks (see Lemma
\ref{l:mult-loc-sys-unip-centr-ext-qzp}) that every multiplicative
$\ql$-local system on $G$ comes from a multiplicative $\qzp$-torsor
on $G$ (defined in an obvious manner). This observation allows us to
have a more natural theory which is independent of $\ell$.

\mbr

Next, even for connected $G$, multiplicative $\qzp$-torsors on $G$
are still not rigid, because being multiplicative is only a
property. To rigidify the situation we must look at multiplicative
$\qzp$-torsors $\cE$ on $G$ equipped with a trivialization of the
pullback $1^*\cE$, where $1:\Spec k\rar{}G$ is the multiplicative
identity. Giving such data is equivalent to giving a central
extension of $G$ by the discrete group $\qzp$ in the category of
group schemes over $k$. This is proved in \cite{masoud}. We find it
more natural, and technically much more convenient, to work with
central extensions of group schemes rather than multiplicative local
systems or torsors. Therefore the results of this appendix will
usually be phrased in the language of (bi-)extensions.

\mbr

Finally, we recall (see Remark \ref{r:need-perfect} below) that the
``moduli space'' of central extensions of $G$ by $\qzp$ can only be
canonically defined as a \emph{perfect} scheme.

\subsection{Organization}\label{aa:app-organization} The first
half of the appendix is devoted to the Serre duality\footnote{Not to
be confused with Serre duality in the theory of cohomology of
coherent sheaves.} for connected \emph{commutative} perfect
unipotent groups, the idea of which goes back to \cite{serre}. In
\S\ref{aa:perfect} we provide some background on perfect schemes,
perfect group schemes, and the perfectization functor, following
\cite{greenberg,drinfeld-lectures}. In \S\ref{aa:perfect-unipotent}
we define perfect quasi-algebraic groups and perfect unipotent
groups. In \S\ref{aa:Serre-duality} we recall the main statement of
the classical Serre duality theory following \cite{begueri}. In
\S\ref{aa:bi-ext} we recall Mumford's notion \cite{mumford} of a
bi-extension, and in \S\ref{aa:bi-ext-bi-mult} we relate it to the
notion of a ``bimultiplicative torsor''. In
\S\ref{aa:Serre-Pontryagin} we relate Serre duality for connected
commutative unipotent groups over finite fields to Pontryagin
duality for finite $p$-groups. In \S\ref{aa:pairing-bi-ext} we prove
a result on bi-extensions of commutative connected unipotent groups
which is used in the study of admissible pairs. Finally, in
\S\ref{aa:lagrangian} we prove the existence of ``almost
Lagrangian'' subgroups with respect to a skewsymmetric bi-extension
(defined in \S\ref{aa:symm-skew-symm}) of a commutative unipotent
group by $\qzp$, under suitable additional assumptions,
cf.~\cite{drinfeld-lectures}.

\mbr

The second half of the appendix discusses Serre duality for
\emph{noncommutative} groups. In \S\ref{aa:noncomm-Serre} we define
the Serre dual of any connected perfect unipotent group,
and in
\S\S\ref{aa:construction-alt}--\ref{aa:proof-l:vanishing-ext-2} we
establish the geometric analogues of certain standard constructions
and results on $1$-dimensional characters of abstract groups.

\subsection{Perfect schemes and group schemes}\label{aa:perfect} Fix a prime $p$. Let us
recall that a scheme $S$ in characteristic $p$, i.e., such that $p$
annihilates the structure sheaf $\cO_S$ of $S$, is said to be
\emph{perfect} if the morphism $\cO_S\rar{}\cO_S$, given by
$f\longmapsto f^p$ on the local sections of $\cO_S$, is an
isomorphism of sheaves. In particular, a commutative ring $A$ of
characteristic $p$ is perfect \cite{greenberg} if and only if $\Spec
A$ is a perfect scheme.

\mbr

Let $\mathfrak{Sch}_p$ denote the category of all $\bF_p$-schemes,
and let $\mathfrak{Perf}_p$ be the full subcategory of
$\mathfrak{Sch}_p$ formed by perfect schemes. The inclusion functor
$\mathfrak{Perf}_p\into\mathfrak{Sch}_p$ has a right adjoint which
we will call the \emph{perfectization functor}, and will denote by
$X\longmapsto X_{per\!f}$. We note that this functor was constructed
by M.J.~Greenberg in \cite{greenberg}, who denotes it by
$X\longmapsto X^{1/p^{\infty}}$, and calls $X^{1/p^{\infty}}$ the
\emph{perfect closure} of $X$.

\mbr

Next let $k$ be a perfect field of characteristic $p$, let
$\mathfrak{Sch}_k$ be the category of $k$-schemes and
$\mathfrak{Perf}_k$ the full subcategory consisting of perfect
schemes. The natural morphism $(\Spec k)_{per\!f}\rar{}\Spec k$ is
an isomorphism, so for any $X\in\mathfrak{Sch}_k$, the
perfectization $X_{per\!f}$ is automatically a scheme over $k$ (if
$k$ is not perfect, then $X_{per\!f}$ is a scheme over the perfect
closure of $k$). Hence $X\longmapsto X_{per\!f}$ can be upgraded to
a functor $\mathfrak{Sch}_k\rar{}\mathfrak{Perf}_k$, which is also
right adjoint to the natural inclusion.

\begin{rem}\label{r:perfect-group-scheme}
If $A$ and $B$ are perfect $k$-algebras, so is their tensor product
$A\tens_k B$. Indeed the $p$-th power homomorphism $A\tens_k
B\rar{}A\tens_k B$, $x\longmapsto x^p$, is the tensor product of the
corresponding homomorphisms $A\rar{}A$ and $B\rar{}B$. It follows
that the product of two perfect schemes over $k$ is perfect, so the
inclusion functor $\mathfrak{Perf}_k\into\mathfrak{Sch}_k$ preserves
products. Hence a group object in the category $\mathfrak{Perf}_k$
is automatically a group scheme over $k$ in the usual sense, which
is perfect as a scheme. In particular, the term ``perfect group
scheme over $k$'' is unambiguous.
\end{rem}

\begin{rem}
On the other hand, the perfectization functor
$\mathfrak{Sch}_k\rar{}\mathfrak{Perf}_k$ preserves limits by
abstract nonsense (because it has a left adjoint). In particular, if
$G$ is a group scheme over $k$, then $G_{per\!f}$ becomes a perfect
group scheme over $k$.
\end{rem}

\subsection{Perfect unipotent groups}\label{aa:perfect-unipotent}
Let us fix a perfect field $k$ of characteristic $p>0$. A perfect
scheme $Y$ over $k$ is said to be \emph{of quasi-finite type over
$k$} if it is isomorphic to $X_{per\!f}$ for a scheme $X$ of finite
type over $k$. We define a \emph{quasi-algebraic group} over $k$ to
be a perfect group scheme such that the underlying scheme is of quasi-finite type over $k$.

\mbr

The next result is not strictly necessary for what follows, but we
find it to be at least psychologically helpful.

\begin{lem}\label{l:quasi-algebraic-perfectization}
If $G$ is an affine quasi-algebraic group over $k$, then $G$ is
isomorphic to the perfectization
of an affine algebraic group over $k$.
\end{lem}
\begin{proof}{Proof}
In view of Remark \ref{r:perfect-group-scheme}, we have $G=\Spec A$,
where $A$ is a commutative Hopf algebra over $k$ which is perfect as
a ring. By assumption, there exists a finitely generated
$k$-subalgebra $B\subset A$ such that $A$ is the perfect closure
\cite{greenberg} of $B$. Every coalgebra over a field is the
filtered union of its finite dimensional sub-coalgebras, so there
is a finitely generated Hopf subalgebra $B'\subset A$ such that
$B\subset B'$. Then $A$ is the perfect closure of $B'$
as well, and $G'=\Spec B'$ is an affine algebraic group over $k$
(because $A$ is reduced), and $G\cong G'_{per\!f}$, as desired.
\end{proof}

\begin{defin}
A \emph{perfect unipotent group} over $k$ is a perfect group scheme
over $k$ which is isomorphic to the perfectization of a unipotent
algebraic group over $k$.
\end{defin}

The two basic examples of perfect unipotent groups over $k$ are the
discrete group $\bZ/p\bZ$ and the perfectization $\bG_{a,\,per\!f}$
of the additive group $\bG_a$. If $k=\overline{k}$, then
every connected perfect unipotent group over $k$ has a finite
filtration by closed normal subgroups with successive subquotients
isomorphic to $\gap$.

\mbr

We denote by $\cpu_k$ the category of all commutative perfect
unipotent groups over $k$, and by $\cpuc_k\subset\cpu_k$ the full
subcategory formed by connected group schemes. It is not hard to see
that $\cpu_k$ is an abelian category; in particular, for a morphism
$f:G\rar{}H$ in $\cpu_k$, we can talk about the kernel, $\Ker f$, of
$f$, and we have the notion of an exact sequence in $\cpu_k$.
Moreover, $\cpuc_k$ is an exact subcategory of $\cpu_k$.

\subsection{Classical Serre duality}\label{aa:Serre-duality} We
continue to work over a fixed perfect field $k$ of characteristic
$p>0$. If $G\in\cpuc_k$, we define a contravariant functor
\begin{equation}\label{e:serre-dual}
G^* : \mathfrak{Sch}_k \rar{} \bigl\{ \text{abelian groups} \bigr\},
\qquad S\longmapsto \Ext^1_S(G\times_k S,\bQ_p/\bZ_p),
\end{equation}
where $\Ext^1$ denotes the first $\Ext$ group computed in the
category of commutative group schemes over $S$, and $\bQ_p/\bZ_p$ is
viewed as a discrete group scheme over $S$. We call this functor the
\emph{Serre dual} of the group $G$.

\mbr

The idea of this construction goes back to Serre's article
\cite{serre}. However, in the form needed for our purposes, the
duality theory appears to be due to L.~Begueri:
\begin{thm}[\cite{begueri}]\label{t:serre-duality}
The restriction of the functor $G^*$ to the subcategory
$\mathfrak{Perf}_k$ is representable by an object of $\cpuc_k$,
which is also denoted by $G^*$. Moreover, the functor $G\longmapsto
G^*$ is an exact anti-auto-equivalence of the category $\cpuc_k$.
\end{thm}

\begin{rem}\label{r:need-perfect}
If $G$ is a connected commutative unipotent group over $k$ in the
usual sense, the natural morphism $G^*\rar{}(G_{per\!f})^*$ is an
isomorphism of functors on $\mathfrak{Sch}_k$. On the other hand, as
explained, e.g., in \cite{intro}, the functor $G^*$ is not
representable\footnote{However, it is ind-representable: see the
appendix on Serre duality in \cite{intro}.} on the whole category
$\mathfrak{Sch}_k$ already for $G=\bG_a$. This is the reason for
working with perfect group schemes in the context of Serre duality.
\end{rem}

\subsection{Bi-extensions}\label{aa:bi-ext} The notion of a
bi-extension of group schemes was discovered by Mumford in
\cite{mumford}, and later generalized by Grothendieck in SGA 7-1. It
can be formulated in several equivalent ways; the following approach
will be convenient for us. Let $G_1$, $G_2$ be group schemes over a
field $k$, and let $A$ be a commutative group scheme over $k$. A
\emph{bi-extension} of $(G_1,G_2)$ by $A$ is a scheme $E$ over $k$,
equipped with an action of $A$ and a morphism
$\pi:E\rar{}G_1\times_k G_2$ which makes $E$ an $A$-torsor over
$G_1\times_k G_2$, together with the following additional
structures. \sbr
\begin{enumerate}[(a)]
\item Choices of sections of $\pi$ along $\{1\}\times G_2$ and $G_1\times\{1\}$, by
means of which the ``slices'' $\pi^{-1}\bigl(\{1\}\times G_2\bigr)$
and $\pi^{-1}\bigl(G_1\times\{1\}\bigr)$ will be identified with
$A\times_k G_2$ and $G_1\times_k A$, respectively, where $1$ denotes
the unit in $G_1$ or $G_2$.
 \sbr
\item A morphism $\bullet_1:E\times_{G_2}E\rar{}E$ which makes $E$ a
group scheme over $G_2$ and makes $\pi$ a central extension of
$G_1\times_k G_2$, viewed as a group scheme over $G_2$, by
$A\times_k G_2$, in a way compatible with the identification
$A\times_k G_2\cong\pi^{-1}\bigl(\{1\}\times G_2\bigr)$.
 \sbr
\item A morphism $\bullet_2:E\times_{G_1}E\rar{}E$ which makes $E$ a
group scheme over $G_1$ and makes $\pi$ a central extension of
$G_1\times_k G_2$, viewed as a group scheme over $G_1$, by
$G_1\times_k A$, in a way compatible with the identification
$G_1\times_k A\cong\pi^{-1}\bigl(G_1\times\{1\}\bigr)$. \sbr
\end{enumerate}
These data are required to satisfy the following compatibility
condition: if $T$ is any $k$-scheme and
$e_{11},e_{12},e_{21},e_{22}\in E(T)=\Hom_{k-schemes}(T,E)$, then
\[
(e_{11}\bullet_2 e_{12}) \bullet_1 (e_{21}\bullet_2 e_{22}) =
(e_{11}\bullet_1 e_{21}) \bullet_2 (e_{12}\bullet_1 e_{22})
\]
whenever both sides of this equality are defined, i.e., whenever
\[
\pi(e_{11})=(g_1,g_2), \quad \pi(e_{12})=(g_1,g_2'), \quad
\pi(e_{21})=(g_1',g_2), \quad \pi(e_{22})=(g_1',g_2')
\]
for some $g_1,g_1'\in G_1(T)$ and $g_2,g_2'\in G_2(T)$.

\begin{defin}\label{d:trivialization-bi-ext}
The notion of an \emph{isomorphism} of bi-extensions is defined in
the obvious way, and bi-extensions of $(G_1,G_2)$ by $A$ form a
groupoid which we denote by $\biext(G_1,G_2;A)$. (It is even a
strictly commutative Picard groupoid, see \cite{drinfeld-lectures},
but we will not use this fact.) A \emph{trivial} bi-extension is one
which is isomorphic to $A\times_k G_1\times_k G_2$ equipped with the
obvious $A$-action, the natural projection $A\times_k G_1\times_k
G_2\rar{}G_1\times_k G_2$, and the obvious partial group laws coming
from the group law on $A$. Equivalently, a bi-extension $E$ as above
is trivial if it has a \emph{trivialization}, i.e., a
bimultiplicative section $\sg:G_1\times_k G_2\rar{}E$ of $\pi$,
which means that for any $k$-scheme $T$ and any choice of
$g_1,g_1'\in G_1(T)$ and $g_2,g_2'\in G_2(T)$, we have
$\sg(g_1g_1',g_2)=\sg(g_1,g_2)\bullet_1\sg(g_1',g_2)$ and
$\sg(g_1,g_2g_2')=\sg(g_1,g_2)\bullet_2\sg(g_1,g_2')$.
\end{defin}

\begin{rem}[\cite{drinfeld-lectures}]\label{r:trivializations-bi-ext}
Bi-extensions are to
central extensions as bimultiplicative maps are to homomorphisms,
where, for abstract groups $\Ga_1,\Ga_2,A$, we say that a map
$\be:\Ga_1\times\Ga_2\rar{}A$ is bimultiplicative if $\be(\ga_1,-)$
is a homomorphism for every fixed $\ga_1\in\Ga_1$, and
$\be(-,\ga_2)$ is a homomorphism for every fixed $\ga_2\in\Ga_2$.
This analogy manifests itself in many different ways.

\mbr

For instance,
if $G$ is a group scheme over $k$ and $A$ is a commutative group
scheme over $k$, then the group of automorphisms of any central
extension of $G$ by $A$ is naturally isomorphic to the group
$\Hom(G,A)$ of morphisms of $k$-group schemes $G\rar{}A$.
Consequently, for any trivial central extension of $G$ by $A$, its
trivializations form a torsor under $\Hom(G,A)$. Similarly, if $G_1$
and $G_2$ are group schemes over $k$, then the group of
automorphisms of any bi-extension of $(G_1,G_2)$ by $A$ is naturally
isomorphic to the group of bi-multiplicative morphisms of
$k$-schemes $G_1\times_k G_2\rar{}A$, and hence trivializations of
any trivial bi-extension of $(G_1,G_2)$ by $A$ form a torsor under
the latter group.
\end{rem}

\begin{cor}\label{c:bi-ext-connected}
If $A$ is a discrete commutative group and $G_1$, $G_2$ are group
schemes over $k$, at least one of which is a connected algebraic or
quasi-algebraic group, then bi-extensions of $(G_1,G_2)$ by $A$ have
no non-trivial automorphisms. In particular, every such bi-extension
has at most one trivialization.
\end{cor}

\begin{defin}\label{d:comm-bi-ext}
A bi-extension $E$ of $(G_1,G_2)$ by $A$ is said to be
\emph{commutative} if the two partial group laws on $E$ are
commutative. Of course, this can only happen if $G_1$ and $G_2$ are
commutative group schemes. In \cite{mumford}, Mumford imposes the
commutativity requirement in the very definition of a bi-extension.
\end{defin}

\subsection{Bi-extensions and bimultiplicative
torsors}\label{aa:bi-ext-bi-mult} Some of the data and the
compatibility conditions that appear in the definition of a
bi-extension can often be ignored. To explain this comment, let us
introduce the notion of a bimultiplicative torsor, by analogy with
the notion of a multiplicative torsor.

\begin{defin}
In the situation above, let \[\mu_1:G_1\times_k G_1\rar{}G_1
\qquad\text{and}\qquad \mu_2:G_2\times_k G_2\rar{}G_2\] be the
multiplication morphisms, and let \[\pr_{13},\pr_{23}:G_1\times_k
G_1\times_k G_2\rar{}G_1\times_k G_2\] and
\[\pr_{12},\pr_{13}:G_1\times_k G_2\times_k G_2\rar{}G_1\times_k
G_2\] be the projections. An $A$-torsor $\cE$ on $G_1\times_k G_2$
is said to be \emph{bimultiplicative} if \[
(\mu_1\times\id_{G_2})^*(\cE) \cong
\pr_{13}^*(\cE)\tens\pr_{23}^*(\cE) \qquad\text{as } A-\text{torsors
on } G_1\times_k G_1\times_k G_2
\]
and
\[
(\id_{G_1}\times\mu_2)^*(\cE) \cong
\pr_{12}^*(\cE)\tens\pr_{13}^*(\cE) \qquad\text{as } A-\text{torsors
on } G_1\times_k G_2\times_k G_2.
\]
\end{defin}

It is clear that a bi-extension $E$ of $(G_1,G_2)$ by $A$ determines
a bimultiplicative $A$-torsor on $G_1\times_k G_2$ by forgetting the
partial group laws on $E$. The proof of the following analogue of
Lemma \ref{l:mult-tors-central-exts} is straightforward, and is
therefore omitted.

\begin{lem}\label{l:bi-ext=bi-mult-tors}
Let $G_1$ and $G_2$ be connected algebraic or quasi-algebraic groups
over a field $k$, and let $A$ be an abstract commutative group,
viewed as a discrete group scheme over $k$. Consider the groupoid
$\sG$ of bimultiplicative $A$-torsors $\cE$ on $G_1\times_k G_2$
equipped with a trivialization of $(1\times 1)^*\cE$. Then the
forgetful functor $\biext(G_1,G_2;A)\rar{}\sG$ is an equivalence of
categories, and both groupoids are discrete. Furthermore, if $G_1$
and $G_2$ are commutative, then every bi-extension of $(G_1,G_2)$ by
$A$ is automatically commutative as well.
\end{lem}

\subsection{Serre duality and Pontryagin
duality}\label{aa:Serre-Pontryagin} In the remainder of the
appendix, the only bi-extensions that we consider will be
bi-extensions of connected unipotent groups by $\qzp$. Let us fix a
perfect field $k$ of characteristic $p>0$ and an object
$G\in\cpuc_k$. If $G^*\in\cpuc_k$ is the Serre dual
(\S\ref{aa:Serre-duality}) of $G$, then, by definition, we have a
central extension $\cU$ of $G\times_k G^*$ by $\qzp$ in the category
of group schemes over $G^*$, which is universal in the obvious
sense. The following claim is readily verified:
\begin{lem}\label{l:homomorphisms-biextensions}
Let $A$ be a perfect group scheme over $k$, and let $f:A\rar{}G^*$
be an arbitrary morphism of $k$-schemes. Then $f$ is a morphism of
group schemes if and only if $(\id_G\times f)^*\cU$ has a structure
of a bi-extension of $(G,A)$ by $\qzp$, compatible with its
structure of a central extension of $G\times_k A$ by $\qzp$ as an
$A$-group scheme.
\end{lem}

In particular, we see that $\cU$ itself comes from a (unique)
bi-extension of $(G,G^*)$ by $\qzp$, which we will also denote by
$\cU$. Furthermore, morphisms of $k$-group schemes $f:A\rar{}G^*$
correspond bijectively with bi-extensions of $(G,A)$ by $\qzp$.

\mbr

In this subsection we assume that $k=\bF_q$ is finite. Then
$G(\bF_q)$ is a finite abelian $p$-group, and we can consider its
Pontryagin dual, which can be canonically defined as
$G(\bF_q)^*=\Hom(G(\bF_q),\qzp)$. It is natural to ask about the
relationship between $G^*(\bF_q)$ and $G(\bF_q)^*$. As we will see
shortly, these two groups are canonically isomorphic. First, note
that we have an analogue of the sheaves-to-functions correspondence
(\S\ref{ss:sheaves-to-functions}) in the context of $\qzp$-torsors.
Namely, isomorphism classes of $\qzp$-torsors on $\Spec\bF_q$ are in
natural bijection with continuous homomorphisms
$\phi:\Gal(\bF/\bF_q)\rar{}\qzp$. In turn, such homomorphisms are in
bijection with elements of $\qzp$, via $\phi\longmapsto\phi(F_q)$,
where $F_q\in\Gal(\bF/\bF_q)$ is the geometric Frobenius. Now if $X$
is an arbitrary scheme over $\bF_q$ and $\cE$ is a $\qzp$-torsor
over $X$, we obtain a functor $t_\cE:X(\bF_q)\rar{}\qzp$, defined by
sending $x\in X(\bF_q)$ to the element of $\qzp$ corresponding to
the $\qzp$-torsor $x^*\cE$ over $\Spec\bF_q$.

\begin{prop}\label{p:Serre-Pontryagin}
If $G$ is a perfect connected commutative unipotent group over
$\bF_q$, $G^*$ is its Serre dual, and $\cU$ is the universal
bi-extension of $(G,G^*)$ by $\qzp$, then the map
\[
G(\bF_q)\times G^*(\bF_q) = (G\times G^*)(\bF_q) \xrar{\ \ t_\cU\ \
} \qzp
\]
is a perfect pairing, i.e., it is bi-additive and identifies
$G^*(\bF_q)$ with $G(\bF_q)^*$.
\end{prop}

\begin{proof}{Proof}
The only nontrivial part is to show that every $\chi\in G(\bF_q)^*$
can be represented as $x\longmapsto t_\cU(x,y)$ for some $y\in
G^*(\bF_q)$. By Lang's theorem \cite{lang}, we have an exact
sequence of perfect commutative unipotent groups over $\bF_q$,
\begin{equation}\label{e:lang}
0 \rar{} G(\bF_q) \rar{} G \rar{L} G \rar{} 0,
\end{equation}
where $L:G\rar{}G$, $g\longmapsto \Phi(g)g^{-1}$ is the Lang
isogeny\footnote{Here, $\Phi:G\rar{}G$ is the absolute Frobenius
morphism, defined as the identity map on the underlying set of $G$,
and the map $f\longmapsto f^q$ on local sections of the structure
sheaf $\sO_G$ of $G$.}. Choose a homomorphism
$\chi:G(\bF_q)\rar{}\qzp$. The pushforward of \eqref{e:lang} by
$\chi$ is an extension
\[
0 \rar{} \qzp \rar{} \Gt_\chi \rar{} G \rar{} 0,
\]
which defines an element $y\in G^*(\bF_q)$. One checks easily
(cf.~\cite{sga4.5}, \emph{Sommes. trig.}) that the function
$t_{\Gt_\chi}:G(\bF_q)\rar{}\qzp$, defined by $\Gt_\chi$ (viewed as
a $\qzp$-torsor over $G$), is equal to $\chi$. Hence
$\chi(x)=t_{\cU}(x,y)$ for all $x\in G(\bF_q)$.
\end{proof}

\subsection{Canonical pairing associated to a
bi-extension}\label{aa:pairing-bi-ext} Let us fix two objects
$G_1,G_2\in\cpuc_k$ and a morphism $f:G_1\rar{}G_2$. Since
$(G_2^*)^*$ is canonically identified with $G_2$ by Theorem
\ref{t:serre-duality}, we see from Lemma
\ref{l:homomorphisms-biextensions} that $f$ corresponds to a
bi-extension of $(G_2^*,G_1)$ by $\qzp$. Consider the dual morphism
$f^*:G_2^*\rar{}G_1^*$. The kernels $\Ker f$ and $\Ker f^*$ are
possibly disconnected objects of $\cpu_k$.

\mbr

\emph{Until the end of the subsection we assume that $k$ is an
algebraically closed field of characteristic $p>0$.} Then the groups
of connected components $\pi_0(\Ker f)$ and $\pi_0(\Ker f^*)$ are
finite discrete abelian $p$-groups. Our goal is to define, following
\cite{drinfeld-lectures}, a canonical nondegenerate pairing of
abelian $p$-groups
\begin{equation}\label{e:pairing-bi-ext}
B_f : \pi_0(\Ker f^*) \times \pi_0(\Ker f) \rar{} \qzp,
\end{equation}
i.e., a bi-additive map inducing an isomorphism of abelian groups
\begin{equation}\label{e:B-f-induced}
\pi_0(\Ker f^*)\rar{}\Hom(\pi_0(\Ker f),\qzp).
\end{equation}

\mbr

The definition of \eqref{e:pairing-bi-ext} is simple. Since
$f\bigl\lvert_{\Ker f}=0$, the restriction
$E\bigl\lvert_{G_2^*\times(\Ker f)}$ is a trivial bi-extension.
Since $G_2^*$ is connected, there is only one trivialization of
$E\bigl\lvert_{G_2^*\times(\Ker f)}$ by Corollary
\ref{c:bi-ext-connected}. Similarly, the bi-extension
$E\bigl\lvert_{(\Ker f^*)\times G_1}$ has a unique trivialization.
Thus we obtain \emph{two} trivializations of $E\bigl\lvert_{(\Ker
f^*)\times(\Ker f)}$, which, by Remark
\ref{r:trivializations-bi-ext}, must ``differ'' by a bi-additive
morphism $(\Ker f^*)\times(\Ker f)\rar{}\qzp$. Since $\qzp$ is
discrete, the latter factors through a bi-additive map
\eqref{e:pairing-bi-ext}.

\begin{prop}\label{p:pairing-bi-ext-nondeg}
The pairing \eqref{e:pairing-bi-ext} we just defined is
nondegenerate.
\end{prop}

To prove this proposition, it suffices to show the injectivity of
the induced homomorphism \eqref{e:B-f-induced}, for then we can
apply the same argument replacing $f$ with $f^*$. Thus let
$g\in(\Ker f^*)(k)$ be such that $B_f(\overline{g},x)=0$ for all
$x\in\pi_0(\Ker f)$, where $\overline{g}$ denotes the image of $g$
in $\pi_0(\Ker f^*)$. We must show that $g\in(\Ker f^*)^\circ$. To
this end we need to obtain a more concrete description of $(\Ker
f^*)^\circ$.

\mbr

Consider an arbitrary extension of commutative group schemes over
$k$,
\[
0 \rar{} \qzp \rar{} \widetilde{G_2} \rar{\eta} G_2 \rar{} 0.
\]
Let $t\in G_2^*(k)$ denote the corresponding element, and assume
that $t\in\Ker f^*$, which means that there exists a morphism
$\widetilde{f}:G_1\rar{}\widetilde{G_2}$ such that
$\eta\circ\widetilde{f}=f$. Since $G_1$ is connected and $\qzp$ is
discrete, $\widetilde{f}$ is unique. Moreover, it takes $\Ker f$ to
$\qzp$, so we obtain an induced homomorphism
$t':\pi_0(\Ker f)\rar{}\qzp$.

\mbr

The following claim is simply a matter of unwinding the definitions:

\begin{lem}
We have $t'(x)=B_f(\overline{t},x)$ for all $x\in\pi_0(\Ker f)$.
\end{lem}

In view of this result, it is clear that Proposition
\ref{p:pairing-bi-ext-nondeg} follows from

\begin{lem}\label{l:auxiliary-p:pairing-bi-ext-nondeg}
In the situation above, the following are equivalent:
\begin{enumerate}[$($i$)$]
\item $t$ belongs to the neutral connected component $(\Ker f^*)^\circ$;
\item the restriction of $t$ to the image of $f$ is trivial;
\item the homomorphism $t'$ is identically zero.
\end{enumerate}
\end{lem}
\begin{proof}{Proof}
The equivalence between conditions (ii) and (iii) is obvious.
Indeed, since we are already assuming that $f^*(t)=0$, it is clear
that the restriction of $t$ to the image of $f$ is trivial if and
only if the homomorphism $\widetilde{f}$ defined above vanishes on
$\Ker f$, which is in turn equivalent to the vanishing of $t'$.

\mbr

For the equivalence between (i) and (ii), note first that the
composition of $f$ with the quotient map $G_2\onto G_2/f(G_1)$
equals zero, hence the composition \[(G_2/f(G_1))^*\into
G_2^*\rar{f^*}G_1^*\] also equals zero. Since $(G_2/f(G_1))^*$ is
connected, we see that (ii) implies (i). The argument of the
previous paragraph shows that we have an exact sequence
\[
0 \rar{} (G_2/f(G_1))^*(k) \rar{} (\Ker f^*)(k) \rar{}
\Hom(\pi_0(\Ker f),\qzp),
\]
and the last group is finite. Thus $(G_2/f(G_1))^*$ maps
isomorphically onto $(\Ker f^*)^\circ$, whence (i) also implies (ii)
and the proof of the lemma is complete.
\end{proof}

\subsection{Symmetric and skewsymmetric
bi-extensions}\label{aa:symm-skew-symm} Let $k$ be a perfect field
of characteristic $p>0$, let $G\in\cpuc_k$, and let
$E\rar{\pi}G\times_k G$ be a bi-extension of $(G,G)$ by $\qzp$. As
explained above, $E$ determines (and is determined by) a
homomorphism $f:G\rar{}G^*$. We will think of its dual, $f^*$, as a
homomorphism $G\rar{}G^*$ as well, using the canonical
identification $(G^*)^*\cong G$ (see Theorem \ref{t:serre-duality}).

\begin{defin}\label{d:symm-skew-symm}
The bi-extension $E$ is said to be \emph{symmetric} if
$\tau^*(E)\cong E$, where $\tau:G\times_k G\rar{}G\times_k G$ is the
transposition of the two factors. Equivalently, $E$ is symmetric if
$f^*=f$. We say that $E$ is \emph{skewsymmetric}\footnote{The term
``alternating'' might be more appropriate, but we decided to be
consistent with \cite{drinfeld-lectures}.} if the restriction of $E$
to the diagonal in $G\times_k G$ is a trivial $\qzp$-torsor.
\end{defin}

\begin{rem}
In general, skewsymmetry is not a property of a bi-extension;
rather, one needs to define an extra structure, called a
\emph{skewsymmetry constraint} in \cite{drinfeld-lectures}.
However, since $G$ is connected and $\qzp$ is discrete, being
skewsymmetric becomes a property in our situation (similarly to
Lemma \ref{l:bi-ext=bi-mult-tors}).
\end{rem}

\begin{rem}
If the bi-extension $E$ considered above is skewsymmetric, it is
easy to check that $f=-f^*$. The converse statement holds if (and
only if) $p>2$.
\end{rem}

\begin{lem}[``Parity change'']\label{l:parity-change}
In the situation of this subsection, assume that $k$ is
algebraically closed, and consider the canonical pairing
\eqref{e:pairing-bi-ext} defined in \S\ref{aa:pairing-bi-ext}. If
$E$ is symmetric $($respectively, skewsymmetric$)$, so that, in
particular, $\Ker f=\Ker f^*$, then $B_f$ is alternating\footnote{In
other words, $B_f(x,x)=0$ for all $x\in\pi_0(\Ker f)=\pi_0(\Ker
f^*)$.} $($respectively, symmetric$)$.
\end{lem}

This lemma is proved at the end of the subsection. In the
skewsymmetric case, we can in fact prove a more precise result (see
Lemma \ref{l:skewsym=>nondeg-quadr-form} below). To state it we
introduce some notation. Because we are working in the commutative
situation, we will denote the partial group laws on $E$ by $+_1$ and
$+_2$, as opposed to $\bullet_1$ and $\bullet_2$. The group laws on
$G$ and $A$ will also be written additively. Since $E$ is either
symmetric or skewsymmetric, we have $\Ker f=\Ker f^*$, and thus we
have unique trivializations $\la:(\Ker f)\times_k G\rar{}E$ and
$\rho:G\times_k(\Ker f)\rar{}E$ (cf.~Definition
\ref{d:trivialization-bi-ext}), as in \S\ref{aa:pairing-bi-ext}.

\mbr

With this notation, the bi-additive map \eqref{e:pairing-bi-ext} is
explicitly defined by
\[
\la(x,y) = B_f(x,y) + \rho(x,y) \qquad \forall\, x,y\in\Ker f.
\]
If, in addition, $E$ is skewsymmetric, then, by Definition
\ref{d:symm-skew-symm}, we have a unique morphism
$\widetilde{\De}:G\rar{}E$ such that $\widetilde{\De}(0)=0$ and
$\pi\circ\widetilde{\De}$ is the diagonal morphism
$\De:G\rar{}G\times_k G$. In this case we define a morphism $q:\Ker
f\rar{}\qzp$ by
\[
\la(x,x)=q(x)+\widetilde{\De}(x) \qquad \forall\, x\in\Ker f.
\]

\begin{lem}\label{l:skewsym=>nondeg-quadr-form}
Assume that $k$ is algebraically closed, and that $E$ is
skewsymmetric. Then $q$ descends to a map $q:\pi_0(\Ker
f)\rar{}\qzp$ satisfying the following identities:
\begin{enumerate}[$(1)$]
\item $q(nx) = n^2\cdot q(x)$ for all $n\in\bZ$ and all
$x\in\pi_0(\Ker f)$; and
\item $B_f(x,y) = q(x+y)-q(x)-q(y)$ for all $x,y\in\pi_0(\Ker f)$.
\end{enumerate}
In particular, $q$ is a ``nondegenerate $\qzp$-valued quadratic
form'' on $\pi_0(\Ker f)$.
\end{lem}

\begin{proof}{Proof}
(1) Given $n\in\bZ$, let $e\longmapsto n*_1e$ and $e\longmapsto
n*_2e$ denote the ``multiplication by $n$'' action on $E$ obtained
from the first and the second partial group laws, respectively.
Consider the identity $\widetilde{\De}(n\cdot
x)=n*_1(n*_2\widetilde{\De}(x))$ for $x\in G$. It holds for $x=0$ by
construction, and becomes true after we apply $\pi$ to both sides.
Hence this identity holds for all $x$, by continuity. On the other
hand, since $\la$ is bi-additive, we have $\la(n\cdot x,n\cdot
x)=n*_1(n*_2\la(x,x))$ for all $x\in\Ker f$, proving (1).

\mbr
 \noindent
(2) We take $x,y\in(\Ker f)(k)$. By construction, we have
\begin{equation}\label{e:sum-def}
\la(x+y,x+y)=q(x+y)+\widetilde{\De}(x+y).
\end{equation}
On the other hand, using the bi-additivity of $\la$, we find
\begin{equation}\label{e:sum-expand}
\begin{split}
\la(x+y,x+y) &= \bigl( \la(x,x) +_2 \la(x,y) \bigr) +_1 \bigl(
\la(y,x) +_2 \la(y,y) \bigr) \\
&= q(x) + q(y )+ \bigl( \widetilde{\De}(x) +_2 \la(x,y) \bigr) +_1
\bigl( \la(y,x) +_2 \widetilde{\De}(y) \bigr)
\end{split}
\end{equation}
Comparing \eqref{e:sum-def} and \eqref{e:sum-expand}, we see that
(2) reduces to verifying the identity
\[
\widetilde{\De}(x+y) = \bigl( \widetilde{\De}(x) +_2 \rho(x,y)
\bigr) +_1 \bigl( \la(y,x) +_2 \widetilde{\De}(y) \bigr).
\]
However, the last identity is meaningful for all $(x,y)\in
G(k)\times(\Ker f)(k)$. Moreover, it is satisfied by continuity,
because it holds whenever $x=0$, because $G$ is connected, and
because it becomes true after we apply $\pi$ to both sides.
\end{proof}

\begin{proof}{Proof of Lemma \ref{l:parity-change}}
If $E$ is skewsymmetric, the symmetry of $B_f$ follows from Lemma
\ref{l:skewsym=>nondeg-quadr-form}(2). Now assume that $E$ is
symmetric and use the notation introduced above. Let
$\widetilde{\tau}:E\rar{\simeq}E$ be the isomorphism of $k$-schemes
which induces an isomorphism of bi-extensions between $\tau^*(E)$
and $E$. In other words, $\pi\circ\widetilde{\tau}=\tau\circ\pi$,
and $\widetilde{\tau}$ interchanges the two partial group laws on
$E$. Uniqueness of trivializations implies that
$\la(x,y)=\widetilde{\tau}(\rho(y,x))$ for all $(x,y)\in(\Ker
f)(k)\times G(k)$. On the other hand, by continuity,
$\widetilde{\tau}$ is the identity above the diagonal in $G\times_k
G$. Thus $\la(x,x)=\rho(x,x)$ for all $x\in(\Ker f)(k)$, which means
that $B_f(x,x)=0$, i.e., $B_f$ is alternating.
\end{proof}

\subsection{Lagrangian subgroups}\label{aa:lagrangian}
Every finite dimensional vector space $V$ equipped with
an alternating bilinear form $\om$ admits a Lagrangian subspace,
i.e., a subspace $L\subset V$ which coincides with its annihilator
with respect to $\om$. In particular, the rank of
$\om$ is even.

\mbr

Let us study the geometric analogue of this
statement. Fix $G\in\cpuc_k$ and a skewsymmetric bi-extension $E$ of
$(G,G)$ by $\qzp$. Let $f:G\rar{}G^*$ denote the corresponding
morphism. The \emph{rank} of $E$ is defined to be $\rk E=\dim f(G)$.

\begin{defin}\label{d:lagrangian}
An \emph{almost Lagrangian} subgroup of $G$ with respect to $E$ is a
closed connected subgroup $L\subset G$ such that $L\subset
f^{-1}(\Ann(L))$, and $L$ has finite index in $f^{-1}(\Ann(L))$,
where $\Ann(L)=\Ker(G^*\onto L^*)\subset G^*$ is the annihilator of
$L$ in $G^*$. We say that $L$ is simply \emph{Lagrangian} if
$f^{-1}(\Ann(L))=L$.
\end{defin}

Note that if $G$ has an almost Lagrangian subgroup with respect to
$E$, then $\rk E$ must be even. However, $\rk E$ is not always even
in the geometric setup: for instance, there exist nontrivial
skewsymmetric bi-extensions of $(\gap,\gap)$ by $\qzp$. This is the
first obstruction to the existence of almost Lagrangian subgroups.
If the base field $k$ is algebraically closed, this is the only
obstruction. The non-algebraically closed case appears to be more
intricate. Part (a) of the following result was conjectured in
\cite{drinfeld-lectures}.

\begin{prop}\label{p:existence-lagr}
\begin{enumerate}[$($a$)$]
\item If $k$ is algebraically closed and $\rk E$ is even, $G$ has an almost Lagrangian subgroup with
respect to $E$.
\item Allow $k$ to be merely perfect, and suppose that $E$ is nondegenerate
in the strong\footnote{The ``weak'' nondegeneracy condition is that
$f$ is an isogeny.} sense, i.e., $f:G\rar{}G^*$ is an isomorphism.
Then $\dim G$ is even and every almost Lagrangian subgroup of
$G$ with respect to $E$ is Lagrangian.
\end{enumerate}
\end{prop}

Before proving this proposition we need to understand Serre duality
more explicitly in a special case. It is well known that the objects
of $\cpuc_k$ that are annihilated by $p$ are isomorphic to direct
sums of copies of $\gap$. On the other hand, we have
\begin{equation}\label{e:end-gap}
\End_{\cpuc_k}(\gap) \cong R:=k\{\tau,\tau^{-1}\},
\end{equation}
the ring of twisted Laurent polynomials, determined by $\tau\cdot
c=c^p\cdot\tau$ for $c\in k$. The isomorphism \eqref{e:end-gap} is
obtained by letting elements of $k$ act on $\gap$ by dilations, and
letting $\tau$ act by $x\longmapsto x^p$.

\mbr

To describe the Serre duality functor on the subcategory of
$p$-torsion objects of $\cpuc_k$, we only need to explain how it
acts on $\gap$ and on endomorphisms of $\gap$. This was done in
\cite{intro}. One can identify $\gap^*$ with $\gap$ so that the
action of the Serre duality functor on endomorphisms of $\gap$
becomes the anti-involution $f\longmapsto f^*$ of the ring $R$
determined by $c^*=c$ for $c\in k$, and $\tau^*=\tau^{-1}$.

\mbr

Below we will prove the following purely algebraic result:

\begin{lem}\label{l:gap-2-skewsymm}
Let $k$ be algebraically closed, and let $f:\gap^2\rar{}(\gap^2)^*$
be a morphism of $k$-group schemes such that $f^*=-f$. Then there
exists a nonzero morphism $\al:\gap\rar{}\gap^2$ such that
$\al^*\circ f\circ \al=0$.
\end{lem}

Let us explain why this lemma implies Proposition
\ref{p:existence-lagr}.

\mbr

For part (a) of the proposition, we use induction on $\dim G$. What
allows us to reduce $\dim G$ is the construction of subquotients of
$E$. Namely, if $H\subset G$ is a closed connected \emph{isotropic}
subgroup, i.e., such that $H\subset H^{\perp}:=f^{-1}(\Ann(H))$,
then $E$ induces a skewsymmetric bi-extension $E'$ of
$G':=(H^\perp)^\circ/H$ by $\qzp$, and if $L'$ is an almost
Lagrangian subgroup of $G'$ with respect to $E'$, then the preimage
of $L'$ in $H^\perp$ is an almost Lagrangian subgroup of $G$ with
respect to $E$. Furthermore, $\rk E'$ has the same parity as $\rk
E$, and $\dim G'<\dim G$ so long as $H$ is nontrivial.

\mbr

First we reduce to the case $p\cdot G=0$. If $p\cdot G\neq 0$, let
$n\geq 2$ be the smallest integer for which $p^n\cdot G=0$ (it
exists because $G$ is a commutative perfect unipotent group). Then
$p^{n-1}\cdot G$ is a nontrivial connected isotropic subgroup of $G$
with respect to $E$, so we are done by induction on $\dim G$, as
explained above.

\mbr

Now assume that $p\cdot G=0$, so that $G\cong\gap^d$ for some
$d\in\bN$. If $d=1$, then $E=0$ because $\rk E$ is even by
assumption, so $G$ is almost Lagrangian in itself. Otherwise, let
$G''\subset G$ be a closed subgroup isomorphic to $\gap^2$. By Lemma
\ref{l:gap-2-skewsymm}, $G''$ has a nontrivial connected isotropic
subgroup $H$ with respect to $E\bigl\lvert_{G''\times_k G''}$. Then
$H$ is also isotropic in $G$, and we are again done by induction.

\mbr

Next we prove part (b) of the proposition. The second statement is
obvious: if $L\subset G$ is any connected subgroup, then
$\Ann(L)\subset G^*$ is also connected, whence
$L^\perp=f^{-1}(\Ann(L))$ is connected because $f$ is an
isomorphism. Thus $L$ is Lagrangian if and only if it is almost
Lagrangian. For the first statement, use the base change from $k$ to
an algebraic closure of $k$ (this changes neither the nondegeneracy
property of $E$ nor $\dim G$). Now let us try to repeat the same
inductive argument as above to prove that $G$ has an almost
Lagrangian subgroup with respect to $E$, which will imply that $\dim
G$ is even. Note that if $H\subset G$ is a closed connected subgroup
which is isotropic with respect to $E$, then the induced
skewsymmetric bi-extension $E'$ of $H^\perp/H$ by $\qzp$ is also
nondegenerate. Thus the only place where we cannot repeat the same
argument as above is when $G=\gap$.

\mbr

However, we claim that $\gap$ cannot have nondegenerate
skewsymmetric bi-extensions by $\qzp$. Indeed, consider a morphism
$f:\gap\rar{}\gap^*$ defining a skewsymmetric bi-extension. We
identify $\gap^*$ with $\gap$ as explained above, so that $f$
becomes an element of the ring $R=k\{\tau,\tau^{-1}\}$. It is easy
to check that being skewsymmetric becomes the property that $f$ can
be written as a sum of elements of the form $\tau^j\cdot c -
c\cdot\tau^{-j}$, where $j\in\bN$ and $c\in k$. However, no such
element is invertible in $R$, since all units of $R$ are of the form
$c\cdot\tau^i$, for $c\in k^\times$ and $i\in\bZ$.

\mbr

Thus we have proved both parts of Proposition
\ref{p:existence-lagr}.

\begin{proof}{Proof of Lemma \ref{l:gap-2-skewsymm}}
Our argument uses dimension counting, which is why we need to assume
that $k$ is algebraically closed. Using the identification of
$\gap^*$ with $\gap$, we can represent the morphism $f$ by a
two-by-two matrix
\[
\matr{a}{b}{-b^*}{d}, \qquad \text{where } a,b,d\in R \text{ and }
a=-a^*,\ d=-d^*.
\]
An arbitrary morphism $\al:\gap\rar{}\gap^2$ can be represented by a
vector $(x,y)\in R^2$, and then the element $\al^*\circ
f\circ\al\in\End_{\cpuc_k}(\gap)$ equals
\begin{equation}\label{e:F}
F(x,y) := x^* a x - y^* b^* x + x^* b y + y^* d y.
\end{equation}
We must show that in this situation there exist $x,y\in R$, not both
zero, such that $F(x,y)=0$. Note that we have $F(x,y)^*=-F(x,y)$ for
all $x,y\in R$.

\mbr

In what follows we will view $R$ as a vector space over $k$ with
respect to the action of $k$ on $R$ by left multiplication. For each
$N\in\bN$, let $R_N$ be the subspace of $R$ spanned over $k$ by
$\bigl\{\tau^{-N},\tau^{-N+1},\dotsc,\tau^N\bigr\}$. It has
dimension $2N+1$.

\mbr

Let $R^{skew}_N$ denote the subset of $R_N$ consisting of the
elements $z$ satisfying $z^*=-z$. It is not a $k$-subspace of $R_N$.
However, we can identify it with a suitable $k$-vector space.
Namely, let $R^+_N$ be the subspace of $R$ spanned by
$\bigl\{\tau,\tau^2,\dotsc,\tau^N\bigr\}$ if $p>2$, and by
$\bigl\{1,\tau,\dotsc,\tau^N\bigr\}$ if $p=2$. There is a unique
additive bijection $\phi:R^+_N\rar{\simeq}R^{skew}_N$ given by
$c\tau^j\longmapsto c\tau^j-\tau^{-j}c$ for $1\leq j\leq N$, $c\in
k$, and $c\longmapsto c$ if $p=2$. We have $\dim R^+_N=N$ if $p>2$,
and $\dim R^+_N=N+1$ if $p=2$.

\mbr

Choose $m\in\bN$ such that $a,b,d\in R_m$. Then the map $F$ defined
by \eqref{e:F} takes $R^2_N$ to $R^{skew}_{2N+m}$. Hence
$F':=\phi^{-1}\circ\bigl(F\bigl\lvert_{R_N}\bigr)$ is a map
$R^2_N\rar{}R^+_{2N+m}$. This map is not quite polynomial with
respect to the obvious coordinates on the $k$-vector spaces $R^2_N$
and $R^+_{2N+m}$, because of the equation $\tau^{-1}\cdot
c=c^{1/p}\cdot\tau^{-1}$ for $c\in k$. However, for any $s\in\bN$,
we have a bijection $k^{4N+2}\rar{\simeq}R^2_N$ given by
\[
\bigl(x_{-N},\dotsc,x_N,y_{-N},\dotsc,y_N\bigr)\longmapsto \Bigl(
\sum_i x_i^{p^s}\cdot\tau^i, \sum_j y_j^{p^s}\cdot\tau^j \Bigr),
\]
and the resulting composition $k^{4N+2}\rar{}R^+_{2N+m}$ is given by
a polynomial map with coefficients in $k$ if $s$ is large enough.
Since $4N+2>\dim_k R^+_{2N+m}$ whenever $N\geq m$, and since $k$ is
assumed to be algebraically closed, the standard theorem of the
dimension of fibers of an algebraic map implies that the equation
$F(x,y)=0$ has a nonzero solution $(x,y)\in R^2_N$ for $N\geq m$,
which proves the lemma.
\end{proof}

\subsection{Noncommutative Serre duality}\label{aa:noncomm-Serre}
Let us once again fix a perfect field $k$ of characteristic $p>0$.
Let $G$ be a connected perfect unipotent group over $k$. This time
we do not assume that $G$ is commutative. The Serre dual of $G$ can
be defined, as a functor, in the same way as in
\S\ref{aa:Serre-duality}, by formula \eqref{e:serre-dual}, except
that the right hand side has to be interpreted as the group of
isomorphism classes of central extensions of $G\times_k S$ by the
discrete group scheme $\qzp$. The following result was conjectured
by Drinfeld in \cite{drinfeld-lectures}; the key idea of the proof
is also due to him.

\begin{prop}\label{p:serre-noncomm}
If $G$ is a connected perfect unipotent group over $k$, the
restriction of the functor $G^*$ to $\mathfrak{Perf}_k$ is
represented by an object of $\cpu_k$, which we also denote by $G^*$.
Moreover, the natural homomorphism $(G^{ab})^*\rar{}G^*$, induced by
the quotient map $G\rar{}G^{ab}=G/[G,G]$, identifies $(G^{ab})^*$
with $(G^*)^\circ$.
\end{prop}

Let us first explain why the problem is nontrivial. Naively one
might expect that every central extension of $G$ by $\qzp$ restricts
to a trivial extension of $[G,G]$, so that $G^*=(G^{ab})^*$ as
functors. However, this is not so, as was already observed in
\cite{drinfeld-lectures}; see \cite{masoud} for more details.
Fortunately, as we will see below, this naive expectation only
``fails by a finite amount,'' which allows us to get a handle on
$G^*$.

\mbr

We use the following result of \cite{masoud} in a crucial way:
\begin{thm}[M.~Kamgarpour]\label{t:masoud}
There exists a $($unique$)$ central extension
\begin{equation}\label{e:true-comm}
1 \rar{} \Pi \rar{} [G,G]_{true} \rar{} [G,G] \rar{} 1,
\end{equation}
where $\Pi$ is a finite unipotent $k$-group\footnote{In other words,
a finite \'etale unipotent group scheme over $k$.}, characterized by
the properties that
\begin{itemize}
\item $[G,G]_{true}$ is connected,
\item every central
extension of $G$ by a finite unipotent $k$-group splits after pullback to
$[G,G]_{true}$, and
\item the commutator morphism $G\times_k G\rar{}G$
lifts to $[G,G]_{true}$.
\end{itemize}
Moreover, there exists a central extension
$\widetilde{G}\rar{\pi}G$ of $G$ by a finite unipotent $k$-group
such that \eqref{e:true-comm} is isomorphic to
$\pi^{-1}([G,G])^\circ\rar{}[G,G]$. Finally, the formation of
\eqref{e:true-comm} commutes with base change to algebraic
extensions of $k$.
\end{thm}

The group $[G,G]_{true}$, together with the homomorphism
$[G,G]_{true}\rar{}G$, is called the \emph{true $($\'etale$)$
commutator} \cite{drinfeld-lectures} of $G$, for the obvious reason.

\mbr

To prove Proposition \ref{p:serre-noncomm}, we use induction of
$\dim G$. If $\dim G=1$, then $G$ is commutative and we can apply
Theorem \ref{t:serre-duality}. Assume that $\dim G>1$ and the result
holds for all connected perfect unipotent groups $H$ over $k$ such
that $\dim H<\dim G$.

\mbr

If $S$ is any perfect scheme over $k$ and $A$ is an abstract
(discrete) abelian group, we will write $H^2(G\times_k S,A)$ for the
abelian group of isomorphism classes of central extensions of
$G\times_k S$ by $A$ in the category of group schemes over $S$.

\begin{lem}\label{l:pullback-true}
If $S$ is a perfect scheme over $k$, the natural homomorphism
\[
H^2(G\times_k S,\qzp) \rar{} H^2([G,G]_{true}\times_k S,\qzp)
\]
equals zero.
\end{lem}

\begin{proof}{Proof}
We may assume that $S$ is affine. There exists $n\in\bN$ such that
$g^{p^n}=1$ for all $g\in G$ or $[G,G]_{true}$. Hence, for any $S$
as above, the natural homomorphisms
\[
H^2(G\times_k S,\bZ/p^n\bZ) \rar{} H^2(G\times_k S,\qzp)
\]
and
\[
H^2([G,G]_{true}\times_k S,\bZ/p^n\bZ) \rar{}
H^2([G,G]_{true}\times_k S,\qzp)
\]
are isomorphisms. On the other hand, it is easy to see that since
$S$ is affine,
\[
\limto H^2(G\times_k S',\bZ/p^n\bZ) \rar{\simeq} H^2(G\times_k
S,\bZ/p^n\bZ),
\]
where the inductive limit is taken over all morphisms $S\rar{}S'$,
where $S'$ is a perfect affine $k$-scheme of quasi-finite type
(\S\ref{aa:perfect-unipotent}) over $k$. Hence it suffices to prove
the lemma in the case where $S$ is of quasi-finite type over $k$.

\mbr

Now, by the induction assumption, the functor $[G,G]_{true}^*$ is
representable by an object of $\cpu_k$. Given an element of
$H^2(G\times_k S,\qzp)$, its pullback to $[G,G]_{true}\times_k S$
defines a morphism of $k$-schemes $S\rar{}[G,G]_{true}^*$. By
Theorem \ref{t:masoud}, the induced map
$S(\overline{k})\rar{}[G,G]_{true}^*(\overline{k})$ is identically
$0$, where $\overline{k}$ is an algebraic closure of $k$. Since $S$
is of quasi-finite type over $k$, this implies that
$S\rar{}[G,G]^*_{true}$ is constant.
\end{proof}

We can now complete the proof of Proposition \ref{p:serre-noncomm}.
Consider the sequence \eqref{e:true-comm} defined in Theorem
\ref{t:masoud}. In view of Lemma \ref{l:pullback-true}, we obtain an
exact sequence of functors from $\mathfrak{Perf}_k$ to the category
of abelian groups,
\[
0 \rar{} (G^{ab})^* \rar{} G^* \rar{} \Hom(\Pi,\qzp) \rar{} 0,
\]
where $\Hom(\Pi,\qzp)$ is viewed as a finite unipotent $k$-group in
the natural way\footnote{For instance, $\Pi$ is nothing but a finite
abelian $p$-group equipped with a continuous action of
$\Gal(\kbar/k)$, where $\kbar$ is an algebraic closure of $k$. Then
$\Hom(\Pi,\qzp)$ can be equipped with the contragredient action of
$\Gal(\kbar/k)$.}. Since $(G^{ab})^*$ is representable by a
connected commutative perfect unipotent group over $k$ by Theorem
\ref{t:serre-duality}, both statements of Proposition
\ref{p:serre-noncomm} follow.

\subsection{An auxiliary construction}\label{aa:construction-alt}
In this subsection we describe a construction used in the definition
of an admissible pair for a unipotent group in characteristic $p>0$
(see \S\ref{ss:def-admissible}). The construction is a geometric
counterpart of the following simple observation. If $\Ga$ is a
finite group, $N\subset\Ga$ is a normal subgroup, $\chi:N\rar{}\qzp$
is a homomorphism, which is invariant under the conjugation action
of $\Ga$, and $Z\subset\Ga$ is a subgroup such that $N\subset Z$ and
$[\Ga,Z]\subset N$, then $\chi$ induces a homomorphism
$\Ga/N\rar{}\Hom(Z/N,\qzp)$ given by $\ga\longmapsto\bigl(z\mapsto
\chi(\ga z\ga^{-1}z^{-1})\bigr)$.

\mbr

We fix a perfect field $k$ of characteristic $p>0$, let $U$ be a
(possibly disconnected) perfect unipotent group over $k$ and let
$N\subset U$ be a normal connected subgroup. By Proposition
\ref{p:serre-noncomm}, we can speak about the Serre dual,
$N^*\in\cpu_k$, of $N$, and since $N^*$ is defined by a universal
property (in the category of perfect $k$-schemes), it is clear that
$U$ acts on $N^*$ regularly by $k$-group scheme automorphisms.

\mbr

Let $\nu\in N^*(k)$ be a $U$-invariant element, and let $Z\subset U$
be a connected subgroup such that $N\subset Z$ and $[U,Z]\subset N$.
(Without loss of generality, one can take $Z$ to be the preimage in
$U$ of the neutral connected component of $Z(U/N)$.)

\mbr

\noindent \framebox[1.03\width]{We claim that $\nu$ defines a $k$-group scheme morphism
$\vp_\nu:U/N\rar{} (Z/N)^*$.}

\mbr

\noindent In the proof of this claim we will use the standard correspondence
between central extensions (or bi-extensions) of connected
(quasi-)algebraic groups by $\qzp$ and (bi)multiplicative
$\qzp$-torsors (see \cite{masoud} and Lemma
\ref{l:bi-ext=bi-mult-tors} above). Thus we will also
denote by $\nu$ the multiplicative $\qzp$-torsor on $N$ defined by
$\nu$.

\mbr

 First we will define a morphism of $k$-schemes $U\rar{} Z^*$.
By definition, this is the same as constructing a central extension
of $U\times_k Z$, viewed as a group scheme over $U$, by
$\bQ_p/\bZ_p$. We define a $\bQ_p/\bZ_p$-torsor on $U\times_k Z$ by
$\cE=c^*\nu$, where $c:U\times_k Z\rar{} N$,
$c(u,z)=[u,z]:=uzu^{-1}z^{-1}$. Now we apply the following result.

\begin{lem}\label{l:U-Z}
The restrictions of $\cE$ to $N\times_k Z$ and $U\times_k N$ are
trivial torsors. Moreover, let $\mu_U:U\times_k U\rar{} U$,
$\mu_Z:Z\times_k Z\rar{} Z$ denote the multiplication morphisms, and
let $p_1,p_2:U\times_k U\rar{} U$ and $q_1,q_2:Z\times_k Z\rar{} Z$
be the natural projections. Then
\begin{equation}\label{e:Z}
(\id_U\times\mu_Z)^*\cE\cong(\id_U\times q_1)^*\cE\tens(\id_U\times
q_2)^*\cE
\end{equation}
as $\bQ_p/\bZ_p$-torsors on $U\times_k Z\times_k Z$, and
\begin{equation}\label{e:U}
(\mu_U\times\id_Z)^*\cE\cong(p_1\times\id_Z)^*\cE\tens(p_2\times\id_Z)^*\cE
\end{equation}
as $\bQ_p/\bZ_p$-torsors on $U\times_k U\times_k Z$.
\end{lem}

The proof of Lemma \ref{l:U-Z} is given below. Note that formula
\eqref{e:Z} implies that $\cE$ corresponds to a central extension of
$U\times_k Z$, viewed as a group scheme over $U$, by $\bQ_p/\bZ_p$,
and hence defines a morphism $U\rar{}Z^*$ of $k$-schemes. Formula
\eqref{e:U} implies that, moreover, this morphism is a group
homomorphism. Finally, the first sentence of the lemma means that
this homomorphism factors through a homomorphism
$\vp_\nu:U/N\rar{}(Z/N)^*$, as claimed.

\begin{rem}\label{r:varphi-skew-symmetric}
Note that $Z/N\in\cpuc_k$, because $[U,Z]\subset N$, and hence,
\emph{a fortiori}, $[Z,Z]\subset N$. The restriction of $\vp_\nu$ to
$Z/N$ is a bi-extension of $(Z/N,Z/N)$ by $\qzp$. In fact, this
extension is \emph{skewsymmetric} (Definition
\ref{d:symm-skew-symm}) by the very construction of $\vp_\nu$
(indeed, observe that the restriction of $c$ to the diagonal in
$Z\times_k Z$ is constant).
\end{rem}

\begin{proof}{Proof of Lemma \ref{l:U-Z}}
The following observation will be used several times in the proof
below. Let $\al:U\times_k N\rar{}N$ denote the conjugation action
map: $(u,n)\longmapsto unu^{-1}$. Then $\al^*\nu$ is a $\qzp$-torsor
over $U\times_k N$, and it is clear that if we view $U\times_k N$ as
a group scheme over $U$, then $\al^*\nu$ becomes a
\emph{multiplicative} $\qzp$-torsor, in the sense that
\[
(\id_U\times\mu_N)^*\al^*\nu\cong(\id_U\times
r_1)^*\al^*\nu\tens(\id_U\times r_2)^*\al^*\nu
\]
as $\qzp$-torsors on $U\times_k N\times_k N$, where $\mu_N:N\times_k
N\rar{}N$ is the multiplication morphism and $r_1,r_2:N\times_k
N\rar{}N$ are the two projections. Hence $\al^*\nu$ defines a
morphism of $k$-schemes $U\rar{}N^*$, which is nothing but the orbit
map for the $U$-action on $\nu\in N^*(k)$. By assumption, $\nu$ is
$U$-invariant, whence
\begin{equation}\label{e:inv}
\al^*\nu\cong(\qzp)_U\boxtimes\nu,
\end{equation}
where $(\qzp)_U$ denotes the trivial $\qzp$-torsor on $U$.

\mbr

 Let us prove that, in the notation of Lemma \ref{l:U-Z}, the
torsor $\cE\bigl\lvert_{N\times_k Z}$ is trivial. The following
composition clearly equals $c\bigl\lvert_{N\times_k Z}$:
\[
N\times_k Z \rar{\iota} N\times_k Z\times_k N \xrar{\ \
\id\times(\al\bigl\lvert_{Z\times_k N})\ \ } N\times_k N \xrar{\ \
\mu_N\ \ } N,
\]
where $\iota(n,z)=(n,z,n^{-1})$. Therefore
\begin{eqnarray*}
\cE\bigl\lvert_{N\times_k Z} &\cong& \bigl( c\bigl\lvert_{N\times_k
Z}
\bigr)^*\nu \\
&\cong& \iota^* \bigl[ \id\times(\al\bigl\lvert_{Z\times_k N})
\bigr]^* (\mu_N^*\nu) \\
&\cong& \iota^* \bigl[ \id\times(\al\bigl\lvert_{Z\times_k N})
\bigr]^* (\nu\boxtimes\nu) \\
&\cong& \iota^*\bigl(\nu\boxtimes (\qzp)_U \boxtimes\nu \bigr) \cong
(\qzp)_{N\times_k Z},
\end{eqnarray*}
as claimed, where the isomorphism before the last one uses
\eqref{e:inv}.

\mbr

 The triviality of $\cE\bigl\lvert_{U\times_k N}$ is proved by a
completely analogous argument.

\mbr

 Let us prove \eqref{e:Z}. It is straightforward to verify the
identity
\[
c(u,z_1 z_2) = c(u,z_1)\cdot c(u,z_2) \cdot \al(z_1,c(u,z_2)) \qquad
\forall\, u\in U,\ z_1,z_2\in Z.
\]
It translates into the commutativity of the following diagram:
\[
\xymatrix{
  U\times_k Z\times_k Z \ar[d]_{\id_U\times\mu_Z} \ar[rr]^\be & & N\times_k N \times_k N \ar[d]^{\mu_3} \\
  U\times_k Z \ar[rr]^c & & N
   }
\]
where $\mu_3(n_1,n_2,n_3)=n_1 n_2 n_3$ and
$\be(u,z_1,z_2)=\bigl( c(u,z_1), c(u,z_2), \al(z_1,c(u,z_2))
\bigr)$. Therefore, using \eqref{e:inv}, we find that
\begin{eqnarray*}
(\id_U \times\mu_Z)^*\cE &\cong& \be^*(\nu\boxtimes\nu\boxtimes\nu)
\\
&\cong& (\id_U\times q_1)^*c^*\nu \tens (\id_U\times q_2)^*c^*\nu
\tens (\qzp)_{U\times_k Z\times_k Z}
\\
&\cong& (\id_U\times q_1)^*\cE\tens(\id_U\times q_2)^*\cE,
\end{eqnarray*}
which proves \eqref{e:Z}.

\mbr

 Finally, the proof of \eqref{e:U} is very similar, so we omit
the details. The argument uses the easily verifiable identity
\[
c(u_1 u_2,z) = c(u_1,z)\cdot c(u_2,z) \cdot
c\bigl(zu_1z^{-1},c(u_2,z)^{-1}\bigr)^{-1}
\]
together with the multiplicativity property of $\nu$ (i.e.,
$\mu_3^*\nu\cong\nu\boxtimes\nu\boxtimes\nu$) and the fact that
$\cE\bigl\lvert_{U\times_k N}$ is trivial. This completes the proof
of Lemma \ref{l:U-Z}.
\end{proof}

\subsection{Lifting central extensions}\label{aa:lift-central-exts}
In this subsection we prove a result on lifting central extensions
of connected unipotent groups by $\bQ_p/\bZ_p$ that is essentially
equivalent to Proposition \ref{p:ext-loc-sys} used in the main body
of the text.

\begin{prop}\label{p:lifts}
Let $k$ be $($as usual$)$ a perfect field of characteristic $p>0$,
let $G$ be a connected unipotent group over $k$, let $H\subset G$ be
a connected subgroup such that $[G,G]\subset H$, and consider a
central extension
\begin{equation}\label{e:extension-H}
 0 \rar{} \bQ_p/\bZ_p \rar{} \Ht \rar{\pi} H \rar{} 0.
\end{equation}
Assume that the commutator morphism $\operatorname{com}:G\times
G\rar{} H$ lifts to a morphism $G\times G\rar{}\Ht$. Then
\eqref{e:extension-H} lifts to a central extension of $G$ by $\qzp$.
\end{prop}

\begin{rem}\label{r:perfect-irrelevant}
The reader may observe that in this result we departed from our
tradition of working with perfect unipotent groups. However, the
difference between algebraic and quasi-algebraic groups is
absolutely irrelevant in Proposition \ref{p:lifts} (which we could
have equally well stated for perfect unipotent groups). Indeed, if
$G$ is any algebraic group over $k$, the natural pullback map
$H^2(G,\qzp)\rar{}H^2(G_{per\!f},\qzp)$ is easily seen to be an
isomorphism, where, as in the proof of Proposition
\ref{p:serre-noncomm}, we write $H^2(G,\qzp)$ for the abelian group
of isomorphism classes of central extensions of $G$ by $\qzp$ in the
category of $k$-group schemes.
\end{rem}

\begin{proof}{Proof of Proposition \ref{p:lifts}}
It is clear that we can choose a lift $c:G\times G\rar{}\Ht$ of the
commutator morphism which satisfies $c(1,1)=1$. We assert that $c$
makes the homomorphism $\de:\Ht\rar{} G$ obtained by composing $\pi$
with the inclusion $H\hookrightarrow G$ a \emph{strictly stable
crossed module} in the terminology of \cite{breen}.

\mbr

 Let us briefly recall what this means. Define a morphism
$G\times\Ht\rar{}\Ht$ by $(g,h)\longmapsto {}^g h:=c(g,\de(h))\cdot
h$. This morphism is a regular left action of $G$ on $\Ht$ by
algebraic group automorphisms, satisfying
\[
\de({}^g h) = g\de(h)g^{-1} \quad\text{and}\quad {}^{\de(h)} h' =
hh'h^{-1} \quad\forall\, g\in G,\ h,h'\in \Ht.
\]
(This is a slight abuse of notation since we should not be using
``elements'' of $G$ and $\Ht$, but it is easy to rewrite the
identities above purely in terms of morphisms of schemes.) In
addition, $c$ satisfies $c(g,g)=1$, $c(g,h)c(h,g)=1$, and a few
other conditions (coming from the axioms for a braided monoidal
category), all of which are carefully formulated in \cite{breen} and
in \cite{masoud}.

\mbr

 All the equations for the morphism $c$ mentioned in the
previous paragraph are automatically satisfied in our situation
because $G$ is connected (and therefore geometrically integral),
$\bQ_p/\bZ_p$ is discrete, and $c(1,1)=1$ by assumption. The full
proof is left as a simple exercise for the reader. The details can
also be found in \cite{masoud}.

\mbr

 According to \cite{breen}, $c$ induces the structure of a
\emph{strictly commutative Picard stack} on the quotient stack
$G/\Ht$; see \cite{breen} or \cite[Exp. XVIII, \S1.4]{sga4} for the
definition of a strictly commutative Picard stack. According to
Proposition 1.4.15 in \emph{loc.~cit.}, if $\cA$ and $\cB$ are
sheaves of abelian groups on any site, then strictly commutative
Picard stacks $\cP$ with $\pi_0(\cP)=\cA$ and $\pi_1(\cP)=\cB$ are
classified up to equivalence by the group $\Ext^2(\cA,\cB)$. In our
situation, $\pi_0(G/\Ht)=G/H$ and $\pi_1(G/\Ht)=\bQ_p/\bZ_p$.

\mbr

 We claim that $\Ext^2(G/H,\bQ_p/\bZ_p)=0$. Indeed, $G/H$ is a
connected commutative unipotent group over $k$, so since $k$ is
perfect, $G/H$ has a filtration by connected subgroups with all the
successive subquotients isomorphic to $\bG_a$. By induction on
$\dim(G/H)$, we are reduced to the following lemma, which is proved
in \S\ref{aa:proof-l:vanishing-ext-2} below.

\begin{lem}\label{l:vanishing-ext-2}
If $k$ is any perfect field of characteristic $p>0$, then
\[\Ext^r(\bG_a,\bQ_p/\bZ_p)=0 \qquad\text{for all } r\geq 2,\] where $\bQ_p/\bZ_p$
is viewed as a discrete group scheme over $k$ and the group $\Ext^2$
is computed in the category of sheaves of abelian groups on the fppf
site of $\Spec k$.
\end{lem}

We see that the Picard stack $G/\Ht$ is equivalent to the ``trivial
one'', defined as the product $(G/H)\times
(\bQ_p/\bZ_p-\underline{\operatorname{tors}})$, where $G/H$ is the
discrete (i.e., having no nontrivial morphisms) strictly commutative
Picard stack defined by the commutative algebraic group $G/H$ and
$\bQ_p/\bZ_p-\underline{\operatorname{tors}}$ is the Picard stack of
$\bQ_p/\bZ_p$-torsors. This implies that the $1$-morphism
$H\rar{}\bQ_p/\bZ_p-\underline{\operatorname{tors}}$ of gr-stacks
(cf.~\cite{breen} or \cite{masoud}) obtained from the central
extension \eqref{e:extension-H} extends to a $1$-morphism
$G\rar{}\bQ_p/\bZ_p-\underline{\operatorname{tors}}$, which in turn
implies that $\chi$ lifts to a central extension of $G$ by
$\bQ_p/\bZ_p$.
\end{proof}

\subsection{Proof of Lemma
\ref{l:vanishing-ext-2}}\label{aa:proof-l:vanishing-ext-2} The
argument presented below is based on an idea borrowed from \S{}III.0
of \cite{milne}, and uses a result of L.~Breen \cite{breen-exts}.
Breen also has independently suggested another proof of Lemma
\ref{l:vanishing-ext-2} (in private communication).

\mbr

The \emph{perfect site}, $k_{pf}$, of $\Spec k$ is defined as the
category of perfect $k$-schemes of quasi-finite type (see
\S\ref{aa:perfect-unipotent}) over $k$, with the Grothendieck
topology for which the covering families are the surjective families
of \'etale morphisms.

\mbr

 For brevity, we will introduce the following (non-standard)
notation. Let us write $\sA$ for the category of sheaves of abelian
groups on the fppf site $k_{fppf}$ of $\Spec k$, and let us write
$\sA(p)$ for the category of sheaves of $\bF_p$-vector spaces on the
site $k_{pf}$. We claim that there is a natural (quasi-)isomorphism
\begin{equation}\label{e:RHom-quasi-isom}
R\Hom_{\sA(p)}(\bG_a,\bZ/p\bZ) \rar{\simeq}
R\Hom_{\sA}(\bG_a,\bQ_p/\bZ_p).
\end{equation}
Indeed, following the proof of Lemma III.0.13(a) of \cite{milne},
let us choose an injective resolution $\bQ_p/\bZ_p\rar{}\sI^\bullet$
in the category $\sA$. Injective abelian sheaves are divisible, so
if we let $\sI^j_p$ be the kernel of the multiplication by $p$ map
$\sI^j\rar{}\sI^j$ for every $j\geq 0$, then the complex
$\sI^\bullet_p$ is a resolution of $\bZ/p\bZ$, which restricts to an
\emph{injective} resolution of $\bZ/p\bZ$ in the category $\sA(p)$.
Moreover, since any morphism $\bG_a\rar{}\sI^j$ automatically
factors through $\sI^j_p$ (because $\bG_a$ is annihilated by $p$),
we obtain \eqref{e:RHom-quasi-isom}.

\mbr

 To complete the proof of Lemma \ref{l:vanishing-ext-2} we use
the Artin-Schreier sequence
\begin{equation}\label{e:A-S}
0 \rar{} \bZ/p\bZ \rar{} \bG_a \xrar{\ \ x\longmapsto x^p-x\ \ }
\bG_a \rar{} 0
\end{equation}
In \cite{breen-exts} it is proved that
$\Ext^j_{\sA(p)}(\bG_a,\bG_a)=0$ for all $j\geq 1$. Applying the
functor $\Hom_{\sA(p)}(\bG_a,-)$ to \eqref{e:A-S} and using the
associated long exact sequence of the $\Ext$ groups, we find that
$\Ext^r_{\sA(p)}(\bG_a,\bZ/p\bZ)=0$ for all $r\geq 2$. By
\eqref{e:RHom-quasi-isom}, we are done.

\section*{Appendix B: Proof of Theorem \ref{t:strong-higman}}

\setcounter{section}{2}

\setcounter{subsection}{0}

\setcounter{thm}{0}

In this appendix we present a proof of Theorem \ref{t:strong-higman},  which is due to V.~Drinfeld.

\subsection{Setup}\label{ss:facts}
We fix an easy unipotent group $G$ over $\bF_q$ and an irreducible representation $\rho$ of $G(\bF_q)$ over $\ql$. Let us list a few facts, which were established in the course of proving Theorem \ref{t:dim-reps-easy} in the main body of the text.

\begin{enumerate}[(1)]
\item There exists an admissible pair $(H,\cL)$ for $G$ (where $H$ is a connected subgroup of $G$ defined over $\bF_q$ and $\cL$ is a multiplicative $\ql$-local system on $H$) such that the restriction of $\rho$ to $H(\bF_q)$ contains $t_\cL$ as a direct summand (Theorem \ref{t:adm-pair-compatible}).
 \mbr
\item Let $G'$ denote the normalizer of $(H,\cL)$ in $G$. Then $G'$ is connected and the dimension of $G'/H$ is even (\S\ref{sss:dimensions-step-1}). Also, $G'/H$ is commutative, which results from the connectedness of $G'$ and the definition of an admissible pair.
 \mbr
\item Let $\vp_\cL:(G'/H)_{perf}\rar{}(G'/H)_{perf}^*$ denote the homomorphism induced by as in \S\ref{aa:construction-alt}. Then $\vp_\cL$ is an isomorphism (see \S\ref{sss:dimensions-step-1}).
 \mbr
\item The skewsymmetric bi-additive map
$B_{\cL} : (G'/H)(\bF_q) \times (G'/H)(\bF_q) \rar{} \ql^\times$
induced by $\vp_\cL$ is nondegenerate (see \S\ref{sss:dimensions-step-2}).
 \mbr
\item Let $L\subset (G'/H)(\bF_q)$ be any Lagrangian subgroup with
respect to $B_\cL$, and let $\widetilde{L}$ denote the preimage of
$L$ in $G'(\bF_q)$. Then $t_\cL:H(\bF_q)\rar{}\ql^\times$ can be
extended to a homomorphism $\chi:\widetilde{L}\rar{}\ql^\times$, and
for any such extension, we have
$\rho\cong\Ind_{\widetilde{L}}^{G(\bF_q)}\chi$ (see
\S\ref{sss:dimensions-step-2}).
\end{enumerate}

\subsection{Existence of Lagrangian subgroups} Next we will formulate a result that implies Theorem
\ref{t:strong-higman} in view of the facts we listed above. Let
$A$ be a perfect connected commutative unipotent group over $\bF_q$,
and let $\vp:A\rar{}A^*$ denote an isomorphism that induces a
skewsymmetric bi-extension (Definition \ref{d:symm-skew-symm}) of $(A,A)$
by $\qzp$ (in particular, $\vp^*=-\vp$). Let $B_\vp:A(\bF_q)\times
A(\bF_q)\rar{}\ql^\times$ denote the skewsymmetric bi-additive map
induced by $\vp$; it is nondegenerate by
Proposition \ref{p:Serre-Pontryagin}.

\begin{prop}\label{p:auxiliary}
With the notation above, there is a Lagrangian subgroup
$L\subset A(\bF_q)$ with respect to $B_\vp$ such that
$L=\underline{L}(\bF_q)$ for some connected subgroup
$\underline{L}\subset A$.
\end{prop}

The proof of the proposition will be given below. It is clear that
Theorem \ref{t:strong-higman} follows from the proposition,
because the latter implies that the Lagrangian subgroup $L$
mentioned in \S\ref{ss:facts}(5) can be chosen to have the form
$\underline{L}(\bF_q)$ for a connected subgroup
$\underline{L}\subset G'/H$. Letting $P$ be the preimage of
$\underline{L}$ in $G'$, we have $\widetilde{L}=P(\bF_q)$, and the
proof of the theorem is complete.

\subsection{Strategy of the proof} The proof of Proposition \ref{p:auxiliary} has two steps. First
we will reduce it to the case where $A$ is annihilated by $p$. In
this case, a stronger statement holds:

\begin{lem}\label{l:many-subgroups}
Let $A$ be a connected commutative unipotent group\footnote{We can
take $A$ to be either an algebraic group in the usual sense, or a
perfect algebraic group. The distinction between the two classes of
groups is irrelevant here.} over $\bF_q$ such that $p\cdot A=0$, and
let $L\subset A(\bF_q)$ be \emph{any} subgroup whose order is a
power of $q$. Then there exists a connected subgroup
$\underline{L}\subset A$ such that $L=\underline{L}(\bF_q)$.
\end{lem}

Note that when $p\cdot A=0$, Lemma \ref{l:many-subgroups} implies
Proposition \ref{p:auxiliary}, because in the setting of the
proposition, the order of any Lagrangian subgroup $L\subset
A(\bF_q)$ with respect to $B_\vp$ equals $A(\bF_q)^{1/2}=q^{(\dim
A)/2}$, which is a power of $q$ because $\dim A$ is even by
Proposition \ref{p:existence-lagr}(b). The lemma is proved in
\S\ref{ss:proof-l:many-subgroups} below.

\subsection{Reduction of Proposition \ref{p:auxiliary} to the case
$p\cdot A=0$} Let us suppose that $A$ is not annihilated by $p$, and
let $n\geq 2$ be the smallest integer such that $p^n\cdot A=0$. Let
$A_0=p^{n-1}\cdot A$. Then $A_0$ is a nontrivial connected subgroup
of $A$, and is isotropic with respect to the skewsymmetric
bi-extension induced by $\vp$, because $p^{n-1}\cdot p^{n-1}\geq
p^n$. Let $A_0^\perp$ be the orthogonal complement to $A_0$ in $A$,
i.e.,
\[
A_0^\perp = \vp^{-1} \bigl( \Ker ( A^* \rar{} A_0^* ) \bigr) \subset
A.
\]
Then $A_0\subset A_0^\perp$, and $A_1:=A_0^\perp/A_0$ is also
connected. Moreover, $\vp$ induces a strongly nondegenerate
skewsymmetric bi-extension of
$(A_1,A_1)$ by $\qzp$. Since $\dim A_1<\dim A$, we may assume that
Proposition \ref{p:auxiliary} holds for this bi-extension. Since
$A_0$ is a connected algebraic subgroup of $A$, the reduction to the
case $p\cdot A=0$ is complete.

\subsection{Proof of Lemma
\ref{l:many-subgroups}}\label{ss:proof-l:many-subgroups} Throughout
the proof, $q$ is assumed to be fixed, and $\bG_a$ denotes the
additive group over the field $\bF_q$. Let $W=A(\bF_q)$; it is
canonically an $\bF_p$-vector space. Since $A$ is annihilated by
$p$, it is isomorphic to a direct sum of copies of $\bG_a$, whence
we may assume that $A=\bG_a^n$. Let $L\subset W$ be any subgroup of
order $q^k$, where $0\leq k\leq n$. Let $A'\subset A$ be the direct
sum of the first $k$ copies of $\bG_a$. Then $A'(\bF_q)$ and $L$ are
$\bF_p$-subspaces of $W$ and have the same dimension over $\bF_p$.
Hence there exists an $\bF_p$-linear map $f:W\rar{}W$ such that
$L=f(A'(\bF_q))$.

\mbr

The ring $\End(\bG_a)$ of endomorphisms of $\bG_a$ as an algebraic
group over $\bF_q$ contains all the elements of $\bF_q$ (acting by
dilations), as well as the Frobenius map $x\mapsto x^p$. It follows
from Lemma \ref{l:Galois} below that the natural ring homomorphism
$\End(\bG_a)\rar{}\End_{\bF_p}(\bF_q)$ is surjective. This easily
implies that the natural homomorphism $\End(A)\rar{}\End_{\bF_p}(W)$
is surjective as well. In particular, the linear map $f$ in the
previous paragraph is induced by an endomorphism of $A$, which, by
abuse of notation, we will also denote by $f:A\rar{}A$. Then $f(A')$
is a connected subgroup of $A$. Moreover, $L=f(A'(\bF_q))\subset
f(A')(\bF_q)$, while $\dim f(A')\leq \dim A'=k$, which implies that
$\abs{f(A')(\bF_q)}\leq q^k=\abs{L}$. Thus $L=f(A')(\bF_q)$. \qed

\subsection{An auxiliary result} Let us recall a construction. If $R$ is a ring and $\Ga$ is a group acting
on $R$ by ring automorphisms, the smash product $R\#\Ga$ is defined
to be the ring whose elements are formal sums
$\sum_{\ga\in\Ga}a_\ga\ga$, where $a_\ga\in R$ and all but finitely
many $a_\ga$ are $0$; the addition in $R\#\Ga$ is defined in the
obvious way; and the multiplication is determined by
$(a_1\ga_1)\cdot(a_2\ga_2)=(a_1\cdot\ga_1(a_2))\cdot(\ga_1\ga_2)$.

\begin{lem}\label{l:Galois}
Let $K\subset L$ be a finite Galois extension of fields, and let
$\Ga=\Gal(L/K)$. The natural homomorphism\footnote{Induced by the
action of $L$ on itself by left multiplication, and by the
tautological action of $\Ga$ on $L$; we denote by $\End_K(L)$ the
algebra of endomorphisms of $L$ as a vector space over $K$.}
$L\#\Ga\rar{}\End_K(L)$ is an isomorphism of $K$-algebras.
\end{lem}

\begin{proof}{Proof}
One can easily show that $L\#\Ga$ is a simple ring\footnote{Let
$I\subset L\#\Ga$ be a nonzero two-sided ideal, let
$\sum_{\ga\in\Ga}a_\ga\ga$ be a nonzero element in $I$ such that the
number of nonzero coefficients $a_\ga$ is as small as possible,
etc.}. Furthermore, $L$ is simple as a module over itself, and
hence, \emph{a fortiori}, as a module over $L\#\Ga$. Since
$\End_L(L)=L$, it follows that $\End_{L\#\Ga}(L)=L^\Ga=K$. Using
Wedderburn theory, we conclude that $L\#\Ga\rar{}\End_K(L)$ is an
isomorphism.
\end{proof}



\end{document}